\documentclass[11pt]{amsart}
\usepackage{amsfonts, amsmath, amssymb, amscd, amsthm, graphicx, setspace, enumitem, color}
\usepackage{hyperref}
\newcommand{\wh}[1]{\widehat{{#1}}}

\newtheorem{thm}{Theorem}[section]
\newtheorem{cor}[thm]{Corollary}
\newtheorem{lem}[thm]{Lemma}
\newtheorem{prop}[thm]{Proposition}

\theoremstyle{definition}
\newtheorem{defn}[thm]{Definition}

\newtheorem{const}[thm]{Construction}
\theoremstyle{remark}
\newtheorem{rem}[thm]{Remark}
\newtheorem{ex}[thm]{Example}

\newcommand{\gen}[1]{\left\langle#1\right\rangle}

\newcommand{\prs}[2]{\gen{#1\parallel #2}}

\newcommand{\ga }{\Gamma}

\newcommand{\e }{\varepsilon }
\renewcommand{\kappa }{\varkappa}

\renewcommand{\d }{{\rm d} }
\renewcommand{\le }{\leqslant}

\newcommand{\cH }{\mathcal H}
\newcommand{\cL }{\mathcal L}

\newcommand{\cG }{\mathcal G}

\newcommand{\cP }{\mathcal P}

\newcommand{\cS }{\mathcal S}

\newcommand{\N }{\mathbb N}
\newcommand{\Z }{\mathbb Z}
\newcommand{\R }{\mathbb R}

\newcommand{\fftp}{{\rm FFTP}}
\newcommand{\bcd}{{\rm BCD}}
\newcommand{\Lab }{{\bf Lab}}

\newcommand{\CG}{\mathsf{ConjGeo}}

\newcommand\rell{\mathsf{Rel}} 
\newcommand\geol{\mathsf{Geo}} 
\newcommand\cycgeo{\mathsf{CycGeo}}
\newcommand\geocl{\mathsf{ConjGeo}} 
\newcommand{\wt}[1]{\widetilde{#1}}
\newcommand{\wdl}{\widetilde{L}}

\newcommand{\nsc}{{\rm NSC}}

\def\coloneq{\mathrel{\mathop\mathchar"303A}\mkern-1.2mu=}

\begin{document}

\title[Finite generating sets of relatively hyperbolic groups]{Finite generating sets of relatively hyperbolic groups and applications to geodesic languages}

\address{Vanderbilt University,
Department of Mathematics,
1326 Stevenson Center,
Nashville, TN 37240, USA.}
\author{Yago Antol\'{i}n}
\email[Yago Antol\'{i}n]{yago.anpi@gmail.com}

\address{University of Neuch\^{a}tel,
Institut de Math\'{e}matiques, R. Emile-Argand 11, CH-2000 Neuch\^{a}tel, Switzerland.}
\author{Laura Ciobanu}
\email[Laura Ciobanu]{laura.ciobanu@unine.ch}

\thanks{}

\begin{abstract}
Given a finitely generated relatively hyperbolic group $G$, we  construct a finite generating set $X$ of $G$ such that $(G,X)$ has the `falsification
by fellow traveler property' provided that the parabolic subgroups $\{H_\omega\}_{\omega\in \Omega}$ have this property with respect to the generating sets $\{X\cap H_\omega\}_{\omega\in \Omega}$.  This implies that groups hyperbolic relative to virtually abelian subgroups, which include all limit groups and groups acting freely on $\mathbb{R}^n$-trees, or geometrically finite hyperbolic groups, have generating sets for which the language of geodesics is regular, and the complete growth series and complete geodesic series are rational. As an application of our techniques, 
we prove that if each $H_\omega$ admits a geodesic biautomatic structure over $X\cap H_\omega$, then $G$ has a geodesic biautomatic structure. 

Similarly, we  construct a finite generating set $X$ of $G$ such that $(G,X)$  has the `bounded conjugacy diagrams' property or the `neighbouring shorter conjugate' property if the parabolic subgroups $\{H_\omega\}_{\omega\in \Omega}$ have this property with respect to the generating sets $\{X\cap H_\omega\}_{\omega\in \Omega}$. This implies that a group hyperbolic relative to abelian subgroups has a generating set for which its Cayley graph has bounded conjugacy diagrams, a fact we use to give a cubic time algorithm to solve the conjugacy problem.
Another corollary of our results is that groups hyperbolic relative to virtually abelian subgroups have a regular language of conjugacy geodesics.
\end{abstract}



\maketitle



\section{Introduction}
In this paper all generating sets generate the group as a monoid, and are not necessarily assumed to be closed under taking inverses.

Let (P) be a property of (ordered) generating sets of groups. We say that a group $G$ is (P)-{\it completable} if every finite (ordered)
generating set can be enlarged to a finite (ordered) generating set $X$ of $G$ that has the property (P). 

The results of the paper can be summarized in the following statement.
\begin{thm}\label{thm:summary}
Let $I\subseteq\{1,2,3,4,5\}$. A generating set satisfies {\rm (P$_I$)} if it simultaneously satisfies {\rm (Pi)}, $i\in I$, where 
\begin{itemize}
\item[{\rm (P1)}] is the falsification by fellow traveler property,
\item[{\rm (P2)}] is $\mathsf{ShortLex}$ is a biautomatic structure,
\item[{\rm (P3)}] is admitting a geodesic  biautomatic 
structure,
\item[{\rm (P4)}] is the bounded conjugacy diagrams property, and
\item[{\rm (P5)}] is the neighboring shorter conjugate property.
\end{itemize}
Let $G$ be a finitely generated group, hyperbolic relative to a family of {\rm (P$_I$)}-completable groups.
Then $G$ is {\rm (P$_I$)}-completable.
\end{thm}

Relatively hyperbolic groups vastly generalize hyperbolic groups and have been intensively studied in the 20 years since their initial suggestion by Gromov \cite{Gromov}, and subsequent development by Farb \cite{Farb}, Bowditch \cite{Bowditch}, Osin \cite{Osin06} and others. One line of research has been to show that they inherit important properties of their parabolic subgroups, and numerous results recording this behaviour have been produced (the inheritance of the Rapid Decay property \cite{DrutuSapir}, finiteness of the asymptotic dimension \cite{Osin05}, decidability of the existential theory with parameters \cite{Dahmani09} or other algorithmic properties, to name just a few). While lifting characteristics of the parabolic subgroups naturally relies on and extends hyperbolic group machinery, the proofs often have to surmount substantial technical obstacles. In this paper we prove that several geometric, language theoretic and algorithmic properties are inherited from the parabolic subgroups, and in doing so, we produce appropriate generating sets, which are both essential for the properties we discuss, and require the development of several technical tools.

Our main theorem, Theorem \ref{thm:summary}, collects the five properties that we show can be lifted from parabolic subgroups. Most importantly, Theorem \ref{thm:summary} applies to groups that are hyperbolic with respect to abelian subgroups, since finitely generated abelian groups satisfy (P1) -- (P5) for all finite generating sets.
Some of the prominent examples in this class are limit groups in the sense of Sela (see \cite[Theorem 0.3]{Dahmani}), which coincide with the finitely presented fully residually free groups, as well as 
groups acting freely on $\R^n$-trees (see \cite[Theorem 7.1]{Guirardel}). More generally, it was shown in \cite[Theorem 63]{KMS} that finitely presented $\Lambda$-free groups, where $\Lambda$ is an ordered abelian group, are also hyperbolic relative to non-cyclic abelian subgroups.

Theorem \ref{thm:summary} also applies to groups that are hyperbolic with respect to virtually abelian groups, since finitely generated virtually abelian groups are (P$_{\{1,3,5\}}$)-completable.
Geometrically finite hyperbolic groups are hyperbolic relative to their
parabolic subgroups that are virtually abelian (see \cite[Theorem 6.1]{Rebbechi} or
\cite[Proposition 7.9]{Bowditch}). Note that Theorem \ref{thm:summary} recovers and generalizes the result of Neumann and Shapiro \cite[Theorem 4.3]{NeumannShapiro}, which states that geometrically finite hyperbolic groups have the falsification by fellow traveler property and admit a geodesic biautomatic structure \cite[Theorem 5.6]{NeumannShapiro}.

 We now discuss each of these properties and its applications. Formal definitions will be given in Section \ref{sec:prelim}.
\subsection{The falsification by fellow traveler property (FFTP)}
The `falsification by fellow traveler' property is a property of graphs 
(see Definition \ref{def:fftp}) and
was introduced in \cite{NeumannShapiro} by W. Neumann and M. Shapiro inspired
by ideas of J. Cannon. Informally, a graph has this property if for each non-geodesic path there is a shorter path with the same end-points such that the two paths fellow travel, i.e. there is a global constant that bounds the distance between the two paths when `traveling' along them from one end-point to the other.

Suppose that the  Cayley graph $\ga(G,X)$ of a finitely generated group $G$ with 
respect to a finite generating set $X$ has \fftp.
Then $\ga(G,X)$ is almost convex (\cite[Proposition 1]{Elder02}),  $G$ has a finite presentation with a Dehn function that is at most quadratic (\cite[Proposition 2]{Elder02}), and is of type $F_3$ (\cite[Theorem 1]{Elder02}).
Moreover, the barycentric subdivision of the injective hull of $G$ is a model for the classifying space for proper actions (\cite[Theorem 1.4]{Lang}).

Recall that {\it the complete growth series} $\wh{\cG}(G,X)(z)$ is the formal power series with coefficients in the group ring $\Z G$ given by $$\wh{\cG}(G,X)(x)= \sum_{g\in G} g z^{|g|_X} \, \in \Z G[[z]].$$
There are standard definitions determining when this series is a rational or algebraic function (see \cite{SS}).
The natural map $\Z G \to \Z$ sending each $g$ to $1$ shows that rationality of $\wh{\cG}(G,X)(z)$ implies the rationality of the usual growth series $\cG(G,X)(z)=\sum_{g\in G}  z^{|g|_X} \, \in \Z [[z]]$.
In general, if $\cL$ is a language over $X$, one can define
the associated complete growth series $$\wh{\cG}(\cL)(z)= \sum_{W\in \cL} W z^{\ell_X(W)} \,\in \Z [X^*][[z]].$$
If the language $\cL$ is regular, then the complete growth series is rational, as shown in \cite{SS}.
 
If $(G,X)$ has \fftp, then the language of geodesic words over $X$ is regular (\cite[Proposition 4.1]
{NeumannShapiro}), 
 which implies, for example, that dead-end elements in $G$ have bounded depth (\cite{Warshall}); $\ga(G,X)$ has finitely many cone types (\cite{NS2}), so the complete
growth series, as well as
and the complete geodesic growth series of $G$ with respect to $X$, are rational (see \cite{GN}).

Among all the applications listed above, the one that interests us most is the regularity of the language of geodesics. For a word $W\in X^*$, $\ell(W)$ denotes the length of the word $W$. For $g\in G$, $|g|_X$ denotes the the minimal length of a word $W\in X^*$ representing $g$.
Let $\pi \colon X^*\to G$ denote the natural surjection. 
A word $W$ is {\it geodesic} if $\ell(W)=|\pi(W)|_X$.
The language of geodesics of $(G,X)$, denoted by $\geol(G,X)$, is the set
$$\geol(G,X)\coloneq \{W \in X^* \mid \ell(W) =  |\pi(W)|_X\}.$$

\begin{prop}\cite[Proposition 4.1]{NeumannShapiro}\label{prop:NS}
If $(G,X)$ has the falsification by fellow traveler property, then $\geol(G,X)$ is regular.
\end{prop}
 
The only classes of groups that are known to have the falsification by fellow traveler property with respect to 
any generating set are the hyperbolic and abelian ones. In general,
this property is sensitive to changing the generating set, as was shown in \cite{NeumannShapiro}, where
there is an example, due to Cannon, of a virtually abelian group $G$ and two different generating 
sets $X_1$ and $X_2$, such that $\ga(G,X_1)$ has the falsification by fellow traveler property
 and $\ga(G,X_2)$ does not.
 
Families of groups that have the falsification by fellow traveler property with respect to
some generating set  include virtually abelian groups \cite{NeumannShapiro}, geometrically 
finite hyperbolic groups \cite{NeumannShapiro}, Coxeter groups and groups acting simply transitively on the chambers of locally finite buildings \cite{NoskovCoxeter}, 
groups acting cellularly on locally finite CAT(0) cube complexes
where the action is simply transitive on the vertices  \cite{NoskovCC}, 
Garside groups \cite{Holt} and
Artin groups of large type \cite{HoltRees}.

In Section \ref{sec:gensetfftp} we prove that if a group $G$ has \fftp{} with respect to some generating set, then $G$ is \fftp{}-completable. 
With this fact in mind, it is reasonable to say that a group $G$ has the {\it falsification by fellow traveler property} if there exists a finite generating set $X$ of $G$ such that $\ga(G,X)$ has the falsification by fellow traveler property. 

The paper's first main result, which is an immediate consequence of the more technical Theorem \ref{thm:fftpgenset}, is the following (see Section \ref{sec:RH} for a definition of relative hyperbolicity).

\begin{thm} \label{thm:main_intro}
Let $G$ be a finitely generated group, hyperbolic relative to a collection of subgroups with the falsification by fellow traveler property. 
Then $G$ has the falsification by fellow traveler property.
\end{thm}

The fundamental group of a finite graph of  groups with finite edge stabilizers is hyperbolic relative to its vertex groups by \cite[Definition 2]{Bowditch} and Bass-Serre theory. In particular, if $G=A*_C B$, where $C$ is finite, then $G$ is hyperbolic relative to
$A$ and $B$. Similarly, if $G$ is an HNN extension of $A$ and the associated
subgroup is finite, then $G$ is hyperbolic relative to $A$. We thus obtain the following, which seems to have been previously unknown.

\begin{cor}
Suppose that $G$ is the fundamental group of a finite graph of groups where the vertex groups have the falsification by fellow traveler property and the edge groups are finite. Then $G$ has the falsification by fellow traveler property.
\end{cor}

\subsection{Geodesically biautomatic groups}
Automatic and biautomatic groups feature prominently in geometric group theory, and \cite{ECHLPT} constitutes an excellent reference for this topic.

We recall here that a geodesic biautomatic structure for $G$ is a regular language $\cL$  of geodesic words such that paths that start or end at distance one apart and are labeled by words in $\cL$ fellow travel. We note that by an example of Neumann and Shapiro \cite{NS3}, admitting a geodesic biautomatic structure is a generating set dependent property.

The techniques we developed in order to prove Theorem \ref{thm:main_intro} lead to:
\begin{thm}
Let $G$ be a finitely generated group, hyperbolic relative to a family of subgroups (admitting a geodesic biautomatic structure)-completable.
Then $G$ is (admitting a geodesic biautomatic structure)-completable.
\end{thm}

Rebbechi established in his thesis  \cite{Rebbechi} that
groups hyperbolic relative to biautomatic groups are biautomatic. However, the geometric properties of  his biautomatic structure are not clear.
The fact that the biautomatic structure we produce above is geodesic was used in \cite[Section 6]{KMW} to produce Stallings-like graphs for quasi-convex subgroups in toral relatively hyperbolic groups and compute the intersection of such subgroups.

In terms of proofs, both Rebbechi's and our approach consider what is probably the most natural biautomatic structure for $G$, namely the set of paths of minimal length which, when travelling inside the parabolic subgroups, follow fixed normal forms.
We diverge on how we prove that such a language is regular.
One of the technical parts in Rebbechi's thesis is dealing with a variation 
of the falsification by fellow traveler property in order to obtain that this 
language is regular.  Since we are restricting ourselves to geodesic biautomatic structures, we can use the standard falsification by fellow traveler property to prove that the language is regular. In our case, the difficulties arise when finding the appropriate generating sets.

The extra assumption we have imposed, when compared with Rebbechi's work, gives also a stronger conclusion, namely that the biautomatic structure is geodesic. 
\subsection{Conjugacy diagram properties}
In this paper we also investigate conjugacy diagrams in Cayley graphs of 
relatively hyperbolic groups. We show that  in groups hyperbolic relative to
abelian subgroups conjugacy behaves as in free groups, as explained below.

A {\it cyclic geodesic} word is a word for which 
all its cyclic shifts are geodesic.  A conjugacy diagram, as defined in Section 
\ref{sec:conjhyp}, is a $4$-gon in which two opposite sides correspond to 
conjugate words, and the other two sides to the conjugating element. A 
conjugacy diagram is {\it minimal} if the conjugate words $U$ and $V$ are 
cyclic geodesics, and the conjugator is the shortest possible after cyclically 
shifting $U$ and $V$. Minimal conjugacy diagrams have a particularly nice 
behaviour in hyperbolic groups, as shown in \cite[III.$\Gamma$.Lemma 2.9]{BH}: 
the length of two of the opposite sides must be bounded by a universal 
constant. 

More precisely, if $\ga(G,X)$ is hyperbolic, then the following holds (see also Lemma \ref{lem:conjhyp}):  

($\ast$) there is a constant $K>0$ such that for any pair of cyclic geodesic words $U$ and $V$ over $X$ representing conjugate elements either $\max\{\ell(U),\ell(V)\}\leq K$, or there is a word  $C$ over $X$ such that $\ell(C)\leq K$ and $\pi(CU')=\pi(V'C)$, where $U'$ and $V'$ are cyclic shifts of $U$ and $V$.

We will say that pairs $(G,X)$ with the property ($\ast$) have {\it bounded conjugacy diagrams} (\bcd). 

The bounded conjugacy diagrams property is  useful for proving that the language of  conjugacy  geodesics, i.e. the set of minimal length words in a conjugacy class, is regular. 
Let $\sim_c$ be the equivalence relation defined by conjugation in $G$. For $g\in G$, we denote by $[g]_c$ the equivalence class of $G$. A word $W$ is a {\it  conjugacy geodesic} if $\ell(W) = \min_{g \in [\pi(W)]_c} |g|_X$. The language of conjugacy geodesics, denoted by $\CG (G,X)$,
is the set
$$\CG (G,X) \coloneq  \{W \in X^* \mid \ell(W) = \min_{g \in [\pi(W)]_c} |g|_X\}.$$

The language of conjugacy geodesics was introduced by the second author and Hermiller in \cite{CH}, 
where they proved that its regularity is preserved by graph products and free 
products with finite amalgamation. In \cite{CHHR}, the second author, 
Hermiller, Holt and Rees showed that hyperbolic groups, virtually abelian 
groups, extra-large type Artin groups and homogeneous Garside groups have 
generating sets for which the language of conjugacy geodesics is regular.

In fact, in order to prove the regularity of the language of conjugacy 
geodesics, one can use a weaker property than \bcd{}, which we call in this 
paper the {\it neighboring shorter conjugate} property (\nsc). This property 
appears implicitly in \cite[Cor. 3.8]{CHHR}, and it says that any cyclic 
geodesic word that is not a conjugacy geodesic is conjugate (up to cyclic 
shifts) to a shorter word by a conjugator of a uniformly bounded length (see 
Definition \ref{def:nsc}).

The following proposition is one of  our motivations for studying \bcd{} and \nsc{}.
\begin{prop}\cite[Cor. 3.8]{CHHR}\label{prop:CHHR38}
If $(G,X)$ has \fftp{} and \nsc{}, then $\CG (G,X)$ is regular.
\end{prop}

The paper's second main result, which is an immediate consequence of the more technical Theorem \ref{thm:rh_conjhyp_ext}, is the following.

\begin{thm}\label{thm:rh_conjhyp}
Let $G$ be a finitely generated group, hyperbolic relative to a family of \bcd{}-completable  (resp. \nsc{}-completable)  subgroups. Then $G$ is \bcd{}-completable (resp. \nsc{}-completable). 
\end{thm}

In contrast with \fftp{}, we do not know if having $\bcd{}$ (resp. $\nsc{}$) with respect to some generating set implies $\bcd{}$-completable (resp. $\nsc{}$-completable).

A further reason to study the \bcd{} property is the fact that it leads to an efficient solution of the conjugacy problem.
The bounded conjugacy diagrams property was used by Bridson and Howie in 
\cite{BridsonHowie} as part of an algorithm to solve the conjugacy problem 
for lists in hyperbolic groups. In particular, they use bounded conjugacy 
diagrams to show that the conjugacy problem in hyperbolic groups can 
be solved by an algorithm whose time complexity is quadratic in the lengths of 
the input words.
We discuss in Remark \ref{rem:complexity} how to adapt Bridson 
and Howie's algorithm (\cite{BridsonHowie}) to get a cubic time solution of the conjugacy 
problem for groups hyperbolic relative to abelian subgroups. As far as we know, this is the first result that bounds the complexity of the conjugacy problem for this family of groups.

After the completion of this paper, the preprint \cite{Bumagin14} became available. In \cite{Bumagin14} Inna Bumagina determines the complexity of the conjugacy problem in relatively hyperbolic groups as a function of the complexity of the conjugacy problem in the parabolic subgroups.
\subsection{Application to languages}
We now introduce two more languages.

Given a total order order $\prec$ on the set $X$, we can extend it to a 
total order on the monoid  $X^*$ by setting for $U,V\in W^*$, 
$U\prec_{\mathsf{SL}} V$ if either $\ell(U)<\ell(V)$, or $\ell(U)=\ell(V)$ and 
$U$ comes before $V$ in the lexicographic order generated by $\prec$.
We define the set $$\mathsf{ShortLex}(G,X)\coloneq \{ W\in X^* \mid W\prec_{\mathsf{SL}}V \text{ for all }V\in X^* \text{ such that } \pi(V)=\pi(W)\}$$
to be the set of shortlex representatives of elements of $G$, that is, words that are minimal with respect to the length-lexicographic order.

In \cite{CHHR} the following language is also considered
$$\mathsf{ConjGeoMinSL}(G,X)\coloneq \CG(G,X)\cap \mathsf{ShortLex}(G,X).$$ 

Theorem \ref{thm:summary} leads then to a number of results about languages in groups, results which we collect in the following two corollaries.
\begin{cor}\label{cor:toral_languages}
Let $G$ be hyperbolic relative to a family of abelian subgroups.
Then any finite generating set $Y$ can be completed to a finite generating set $X$ satisfying 
\begin{enumerate}
\item[{\rm (1)}] $(G,X)$ has the falsification by fellow traveler property, 
\item[{\rm (2)}] $(G,X)$ has the bounded conjugacy diagram property,
\item[{\rm (3)}] $\mathsf{ShortLex}(G,X)$ is a biautomatic structure for $G$.
\end{enumerate}

In particular, $\mathsf{ShortLex}(G,X)$, $\geol(G,X)$, $\mathsf{ConjGeo}(G,X)$ and $\mathsf{ConjGeoMinSL}(G,X)$ are regular languages and their associated complete growth series, as well as $\wh{\cG}(G,X)(z)$,
are rational.
\end{cor}

In Section \ref{sec:lang} we prove 
that  any finite generating set of a virtually abelian group can be enlarged to a finite  generating set having the falsification by fellow traveler property, the neighboring shorter conjugate property
and a geodesic biautomatic structure simultaneously. 
Thus by combining these facts with Theorem \ref{thm:summary} we obtain:

\begin{cor}\label{cor:vapar_languages}
Let $G$ be hyperbolic relative to a family of virtually abelian subgroups.
Then any finite generating set $Y$ can be completed to a finite  generating set $X$ satisfying 
\begin{enumerate}
\item[{\rm (1)}] $(G,X)$ has the falsification by fellow traveler property, 
\item[{\rm (2)}] $(G,X)$ has the neighboring shorter conjugate property,
\item[{\rm (3)}] $(G,X)$ admits a geodesic  biautomatic structure.
\end{enumerate}

In particular,  $\geol(G,X)$ and $\mathsf{ConjGeo}(G,X)$  are regular languages  and their associated complete growth series, as well as 
$\wh{\cG}(G,X)(z)$,
are rational.
\end{cor}

\subsection{Construction of generating sets}

In Sections  \ref{sec:const} and \ref{sec:goodgenset} we deal with 
one of the main difficulties of the paper, which is finding the appropriate generating set, i.e. one that possesses the desired metric properties. Given a group $G$, hyperbolic relative to
a collection of subgroups $\{H_\omega\}_{\omega\in \Omega}$, the Generating Set Lemma (Lemma \ref{lem:goodgenset}) provides a generating set
$X$ that relates geodesics in the Cayley graph $\ga(G,X)$ to  quasi-geodesics in the
Cayley graph $\ga(G,X\cup \cH)$, where $\cH=\cup_{\omega\in \Omega} (H_\omega-\{1\}).$

We remark that although our main theorems prove the existence of finite generating sets with various properties and no algorithm has been provided to find these sets, they are in fact computable. By starting with a finite group presentation $\prs{X}{R}$ of a  group $G$ and sets of generators, in terms of $X$, for the finitely many subgroups  $\{H_\omega\}_{\omega\in \Omega}$ of $G$, together with a solution of the word problem for those subgroups on these generators, Dahmani showed in \cite{Dahmani2008} that one can compute an explicit relative presentation for $G$ with an explicit linear relative isoperimetric function. This is enough to produce the many other constants used in our paper (hyperbolicity constant, the constants of the bounded coset penetration property, the constant of Lemma \ref{lem:Xi}, etc). With these data, 
the sets $\widehat{\Phi}$ and $\Phi$ in Theorem \ref{thm:Phi} and the Generating Set  Lemma (Lemma \ref{lem:goodgenset}) can be computed following the steps in their corresponding proofs, and with these sets, one can produce the desired generating sets.

\section{Preliminaries}\label{sec:prelim}
We collect notation, definitions and basic results that will be used in the rest of paper.
\subsection{Graphs, words and fellow traveling}

Throughout this section $\ga$ will denote a graph, labelled and directed, where loops and multiple edges are allowed. We will use `$\d$' to denote the combinatorial graph distance between vertices. 

By $L[0,n]$ we denote the unlabelled graph with vertex set $\{0,1,2,\dots, n\}$ and edges joining $i$ 
to $i+1$ for $i=0,\dots, n-1$.  A {\it path $p$ of length $n$} in $\ga$ is a
combinatorial graph map $p\colon L[0,n]\to \ga$.
 In particular, $p(i)$ is a vertex of $\ga$, 
$p(0)$ will be denoted by $p_-$ and $p(n)$ by  $p_+$. For $i,j\in \N$, $i<j$, we use $[i,j]$ to denote the set $\{i,\dots, j\}$; 
then $L[i,j]$ represents the subgraph of $L[0,n]$ spanned by the vertices $[i,j]$. 
A subpath of $p$ is the composition of the graph inclusion $L[i,j]$ into $L[0,n]$ with
the map $p$. We use the notation $p|_{[i,j]}$ for such a subpath. 
We also adopt the convention that 
$p(m)=p(n)$ for all $m\geq n$. 
Let $\ell(p)$ be the length of a path $p$. A path $p$ in $\ga$ is {\it geodesic} if 
$\ell(p)$ is minimal among the lengths of all paths $q$ in $\ga$ with $q_-=p_-$ 
and $q_+=p_+$. Let $k>0$. A path $p$ is a {\it $k$-local geodesic} 
if every subpath of $p$ of length less or equal to $k$ is geodesic. 

Let  $\lambda \geq 1$ and $c\geq 0$. 
A path $p$ is a $(\lambda, c)$-{\it quasi-geodesic} 
if for any subpath $q$ of $p$ we have 
$$\ell ( q )\leq \lambda \d(q_{-},q_{+})+c.$$

Given a set $X$, we denote by $X^*$ the free monoid generated by $X$. Elements of $X^{*}$ are called {\it words}. Sometimes we write $W$ {\it is a word over $X$}, which means that we view $W$ as an element of $X^*$.
Consider a group $G$ that is generated by $X$ as monoid. There is a natural monoid surjection $\pi\colon X^*\to G$. In order to ease the reading we will make an abuse of notation and identify $W\in X^*$ with $\pi(W)\in G$ throughout the paper. Let $\ell_X (W)$ denote the length 
of $W$ as a word over $X$, and let $|W|_X$ be the minimal length of a word over $X$ representing the same element, in $G$, as $W$. If the alphabet can be easily understood from the context we will write $\ell(W)$ for $\ell_X (W)$.
For $U,W\in X^ *$ we write $U\equiv W$ to denote word equality and $U=_G W$ 
to denote group element equality. Let $\ga(G,X)$ be the {\it Cayley graph}
of $G$ with respect to $X$. We use `$\d_X$' for the Cayley graph distance if we need to emphasize the generating set.
The labelling of edges in the Cayley graph by elements of $X$ can be
extended to paths, so for each path $p$ in $\ga(G,X)$ we denote
by $\Lab(p)\in X^{*}$ the word that we obtain reading the labels of the edges along $p$. Notice that $\ell(p)=\ell_X(\Lab(p))$.
A word $W$ is {\it geodesic} if the path
$p$ in $\ga(G,X)$ with $\Lab(p)=W$ and $p_-=1$ is geodesic. A word $W$ is a {\it cyclic geodesic} if all its cyclic shifts are geodesic.

Let $p,q$ be paths in $\ga$  and $K\geq 0$. 
We say that $p,q$ {\it asynchronously $K$-fellow travel}
if there exist non-decreasing functions $\phi\colon \N\to \N$
and $\psi \colon \N \to \N$ such that
$\d(p(t),q(\phi(t)))\leq K$ and $\d(p(\psi(t)),q(t))\leq K$ for all $t\in \N$. 
We say that $p,q$ {\it synchronously $K$-fellow travel}
if  $\d(p(t),q(t))\leq K$ for all $t\in \N$. 
Let $U,V$ be two words over $X$, and 
$p$, $q$ be the paths in $\ga(G,X)$ with $\Lab(p)\equiv U$, $\Lab(q)\equiv V$ and $p_-=q_-=1$. We say that 
$U,V$  asynchronously $K$-fellow travel (resp. synchronously $K$-fellow travel)  if $p$ and $q$ do.

\begin{defn}\label{def:fftp}
Let $\ga$ be a graph and $K\geq 0$. We say that $\ga$ satisfies the {\it falsification by $K$-fellow traveler property} ($K$-\fftp, for short) if for every non-geodesic path $p$ in $\ga$
there exists a path $q$ in $\ga$ such that
$\ell(q)<\ell(p)$, $p_-=q_-$, $p_+=q_+$ and $p$ and $q$ asynchronously $K$-fellow travel. 

Let $G$ be a group with a finite generating set $X$.
Then $(G,X)$ satisfies  $K$-\fftp{} if $\ga(G,X)$ does. In this case, for any non-geodesic word $W\in X^*$ there exists $U\in X^*$ such that $\ell(U)<\ell(W)$, $U=_G W$, and $U$ and $W$ asynchronously $K$-fellow travel.
\end{defn}

The following is a variation \cite[Lemma 1]{Elder02}.
\begin{lem}\label{lem:geosync}
For every $k,K\geq 0$ there exists $M=M(K,k)$ such that the following holds.
If $p$ and $q$ are two geodesics in $\ga(G,X)$,
$\d(p_-,q_-)\leq k$ and $\d(p_+,q_+)\leq k $,
and $p$ and $q$ asynchronously
$K$-fellow travel, then they synchronously
$M$-fellow travel.
\end{lem}

\begin{rem}\label{rem:syncfftp}
Using the previous Lemma, Elder showed in \cite{Elder02} that $(G,X)$ has the
falsification by fellow traveler property if and only if $(G,X)$ has
the {\it synchronous falsification by fellow traveler property}, i.e. there is a constant $K>0$ such that  for every non-geodesic path $p$ in $\ga(G,X)$
there exists a path $q$ in $\ga$ such that
$\ell(q)<\ell(p)$, $p_-=q_-$, $p_+=q_+$ and $p$ and $q$ synchronously $K$-fellow travel.
\end{rem}

The following fact follows easily from the definitions.
\begin{lem}\label{lem:comptravel}
Suppose that $p$ is the composition of paths $p_1, p_2, \dots, p_n$
and $q$ is the composition of paths $q_1, q_2, \dots, q_n$.
If $p_i$ and $q_i$ asynchronously $K$-fellow travel for all $i$, 
then $p$ and $q$ asynchronously $K$-fellow travel.
\end{lem}

\subsection{Relatively hyperbolic groups}\label{sec:RH}

We follow \cite{Osin06} for notation and definitions of relatively hyperbolic groups. The definition we use is equivalent to what Farb's calls `strong relative hyperbolicity' in \cite{Farb}, and we will also use his approach when proving Theorem \ref{thm:main_intro}. 

\begin{defn}
Let $G$ be a group, $\Omega$ a set,
$\{H _\omega \}_{\omega \in \Omega}$ a collection of  subgroups of $G$, and
$X$ a subset of G. 

The set $X$ is a {\it generating set relative to} $\{H _\omega \}_{\omega \in \Omega}$ if the natural homomorphism from
$$ F = (*_{\omega \in \Omega}H_\omega )* F (X)$$
to $G$ is surjective, where $ F (X)$ is the free group with  basis $X.$ 

Assume that $X$ is a generating set relative to $\{H_\omega \}_{\omega \in \Omega}$ and let $R$ be a subset of $F$ whose normal closure is the kernel of the natural map $F\to G$.
In this event, we say that $G$ has relative presentation
$$\prs{X \cup_{\omega \in \Omega} H_\omega}{R}.$$

If $|X|<\infty$  and $|R| <\infty$, the relative presentation  is said to be finite and
the group $G$ is said to be {\it finitely presented relative to the collection of subgroups}
$\{H_\omega \}_{\omega \in \Omega}.$
Set
$$\cH :=\cup_{\omega \in \Omega} (H_\omega -\{1\}).$$

Given a word $W$ over the alphabet $X\cup \cH$ that represents $1$ in $G$
there exists an expression for $W$ in $F$ of the form
\begin{equation}
\label{eq:defarea}
 W=_F \prod_{i=1}^n f_ir_i^{\e_1} f_i^{-1},
\end{equation}
where $r_i\in R$, $f_i\in F$ and $\e_i=\pm 1$ for $i=1,\dots, n$.
The smallest possible number $n$ in an expression of type \eqref{eq:defarea} is
called the {\it relative area of $W$} and is denoted by $\mathrm{Area}_{rel}(W)$.

A group $G$ is {\it hyperbolic relative to the collection of subgroups $\{H_\omega\}_{\omega \in \Omega}$}  if it is finitely presented relative to $\{H_\omega\}_{\omega \in \Omega}$ and there is a constant $C\geq 0$, called an {\it isoperimetric constant},
such that 
$$\mathrm{Area}_{rel}(W)\leq C \ell_{X\cup \cH}(W)$$
for all words $W$ over $X\cup \cH$ that are the identity in $G$.
In particular, $\ga(G, X\cup \cH)$ is Gromov hyperbolic (see \cite[Theorem 2.53]{Osin06}).
The family of subgroups  $\{H_\omega\}_{\omega \in \Omega}$ is called the collection of {\it peripheral
(or parabolic) subgroups} of $G.$ Note that being relatively hyperbolic is a group property independent of the relative presentation \cite[Theorem 2.34]{Osin06}. 
\end{defn}

We collect now a series of facts that will be used in the rest of the paper.
\begin{lem}\label{lem:factsRH}
Suppose that $G$ is hyperbolic with respect to $\{H_\omega\}_{\omega \in \Omega}$.
\begin{enumerate}
\item[{\rm (i)}] If $G$ is finitely generated (in the ordinary sense), then $H_\omega$ is finitely generated
for each $\omega\in \Omega$ {\rm (\cite[Proposition 2.29]{Osin06})}.
\item[{\rm (ii)}] If $G$ is finitely generated (in the ordinary sense) then $|\Omega|<\infty$ {\rm (\cite[Corollary 2.48]{Osin06})}.
\item[{\rm (iii)}] For $\omega,\mu\in \Omega$,  $\omega\neq \mu$ and $g,h\in G$, the following hold: $|H_{\omega}^{g}\cap H_{\mu}^{h}|<\infty$ and $|H_{\omega}^{g}\cap H_{\omega}|<\infty$ if $g\notin H_{\omega}$. Here $H^g \coloneq g^{-1}Hg$ {\rm (\cite[Proposition 2.36]{Osin06})}.
\end{enumerate}
\end{lem}

For the rest of the section, $G$ will be a finitely generated  group,
$\{H_\omega\}_{\omega \in \Omega}$ a collection of parabolic subgroups of $G$, and $X$ a finite generating set.
As before, $\cH=\cup_{\omega \in \Omega} (H_\omega -\{1\})$. 

\begin{defn}
Let $p$ and $q$ be two paths in $\ga(G, X\cup \cH)$.
\begin{enumerate}
\item An {\it $H_\omega$-component} of ${p}$ is a subpath ${s}$ such that $\Lab(s)\in (H_\omega)^*$ and $s$ is not contained in any other subpath whose label is a word on $H_\omega$. A subpath ${s}$ is a {\it component} if it is an $H_\omega$-component for some $\omega\in \Omega$. For a component $s$ of $p$ and a generating set $Y$ of $G$, the  {\it $Y$-length of $s$} is  $\d_Y(s_-,s_+)$.

\item Two components ${s}$ and ${r}$ (not necessarily in the same path) are {\it connected} if both are $H_\omega$-components for
some $\omega\in \Omega$ and $({s}_-)H_\omega=({r}_-)H_\omega$.  
\item A component ${s}$ of ${p}$ is {\it isolated} if it is not connected to any other component ${r}$ of ${p}$.  
\item The path ${p}$ does not 
 {\it backtrack} if all components are isolated. 
\item The path $p$ does not {\it vertex backtrack} if for any subpath $r$ of 
$p$, $\ell(r)>1$, $\Lab(r)$ does not represent an element of some $H_\omega$. In particular, if a path does not vertex backtrack, it does not backtrack and all components are edges.
\item Let ${p}_1,{p}_2$ be components of ${p}$. We write ${p}_1<{p}_2$ if $p_1$ is traversed before $p_2$ in the path $p$. That is, for ${p}_1\neq {p}_2$ and $t,t'\in \N$, if  ${p}(t')\in {p}_2$ and ${p}(t)\in {p}_1$, then $t<t'$.   

\item We say that ${p}$ and ${q}$
are {\it $k$-similar} if
$$\max\{\d_X ({p}_-, {q}_-), \d_X ({p}_+ , {q}_+)\} \leq  k.$$
\end{enumerate}
\end{defn}

The following two results are key ingredients in many of the proofs in this paper.

\begin{thm}[Bounded Coset Penetration Property]\cite[Theorem 3.23]{Osin06}
\label{thm:BCP}
For any $\lambda \geq 1, c \geq 0, k \geq 0$, there exists a constant
$\e = \e (\lambda, c, k)$ such that for any two $k$-similar paths ${p}$ and ${q}$ in 
$\ga(G,X\cup \cH)$
that are  $(\lambda, c)$-quasi-geodesics and do not backtrack,
the following conditions hold.
\begin{enumerate}
\item[{\rm(1)}] The sets of vertices of ${p}$ and ${q}$ are contained in the closed $\e$-neighborhoods 
(with respect to the metric $\d_X$) of each other.
\item[{\rm(2)}] Suppose that ${s}$ is an $H_\omega$-component of ${p}$ such that $\d_X({s}_{-}, {s}_+ ) > \e$; then
there exists an $H_\omega$-component ${t}$ of ${q}$ which is connected to ${s}$.
\item[{\rm(3)}] Suppose that ${s}$ and ${t}$ are connected $H_\omega$-components of ${p}$ and ${q}$ respectively.
Then ${s}$ and ${t}$ are $\e$-similar.
\end{enumerate}
\end{thm}

\begin{lem}\cite[Lemma 2.7]{OsinPF} 
\label{lem:Xi}
Let $G$ be a finitely generated group that is hyperbolic relative to a collection of subgroups
$\{H_\omega\}_{\omega \in \Omega}$. Then there exists a finite subset $\Xi\subseteq G$ and  constant $L\geq 0$ such that the
following condition holds. Let $q$ be a path in $\ga(G,X\cup \cH)$ with $q_-=q_+$ and let ${p}_1,\dots,{p}_k$ 
be the set of isolated components of ${q}$.
Then the $\Xi$-lengths of ${p}_1,\dots, p_k$ satisfy 
$$\sum_{i=1}^k \d_\Xi (({p}_i)_-,({p}_i)_+)\leq L\ell(q).$$
\end{lem}

\begin{rem}
In \cite{Osin06} all the generating sets are assumed to be symmetric. Also, at the beginning of \cite[\S 3]{Osin06}, some additional technical hypotheses are set for the relative 
generating set, and Theorem \ref{thm:BCP} above is 
proved under these assumptions. However, it is easy to check that the statements we are using hold if we change a finite relative generating set by another.
\end{rem}


\section{Generating sets with FFTP} \label{sec:gensetfftp}

In this section we show that if a group has one finite generating set, say $X$, with FFTP, then it has infinitely many generating sets with this property, because for any positive integer $m$, the generating set consisting of the ball of radius $m$ over $X$ has FFTP. This is shown in Proposition \ref{lem:extendingfftpgenset}.

For convenience in this section we  use the synchronous version of \fftp{} (See Remark \ref{rem:syncfftp}).
\begin{lem}\label{lem:ftqg}
Suppose that $(G,X)$ has $M$-\fftp.
Then for any integer $b\geq 0$ and any $(1,b)$-quasi-geodesic $p$ in $\ga(G,X)$ one can find a geodesic path $q$ in $\ga(G,X)$ with $p_-=q_-$, $p_+=q_+$ such that $p$ and $q$ $(M\cdot b)$-synchronously  fellow travel.
\end{lem}

\begin{proof}
Suppose that $p$ is a $(1,b)$-quasi-geodesic and that $\Lab(p)$ represents the element $g\in G$.
Set $p_0=p$ and define a sequence of paths $p_0,p_1,\dots, p_b$ in $\ga(G,X)$ as follows.
If $p_i$ is geodesic, then $p_{i+1}=p_i$. If $p_i$ is not geodesic,
then there is some path $q_i$ shorter than $p_i$ with the same end points
that synchronously $M$-fellow travels with $p_i$, in which case we set $p_{i+1}=q_i$.

Since $\ell(p)\leq |g|_X+b$, if $p_i$ is not geodesic, then $\ell(p_{i+1})<\ell(p_i)$ and so the path $p_b$ is geodesic.
Now, using the fact that $p_i$ and $p_{i+1}$ synchronously $M$-fellow travel,
we get that $p_0$ and $p_b$ synchronously $(M\cdot b)$-fellow travel. 
\end{proof}

A word $W$ is {\it minimal non-geodesic} if it is not geodesic but all its proper subwords are geodesic.

\begin{prop}\label{lem:extendingfftpgenset}
Suppose $(G,X)$ has \fftp.
Let $m>0$ and $Z=\{g\in G : |g|_X\leq m\}$. 
Then $(G,Z)$ has \fftp. 
In particular, $G$ is \fftp -completable.
\end{prop}
\begin{proof}
Let $K_X$ be the falsification by fellow traveler constant for $(G,X)$.

For each $z\in Z$ choose a geodesic word $W_z$ over $X$ that represents $z$.

\textbf{Claim 1.} Let $U\equiv z_1\cdots z_n$ be a geodesic word over $Z$ and let
$V\equiv W_{z_1}\cdots W_{z_n}$ be the corresponding word over $X$. We claim that $0\leq \ell_X(V)-|V|_X\leq m$.

Clearly $\ell_X(V)\geq |V|_X$. Take $a\geq 0$ and $0\leq b<m$ such that
$|V|_X= a m +b$. Then $|V|_X>(n-1)m$ because otherwise $V$ could be written as the product of less than $n$ words, each of length $\leq m$, and that would contradict the fact that $U$ is a geodesic over $Z$.  Thus $(n-1)m<|V|_X=am+b$ and therefore $n=\ell_Z(U)\leq a +1$. Thus, $|V|_X\leq \ell_X(V)\leq m\ell_Z(U)\leq am+m\leq |V|_X+m$, and the claim is proved.

\textbf{Claim 2.} Suppose that $U\equiv z_1\cdots z_n$ is a minimal non-geodesic word over $Z$. Then $V\equiv W_{z_1}\cdots W_{z_n}$ is
a $(1,5m)$-quasi-geodesic word over $X$.

Take a subword $V'$ of $V$. Then $V'= A W_{z_i}\cdots W_{z_j}B$ for some $1<i\leq j<n$, where $A$ 
is a nonempty suffix of $W_{z_{i-1}}$ and $B$ a nonempty prefix of $W_{z_{j+1}}$. Since $U$ is a minimal non-geodesic, the subword $z_i\cdots z_j$ is geodesic over $Z$.  Let $C\equiv W_{z_i}\cdots W_{z_j}$. Then by Claim 1,
$\ell_X(C)\leq |C|_X +m$. Notice that $|C|_X-2m \leq |ACB|_X\leq |C|_X+2m$, and hence
$$\ell_X(A CB)\leq \ell_X(C) +2m \leq |C|_X +3m\leq |ACB|_X+5m,$$
which proves the second claim.

Finally we show that $(G,Z)$ has the synchronous falsification by fellow traveler property with constant $K_Z=5m \cdot K_X +6m$ by proving that any minimal non-geodesic in $\ga(G,Z)$ synchronously $K_Z$-fellow travels with a geodesic in $\ga(G,Z)$. 

Suppose that $U$ is a minimal non-geodesic word over $Z$, and let $V$ be the corresponding word over $X$. 
Let $p_X$ be the path in $\ga(G,X)$ with $(p_X)_-=1$ and such that $\Lab(p_X)\equiv V$, 
and let $p_Z$ be the path in $\ga(G,Z)$ with $(p_Z)_-=1$, $\Lab(p_Z)\equiv U$. 

By Claim 2, $p_X$ is a $(1,5m)$-quasi-geodesic, and by Lemma \ref{lem:ftqg}, 
there exists a geodesic path $q_X$ in $\ga(G,X)$ with 
$(q_X)_-=(p_X)_-$, $(q_X)_+=(p_X)_+$ such that $p_X$ and $q_X$ synchronously $(5m \cdot K_X)$-fellow travel.

Since every subpath of length less than $m$ in $q_X$ represents an element of $Z$, we can subdivide $q_X$ to obtain a path $q_Z$ in $\ga(G,Z)$ such that $q_Z(t)=q_X(tm)$ 
for $t=0,1,\dots, \lfloor\frac {\ell(q_X)}{m}\rfloor$ 
and $q_Z(t)=(p_Z)_+$ for $t>\lfloor\frac {\ell(q)}{m}\rfloor$. Notice that $q_Z$ is a geodesic by construction. 

\textbf{Claim 3.} For every $r,s,t\in \N$ such that  $p_Z(t)=p_X(s)$, $q_Z(t)=q_X(r)$ we have that $|s-r|\leq 6m$.

Indeed, since $p_Z$ is a minimal non-geodesic, $t\geq \d_Z(1,p_Z(t))\geq t-1$ and thus $$(t-1)m\leq \d_X(1,p_Z(t))=\d_X(1,p_X(s))\leq tm.$$
Since $s=\ell_X(p_X([1,s]))\geq \d_X(1,p_X(s))$, we obtain that $s\geq (t-1)m$. 
The fact that $p_X$ is a $(1,5m)$-quasi-geodesic implies that $s=\ell_X(p_X([0,s]))\leq \d_{X}(1,p_X(s))+5m\leq tm +5m$. In summary,
$$(t-1)m\leq  s \leq  tm +5m .$$
By the construction of $q_Z$ and $r$, it similarly follows that $(t-1)m\leq r\leq tm$. Claim 3 now follows.

To conclude the proof, observe that, since $q_X$ is geodesic, we have that for every $t\in \N$ and $s,r$ such that  $p_Z(t)=p_X(s)$, $q_Z(t)=q_X(r)$
\begin{align}
\label{eq:Kz}
\d_Z(p_Z(t),q_Z(t))&\leq \d_X(p_Z(t),q_Z(t))=   \d_X(p_X(s),q_X(r)) \\
&\leq \d_X(p_X(s),q_X(s))+\d_X(q_X(s),q_X(r))\leq K_X\cdot 5m+ |s-r|.\nonumber
\end{align}
Using Claim 3 and \eqref{eq:Kz} $p_Z$ and $q_Z$ synchronously $(5m \cdot K_X+6m)$-fellow travel.
\end{proof}

\section{Constructing paths  in \texorpdfstring{$\ga(G,X\cup \cH)$}{XUH} from paths in \texorpdfstring{$\ga(G,X)$}{X}}
\label{sec:const}

Throughout this section let $G$ be hyperbolic relative to $\{H_\omega\}_{\omega\in \Omega}$, $\cH =\cup_{\omega \in \Omega} (H_\omega -\{1\})$, and $X$ a finite generating set.

The main idea of the paper is to transform paths in $\ga(G,X)$ into paths in 
$\ga(G,X\cup \cH)$ while taking advantage of the hyperbolicity of the latter. 
In this section we establish a canonical way of performing such transformations.

Most of the technicalities in this section are due to the non-trivial intersections of parabolic subgroups. It is worth mentioning that in the case
of torsion-free groups the arguments and the proofs in this section can be greatly simplified.

\begin{const}\label{const:paths}

The {\it $\{H_\omega\}_{\omega\in \Omega}$-factorization} 
(or simply factorization), of a word $W$ over $X$, 
is an expression for $W$ of the form 
$$W\equiv A_0U_1A_1\cdots U_nA_n,$$ 
where $A_0,\dots ,A_n$ are words over $X-\cH$ and 
$U_1,\dots,U_n$ non-empty words over some $X\cap H_\omega$
 such that if $A_i$ is empty, then $U_i x$ cannot be a word 
 over any $X\cap H_\omega$, where $x$ is the first letter of 
 $U_{i+1}$. The number of words $U_i$ is the length of the 
 factorization. We note that the factorization is uniquely determined by $W$.

We say that a word $W$ over $X$ has {\it no parabolic shortenings} if
each $U_i$ in the factorization of $W$,  $U_i$ a word over $X\cap H_\omega$, is geodesic in $(H_\omega, X\cap H_\omega)$.

For a word $W$ with no parabolic shortenings, we define
$$\wh{W}\equiv A_0h_1A_1\cdots h_nA_n$$
to be a word over $X\cup \cH$, where $U_i=_Gh_i\in \cH$. Notice that since there are no parabolic shortenings, each $U_i$ is geodesic and non-empty, and hence $h_i\neq 1$. We say that the resulting word $\wh{W}$ is {\it derived} from $W$.

Similarly, if $p$ is a path in $\ga(G,X)$ and $\Lab(p)$ has no parabolic shortenings, we denote by $\wh{p}$
the path in $\ga(G,X\cup \cH)$ with $p_-=\wh{p}_-$ and $\Lab(\wh{p})\equiv\wh{\Lab(p)}$. This gives a well-defined map $$\wh{-}\colon \{\text{paths in }\ga(G,X)\}\to \{\text{paths in }\ga(G,X\cup \cH)\},$$
$$p\mapsto \wh{p}.$$

\end{const}

In the next section we will show that for a $2$-local geodesic word $W$ that has no parabolic shortenings, we only need to check a finite list of forbidden words to conclude that $\wh{W}$ is a quasi-geodesic with some fixed parameters. In order to prove this, a first step is Lemma \ref{lem:2geod}, where we get sufficient conditions for $\wh{W}$ to be a 2-local geodesic.

For $t>0$ and any finite generating set $X$ of $G$ we set 
$\Theta_X(t)=\{h\in \cH : |h|_X\leq t \}$. 
We will use the notation $\Theta(t)$ instead of $\Theta_X(t)$ 
when the generating set is clear from the context.

\begin{lem}\label{lem:nocancellation}
Let $G$ be hyperbolic relative to $\{H_\omega\}_{\omega\in \Omega}$ and finitely generated by a set $X$. Let $\cH =\cup_{\omega \in \Omega} (H_\omega -\{1\})$. There exists $m=m(X)>1$ such that for every $\omega,\mu\in \Omega$, $\mu\neq \omega$ 
$$(H_\omega -\Theta(m))\cap H_\mu=\emptyset,$$
and for all $\omega \in \Omega$, $f\in H_\omega-\Theta(m)$ and $y\in (X\cup \cH)-H_\omega$
$$|fy|_{X\cup \cH}=|yf|_{X\cup \cH}=2.$$
\end{lem}
\begin{proof}
By Lemma \ref{lem:factsRH}~(iii), $H_\omega\cap H_\mu$ is finite for all $\omega,\mu\in \Omega$. By Lemma \ref{lem:factsRH}~(ii), $\Omega$ is finite and hence
$I=\cup_{\mu \neq \omega} (H_\mu \cap H_\omega)$ is finite
and any $m>0$ such that $I\subseteq \Theta(m)$ satisfies the first claim of the lemma.

In order to prove the second claim, suppose further that $|g|_X>m$ implies $|g|_{\Xi}>3L$, for all $g\in \gen{\Xi}$, where  $\Xi$ and $L$ are the set and the constant of Lemma \ref{lem:Xi}.

Let $f\in H_\omega-\Theta(m)$ and $y\in (X\cup \cH)-H_\omega$. Since $y\notin H_\omega$ we get $fy\neq_G 1$ and $yf \neq_G 1$. So we only need to consider the case $fy=_G h_1\in X\cup \cH$ and the case $yf=_G h_2\in X\cup \cH$.

Let $q_1$ (resp. $q_2$) be the cycle whose label is $\Lab(q_1)\equiv fyh_1^{-1}$ (resp. $\Lab(q_2)\equiv yfh_2^{-1}$).
Here $f$ labels an isolated component of $q_1$ (resp. $q_2$), since if $f$ were connected to some other component, this should be an $H_\omega$-component  because $f\in H_\omega-\Theta(m)$. In this case, we would get that $y\in H_\omega$, which contradicts the hypothesis.

Then by  Lemma \ref{lem:Xi} if $fy=_G h_1$ (resp. $yf=_G h_2$) $|f|_{\Xi}\leq L \ell(q_1)= 3L$ (resp. $|f|_{\Xi}\leq L \ell(q_2)= 3L$), which contradicts $|f|_X>m$.
\end{proof}

We denote by $\cH_I$ the union of all intersections of pairs of parabolic subgroups, that is,  
\begin{align}
\label{eq:Hi}
\cH_I=(\cup_{\mu\neq \omega} (H_\omega\cap H_\mu))-\{1\}.
\end{align}
Notice that by Lemma \ref{lem:factsRH}, if $G$ is finitely generated, then
the set $\cH_I$ is finite.

\begin{lem}\label{lem:h1h2} Suppose that $\cH_I\subseteq X$. 
Let $W$ be a word over $X$ with no parabolic shortenings. 
Then
$\wh{W}$ does not contain a subword of the form $f_1f_2$, where $f_1,f_2$ in $H_\omega$.
\end{lem}

\begin{proof}
Suppose that the factorization of $W$ is $W\equiv A_0U_1A_1\cdots U_nA_n$.
Let $U_i=_G h_i\in \cH$ and $\wh{W}\equiv A_0h_1A_1\cdots h_nA_n$.

If $\wh{W}$ contains $f_1f_2$ as a subword, then $f_1,f_2$ must belong to some $H_\omega$, $f_1=h_i$, $f_2=h_{i+1}$ and $A_{i+1}$ must be empty. We will show that this leads to a contradiction.

Suppose that $h_i$ and $h_{i+1}$ are elements in $H_\omega$, 
$U_i$ is a word in $X\cap H_\mu$ and $U_{i+1}$ is a word in 
$X\cap H_\nu$ with $\mu \neq \nu$.
Then $h_i\in H_\omega\cap H_\mu$ and $h_{i+1}\in H_\omega\cap H_\nu$.
We claim that $U_i$ is a word in $X\cap H_\omega$. The only case we need to
check is $\mu \neq \omega$. Since $U_i$ is a geodesic and represents an
element of $H_\mu\cap H_\omega\subseteq \cH_I\subseteq X$, we have that $U_i\equiv h_i$, and $U_i$ is a word in $X\cap H_\omega$, as claimed.
A similar argument shows that $U_{i+1}$ is a word in  $X\cap H_\omega$.
Then $U_i U_{i+1}$ is  a word in $X\cap H_\omega$ contradicting the maximality of $U_i$.
\end{proof}

\begin{lem}\label{lem:2geod}
Let $Y$ and $X$ be finite generating sets of $G$. Suppose  that  $$Y\cup \{h\in \cH : 0<|h|_Y\leq 2m\}\subseteq X\subseteq Y\cup \cH,$$ where $m(Y)>1$ is the constant of Lemma \ref{lem:nocancellation}. 
Let $W$ be a 2-local geodesic word over $X$ with no parabolic shortenings.
Then $\wh{W,}$ the word over $X\cup \cH$ derived from $W$, is a 2-local geodesic.
\end{lem}
\begin{proof}
Fix a generating set $X$ satisfying the hypothesis. It is immediate to see that $X-\cH=Y-\cH$.
Also, by Lemma \ref{lem:nocancellation}, $\cH_I\subseteq X$.

Let $W$ be a $2$-local geodesic word over $X$ with no parabolic shortenings.

We argue by induction on the length of the factorization of $W$.
If the length of the factorization is zero, then $W$ is a word over $X-\cH$ and $\wh{W}\equiv W$. 
If $\wh{W}$ is not a 2-local geodesic,  it contains a subword
$xy$, $x,y\in X-\cH=Y-\cH$ for which $|xy|_Y= 2$ and $|xy|_{X\cup \cH}\leq 1$. Then $xy\in  X\cup \cH$.
If $xy\in \cH$, since $|xy|_Y= 2$, $xy\in X$.
This contradicts that $W$ is a 2-local geodesic over $X$.

Consider now the case that the factorization is $W\equiv A_0U_1A_1\cdots U_n A_n$ with $n>0$, and $\wh{W}\equiv A_0h_1A_1\cdots h_n A_n$.
The words $A_i$ are words over $X-\cH$ and by the previous discussion, all of them are $2$-local geodesics
over $X\cup \cH$.

So if $\wh{W}$ is not a $2$-local geodesic, either some $A_i$ is empty, and $|h_ih_{i+1}|_{X\cup \cH}\leq 1$,
or  $A_i$ is non-empty and either $|yh_{i+1}|_{X\cup \cH}\leq 1$, where $y$ is the last letter of $A_i$, or
$|h_i x|_{X\cup \cH}\leq 1$, where $x$ is the first letter of $A_i$.
In each case we are going to derive a contradiction.

In the first case, if $A_i$ is empty, Lemma \ref{lem:h1h2} implies that $h_i,h_{i+1}$ do not
belong to the same parabolic subgroup. If $|h_ih_{i+1}|_{X\cup \cH}\leq 1$, Lemma \ref{lem:nocancellation}
implies that $|h_i|_Y\leq m$ and $|h_{i+1}|_Y\leq m$ and in particular, $h_i, h_{i+1}\in X$. Since $h_i,h_{i+1}\in X$, and $U_i, U_{i+1}$ are geodesic, $U_i\equiv h_i, U_{i+1}\equiv h_{i+1}$
and $h_ih_{i+1}$ is a subword of $W$. Since $W$ is a $2$-local geodesic,
$|h_ih_{i+1}|_X=2$. If $|h_ih_{i+1}|_{X\cup \cH}\leq 1$ then $h_ih_{i+1}\in  \cH$.
Notice that since $|h_ih_{i+1}|_Y\leq 2m$, if $h_ih_{i+1}\in \cH$, then $h_ih_{i+1}\in X$ contradicting $|h_ih_{i+1}|_X=2$. 

Similarly, if $A_i$ is non-empty and $|yh_{i+1}|_{X\cup \cH}\leq 1$, where $y$ is the last letter of $A_i$, then by Lemma \ref{lem:nocancellation} $|h_{i+1}|_Y\leq m$, so $h_{i+1}\in X$ and since $U_{i+1}$ is geodesic, $U_{i+1}\equiv h_{i+1}$ and $yh_{i+1}$ is a subword of $W$. Since $W$ is a $2$-local geodesic, $|yh_{i+1}|_X=2$. Thus if $|yh_{i+1}|_{X\cup \cH}\leq 1$, then  $yh_{i+1} \in \cH$, but since $|yh_{i+1}|_Y\leq m+1$, we obtain $yh_{i+1}\in X$, which is a contradiction.

The last case is analogous.
\end{proof}

In section \ref{sec:bicom} we will show that if for each $\omega\in \Omega$ there is a language $\cL_\omega\subseteq \geol(H_\omega, X\cap H_\omega)$ satisfying a certain fellow traveller property, then we can extend this property to a language of words over $X$. A first step in this direction is to consider the following language.

\begin{defn}\label{defn:rel}
For each $\omega\in \Omega$ let $\cL_\omega\subseteq  (X\cap H_\omega)^*$. We define the set  $\rell(X, \{\cL_\omega\}_{\omega\in \Omega})$ to be the subset of words $W$ in $X^*$ such that in the factorization $A_0U_1A_1U_2\cdots U_nA_n$ of $W$ all
$U_i\in \cup \cL_\omega$, $i=1,\dots, n$.
\end{defn}

 Let $K$ be a group, $Z$ a generating set, and $\cL\subseteq Z^*$. We define the following property for $(K,Z,\cL)$:
\begin{enumerate}
\item[{\rm (L1)}] $ \cL \subseteq \geol(K, Z)$ and $\cL$ contains at least one representative for each element of $K$. 
\end{enumerate}
We usually require  that each $(H_\omega, X\cap H_\omega, \cL_\omega)$  satisfies (L1).  Notice that in this case, if $W\in \rell(X, \{\cL_\omega\}_{\omega\in \Omega})$, then $W$ has no parabolic shortening and hence $\wh{W}$ is defined.

\begin{lem}\label{lem:rel}
Suppose $\cH_I\subseteq X$, and suppose that each $(H_\omega, X\cap H_\omega, \cL_\omega)$ satisfies (L1).
\begin{enumerate}
\item[{\rm(i)}] If each $\cL_\omega$ is prefix-closed, then so is  $\rell(X, \{\cL_\omega\}_{\omega\in \Omega})$.
\item[{\rm (ii)}] If each $\cL_\omega$  is regular, then  so is  $\rell(X, \{\cL_\omega\}_{\omega\in \Omega})$.
\end{enumerate}
\end{lem}

\begin{proof}

(i). Suppose that $W\in \rell(X, \{\cL_\omega\}_{\omega\in \Omega})$ has factorization $A_0U_1A_1U_2\cdots U_nA_n$. Then
a prefix $W'$ of $W$ has factorization $A_0U_1A_1U_2\cdots U_i'A_i'$, where $i\leq n$, $A_i'$ is a prefix of $A_i$ and $U_i'\equiv U_i$ if $A_i$ is non-empty or $U_i'$ is a non-empty prefix of $U_i$.  Since $\cup \cL_\omega$ is prefix-closed, it follows that  $W'\in \rell(X, \{\cL_\omega\}_{\omega\in \Omega})$.

(ii). Recall that regular languages are closed under
concatenation, union, Kleene star and complement.
By definition $W \notin \rell(X, \{\cL_\omega\}_{\omega\in \Omega})$ if and only if in the factorization $W\equiv  A_0U_1A_1U_2\cdots U_nA_n$ there is some $i$ such that
$U_i\notin \cup \cL_\omega$. Since the  languages $\cL_\omega$
are geodesic and $X_I\subseteq X$, if a geodesic word $V$ represents an element of $H_\omega\cap H_\mu$, $\omega \neq \mu$, then
$V\equiv x$, where $x\in X$.
Since  $\cL_\omega$ contains at least one representative
for  $x\in H_\omega$, and $x$ is the unique geodesic word
representing $x$, we have that $x\in \cL_\omega$. Thus if a word $W$ belongs to $(X \cap H_{\omega})^* \cap (X \cap H_{\mu})^*$, then either $W \in \cL_\omega \cap \cL_\mu$ or $W \notin \cL_\omega \cup \cL_\mu$.

Let $\cL_\omega^c =(X\cap H_\omega)^*-\cL_\omega$, $\cP_{\mu,\omega}=[(X\cap H_\mu)^*-(X\cap H_\mu \cap H_\omega)^*]\cup (X-\cH)$ and $\cS_\omega =X-H_\omega$. If $\cL_\omega$ is regular, so is $\cL_\omega^c$. Regularity of $\cP_{\mu,\omega}$ and $\cS_\omega$ follows from closure properties of regular languages. The set of words that are not in $\rell(X, \{\cL_\omega\}_{\omega\in \Omega})$ is exactly the set of words that, in their factorization, they contain a $U_i$ in some $\cL_\omega^c$.  Suppose that the factorization of $W$ is  $A_0U_1A_1U_2\cdots U_nA_n$. If $U_i\in (X\cap H_\omega)^*$, then
$A_0U_1A_1\cdots A_{i}$ is either empty or belongs to $\cup_{\mu\neq \omega} X^*\cP_{\mu,\omega}$ and
$A_{i+i}U_{i+1}\cdots A_n$ is either empty or belongs to $\cS_\omega X^*$.

Therefore, $X^*-\rell(X, \{\cL_\omega\}_{\omega\in \Omega})$ is the set
 
$$\bigcup_{\omega\in \Omega} \bigcup_{\mu\neq \omega}\left( X^*\cP_{\mu,\omega}\cL_\omega^{c} \cS_\omega X^*\bigcup X^*\cP_{\mu,\omega}\cL_\omega^c \bigcup  \cL_\omega^{c} \cS_\omega X^* \bigcup\cL_\omega^c \right).$$
It follows that if each $\cL_\omega$ is regular, $\rell(X, \{\cL_\omega\}_{\omega\in \Omega})$ is regular.
\end{proof}

At a later point we will need to be able to substitute a subword $U_i$ of $W$ in the factorization of $W$ by another word $U_i'$ in $\cup \cL_\omega$ representing the same element. Suppose that $W'$ is the word obtained by this substitution. We will need that $\wh{W}\equiv \wh{W'}$, but this is not always the case.
\begin{ex}
Suppose that the factorization of $W$ is $ U_1U_2$, where $U_1\in (X\cap H_\mu)^*, U_2 \in (X\cap H_\nu)^*$.
It might happen that $U_2=_GU_2'\in \cL_\nu$, but the first letter of $U_2'$ is in $X\cap H_\mu \cap H_\nu$, so
write $U_2'\equiv x V_2$. Then the factorization of $W'\equiv U_1 U'_2$ is $(U_1x)V_2$ and $\wh{W}\not\equiv \wh{W'}$.
\end{ex}
To avoid this problem, we will use geodesic words for which the subwords in $H_\omega \cap X$ are maximal, prioritizing maximality for the left-most subwords, as made precise below.


\begin{defn}\label{def:special}
A word $W$ is {\it $\{\cL_\omega\}_{\omega\in \Omega}$-special}, or simply {\it special}, if the family of languages is clear from the context, if the following hold:
\begin{enumerate}
\item[(1)] $W$ is in $\geol(G,X)\cap \rell(G,\{\cL_\omega\}_{\omega\in \Omega})$, and 
\item[(2i)] either the factorization of $W$ has  length zero (i.e. $W$ is a word in $X-\cH$),
 \item[(2ii)] or the factorization of $W$ has  $A_0U_1A_1U_2\cdots U_nA_n$ of length $n>0$ such that $A_1U_2\cdots U_nA_n$ is special and
$\ell(U_1)\geq \ell(U_1')$ for all $U_1'$ in a factorization $A_0'U_1'A_1'\cdots U_m'A_m'$ of a  word $W'$ in $\geol(G,X)\cap \rell(G,\{\cL_\omega\}_{\omega\in \Omega})$ satisfying $W=_GW'$.

\end{enumerate}
\end{defn}

\begin{lem}\label{lem:specialrep}
Suppose that the collection $\{\cL_\omega\}_{\omega\in \Omega}$ satisfy {\rm (L1)}.
Then for every $g\in G$ there is a special word  representing $g$.
\end{lem}
\begin{proof}
Given $g\in G$, we define $$f(g)=\min\{ n \mid W\in\geol(G,X),\, W=_Gg,\,\text{ the factorization of $W$ has length } n\}.$$
Suppose that $f(g)=0$. Then there is a geodesic word $W$ over $X-\cH$ representing $g$, and hence $W$ is special.

Let $g\in G$, assume that $f(g)=n>0$ and that the Lemma holds for all $h\in G$, $f(h)<n$.

Suppose $W\equiv A_0U_1A_1\cdots U_nA_n$ is the factorization of a geodesic word representing $g$ such that $\ell(U_1)$ is maximal, i.e. if $A_0'U_1'A_1'\cdots U_n'A_n'$ is another factorization of a geodesic word representing $g$, then $\ell(U_1)\geq \ell(U_1')$.

Let $h=_G A_1\cdots U_nA_n$. By construction $f(h)<n$. Then there exists a special word $B$ such that $B=_Gh$.
By (L1) there is $C\in \cup \cL_\omega$ such that $C=_G U_1$.
We claim that $A_0 C B$ is special.
First we need to show that $A_0 C B$ is in 
$\geol(G,X)\cap \rell(G,\{\cL_\omega\}_{\omega\in \Omega})$.
By construction, $A_0CB=_G W=_Gg$, $\ell(C)\leq \ell(U_1)$ and $\ell(B)\leq \ell(A_1\cdots U_nA_n)$. Therefore, $\ell(A_0CB)\leq \ell(W)$, and hence $A_0CB$ is geodesic.
It is easy to show that the factorization of $A_0CB$ is $A_0C$ followed by 
the factorization of $B'$ since by maximality of $\ell(U_1)=\ell(C)$, $Cx$ can not be a word 
over some $X\cap H_\omega$, where $x$ is the first letter of $B$.
Thus since $C\in \cup \cL_\omega$ and $B$ is special, 
$A_0CB\in \rell(X,\{\cL_\omega\}_{\omega})$. 
The maximality condition of $C$ follows by construction.\end{proof}

Now when we substitute a substring $S \in \cL_\omega$ inside a special word by another substring $S' \in \cL_\omega$, where $S=_G S'$, we obtain again a special word.
More precisely:

\begin{lem}\label{lem:specialchange}
Let $W\equiv ACB$ be a special word. Suppose that $\wh{W}\equiv \wh{A}h\wh{B}$, where $\wh{A}$ and $\wh{B}$ are derived 
from $A$ and $B$, and $h=_G C$.
Then $ AC'B$ is special for any $C'\in \cup \cL_\omega$ such that $C'=_G h$.
\end{lem}

\begin{proof}

We proceed by induction on the length of the factorization of $W$. If the length is zero, there is nothing to prove.

So assume that $W$ has factorization $A_0U_1\cdots U_nA_n$ with $n>0$ and the result holds for special words with factorization of smaller length.

Since $W$ is  special,  $W\equiv A_0 U_1 E$, $U_1=_Gf\in \cH$,  and $E$ is special.
Then $\wh{W}\equiv A_0f\wh{E}$.

We have two cases: either $\wh{A}h\equiv A_0f$ or $h$ is a letter of $\wh{E}$.

(i). In the first case $A\equiv A_0$, $B\equiv E$, $h=f$ and $C=_G U_1$.
Since both $U_1$ and $C$ are geodesic, $\ell(U_1)=\ell(C)$.
Take any $C'\in \cup \cL_\omega,$ $C'=_G h$. 
We claim that $A_0C'E$ is special. Since $W$ is special and  $\ell(U_1)=\ell(C)=\ell(C')$,
we conclude that  $C'$ cannot  be contained in a longer subword of $A_0C'E$ over some $X\cap H_\omega$. Hence the factorization of $A_0C'E$ is $A_0C'$ followed by the factorization of $E$. As $E$ is special, and $\ell(U_1)=\ell(C')$, we conclude that $A_0C'E$ is special.

(ii). In the second case $B$ is a proper suffix of $E$.
Then $\wh{E}\equiv \wh{D}h\wh{B}$, where $\wh{D}$ might be the empty word.
Let $D$ be a subword of $E$ such that $\wh{D}$ is derived from $D$. 
Since $E$ is special, by the induction hypothesis, $DC'B$ is special for any  $C'\in \cL_\omega$ such that $C'=_G C=_G h$. Therefore $A_0U_1DC'B$ is special, by definition.
\end{proof}

\section{Constructing finite generating sets for relatively hyperbolic groups}\label{sec:goodgenset}

Throughout this section let $G$ be hyperbolic relative to $\{H_\omega\}_{\omega\in \Omega}$, $\cH =\cup_{\omega \in \Omega} (H_\omega -\{1\})$, and $Y$ a finite generating set of $G$.

The objective here is to prove the Generating Set Lemma (Lemma \ref{lem:goodgenset}), a key result of the paper since it provides a finite generating set $X$ (depending on $Y$) for $G$ which makes it possible to relate geodesics in $\ga(G,X)$ to quasi-geodesics in $\ga(G,X\cup \cH)$.
The main ingredient of the Generating Set Lemma is Theorem \ref{thm:Phi} below. We recall in Theorem \ref{lem:local2global} the `local to global' quasi-geodesic feature of hyperbolic spaces.

\begin{thm}\cite[Section 3, Theorem 1.4]{CDP}\label{lem:local2global}
Suppose that $\ga$ is a $\delta$-hyperbolic space. For all $\lambda'\geq 1$ and $c'\geq 0$ there exist $k>0,$ $\lambda\geq 1$ and $c\geq 0$ (depending on $\delta, \lambda'$ and $c'$) such that every $k$-local $(\lambda',c')$-quasi-geodesic path is a $(\lambda,c)$-quasi-geodesic.
\end{thm}

\begin{thm} \label{thm:Phi}
Let $G$ be hyperbolic relative to $\{H_\omega\}_{\omega\in \Omega}$, $\cH =\cup_{\omega \in \Omega} (H_\omega -\{1\})$, and $Y$ a finite generating set of $G$.

There exists a finite set $\wh{\Phi}$ of non-geodesic words over $Y\cup \cH$ and constants $\lambda\geq 1$ and $c\geq 0$ such that if $W$ is a 2-local geodesic word over $Y\cup \cH$ not containing any element of $\wh{\Phi}$ as a subword, then $W$ is a $(\lambda,c)$-quasi-geodesic without vertex backtracking. 
\end{thm}

\begin{proof}
Suppose that $\ga(G,Y\cup \cH)$ is $\delta$-hyperbolic.
Take $\lambda'=4$ and $c'=0$ and let $k,\lambda, c$ be the constants provided by Theorem \ref{lem:local2global}. 
Then every $k$-local $(4,0)$-quasi-geodesic is a $(\lambda,c)$-quasi-geodesic. 
Without loss of generality, we can enlarge $k$ to further assume that $\lambda+c\leq k$.

Let $\Delta$ be the set of closed paths in $\ga(G,Y\cup \cH)$ 
of length at most $2k$ in which all the components are isolated or have $Y$-length less than or equal to $m$, where $m$ is the constant of Lemma \ref{lem:nocancellation}.
Let $L>0$ and $\Xi$ be the constant and the finite subset of $G$ 
provided by Lemma \ref{lem:Xi}. Then, by Lemma \ref{lem:Xi}, 
if $p\in \Delta$, and $p_1,\dots, p_n$ are the isolated components of $p$,
$$\sum \d_\Xi (({p}_i)_-,({p}_i)_+)\leq L\ell(p)\leq L2k.$$ 
In particular, since $\Xi$ is finite, the set of labels of paths in  $\Delta$ is a finite set.

We take $\wh{\Phi}$ to be the set of labels $\Lab(q)$, where $q$ is a subpath of some $p\in \Delta$, $\ell(q)>\ell(p)/2$. Thus $\wh{\Phi}$ is a finite set of non-geodesic words over $Y\cup \cH$.

{\bf Claim 1:} If $p$ is a 2-local geodesic path in $\ga(G,Y\cup \cH)$ of length at most $k$ that vertex backtracks, then $\Lab(p)$ contains an element of $\wh{\Phi}$ as a subword.

Notice first that a $2$-local geodesic of length $2$ cannot vertex backtrack.
Take a 2-local geodesic path $p$, $2<\ell(p)\leq k$. We can assume without loss of generality that $p$ vertex backtracks, but no proper subpath of $p$ vertex backtracks. 
Let $r$ be the edge from $p_-$ to $p_+$ and suppose that $\Lab(r)\in H_\omega$.
Since $p$ is $2$-local geodesic, all the components of $p$ are single edges. If two components of $p$ were connected, a proper subpath of $p$ would backtrack and hence vertex backtrack. Therefore all components of $p$ are isolated.

If a component $p_1$ of $p$ were connected to $r$, the points $p_-,p_+,(p_1)_-$ and $(p_1)_+$ would lie in the same $H_\omega$-coset.
Since $\ell(p)>2$, either the subpath from $p_-$ to $(p_1)_-$ or the subpath from $(p_1)_+$ to $p_+$ would have length greater than 1 and have end points in the same $H_\omega$-coset, which contradicts our minimality assumption. 
Thus all the components of the path $pr^{-1}$ are isolated. Since $\ell(r)<\ell(p)\leq k$, $pr^{-1}$ is in $\Delta$, and $\Lab(p)\in \wh{\Phi}$.

This completes the proof of Claim 1.

{\bf Claim 2:} If $p$ is a 2-local geodesic path in $\ga(G,Y\cup \cH)$ that does not label a $k$-local $(4,0)$-quasi-geodesic,  then $\Lab(p)$ contains an element of $\wh{\Phi}$ as a subword.

We can assume without loss of generality that $\ell(p)\leq k$, and by Claim 1, also assume that $p$ does not backtrack.
Let $q$ be a geodesic path 
with $q_-=p_-$ and $q_+=p_+$. Since $p$ is not a 
$(4,0)$-quasi-geodesic, we can further assume that $\ell(p)>4\ell(q)$.
Also, $p$ does not backtrack, so all the components of $p$ are isolated. 
As $q$ is geodesic, all the components of $q$ are isolated, and $p$ and $q$ being $2$-local geodesics implies that all the components in $p$ and $q$ are edges. 

If all the components on the cycle $pq^{-1}$ are isolated, then   $pq^{-1}\in \Delta$ and $\Lab(p)\in \wh{\Phi}$. So without loss of generality we  assume that at least a component of $p$ is connected to a component of $q$.

We now choose a set of components, i.e.~edges,  
$p_1<\dots< p_n$, $n\geq 1$,  of $p$ such that each $p_i$ is connected to a component $q_i$ of $q$, $q_1<q_2<\dots <q_n$ and no component of the subpath of $p$ from $(p_i)_+$ to $(p_{i+1})_-$ is connected to a component of the subpath of $q$ from $(q_i)_+$ to $(q_{i+1})_-$ for $i=0,\dots, n$, where we understand that $(p_0)_+=(q_0)_+=q_-=p_-$ and $(p_{n+1})_-=(q_{n+1})_-=q_+=p_+$. See Figure \ref{fig:t2leqt1}.

\begin{figure}[ht]
\begin{center}
\includegraphics[scale=0.4]{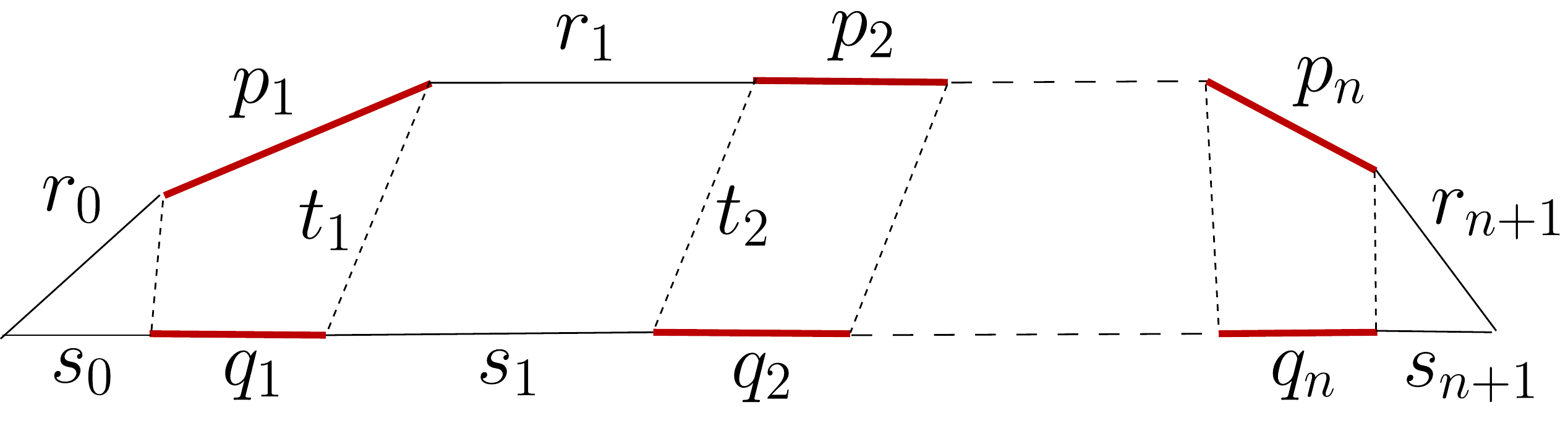} 
\end{center}
\caption{Claim 2.}\label{fig:t2leqt1}
\end{figure}

For $i=0,\dots, n$, we let $r_i$ denote the subpath of $p$ from $(p_i)_+$ to $(p_{i+1})_-$ and $s_i$ denote the subpath of $q$ from $(q_i)_+$ to $(q_{i+1})_-$. 
We remark that by hypothesis $n\neq 0$.
In general,
$$\ell(p)=n+\sum \ell(r_i)>4 \ell(q)= 4n+4\sum \ell(s_i).$$
If $\ell(r_i)\leq \ell(s_i)+2$ for $i=0,\dots, n$, the inequality 
$$4n +4\sum \ell(s_i) < n+\sum \ell(r_i) \leq 3n+\sum \ell(s_i)$$
gives a contradiction.

Therefore, there is an $i$ such that $\ell(r_i)>\ell(s_i)+2$. 
Let $t_i, t_{i+1}$ be geodesic paths from $(r_i)_-$ to $(s_i)_-$ and $(r_i)_+$ to $(s_i)_+$, respectively. We can assume  that $\ell(t_j)=1$ and  $\Lab(t_j)\in H_{\omega_j}$ for $j=i,i+1$. We can view $t_i$ and $t_{i+1}$ as components of the paths $t_i$ and $t_{i+1}$, respectively. If $\d_Y((t_i)_-,(t_i)_+)>m$ and $t_i$ is connected to any component of
$r_i$ then the non vertex backtracking condition implies that this component must be the first edge of $r_i$ and, by Lemma \ref{lem:nocancellation}, this component will have label in $H_{\omega_i}$.  Lemma \ref{lem:nocancellation} also implies that the label of $p_i$ is also in $H_{\omega_i}$,  which contradicts the fact that $p$ doesn't vertex backtrack.
So if $\d_Y((t_i)_-,(t_i)_+)>m$, $t_i$ is not connected to a component of $r_i$. Similarly, if $\d_Y((t_i)_-,(t_i)_+)>m$, $t_i$ is not connected to a component of $s_i$.
If  $t_i$ is connected to $t_{i+1}$, again the non vertex backtracking condition implies that $\ell(r_i)\leq 1$, contradicting $\ell(r_i)>2+\ell(s_i)$.
Therefore $t_i$ is either isolated in the closed path $o=r_it_{i+1}s_i^{-1}t_i^{-1}$ or $\d_Y((t_i)_-,(t_i)_+)\leq m$. The same is true for $t_{i+1}$.
By construction, no component of $s_i$ is connected to a component of $r_i$, so it follows that the closed path $o=r_it_{i+1}s_i^{-1}t_i^{-1}$ has all components isolated or of $Y$-length at most $m$ and the length of $o$ is at most $2k$. Therefore $o\in \Delta$, and since $\ell(r_i)>\ell(s_i)+2$, it follows that $\Lab(r_i)\in \wh{\Phi}$. 

This completes the proof of Claim 2.

Hence, by Claim 2, any 2-local geodesic path $p$ in $\ga(G,Y\cup \cH)$ such that $\Lab(p)$ does not contain any subword in $\wh{\Phi}$ is a $(\lambda,c)$-quasi-geodesic.

Moreover, we remark that such $p$ does not vertex backtrack. Suppose $p$ vertex backtracked. Claim 1 implies that there is a subpath $p_1$ of $p$ with $\ell(p_1)>k$ and
$\d ((p_1)_-,(p_1)_+))\leq 1$.  Since $k$ was chosen such that $\lambda+c \leq k$, Claim 2 implies $\d ((p_1)_-,(p_1)_+))> 1$, which leads to a contradiction.
\end{proof}

\begin{lem}[Generating Set Lemma]\label{lem:goodgenset}
Let $G$ be a finitely generated group, hyperbolic with respect to a family of subgroups $\{H_\omega\}_{\omega\in \Omega}$, and let $Y$ be a finite  generating set.

There exist $\lambda  \geq 1$, $c\geq 0$ and a finite subset 
$\cH'$ of $\cH=\cup_{\omega\in \Omega} (H_\omega - \{1\})$ such that
for every  finite generating set $X$ of $G$ 
with $$Y\cup \cH'\subseteq X\subseteq Y\cup \cH,$$ there is a finite subset $\Phi$ of non-geodesic words over $X$ satisfying: if a word $W \in X^*$ has no parabolic shortenings and does not contain subwords in $\Phi$, then the word $\wh{W} \in (X\cup \cH)^*$ 
is a 2-local geodesic $(\lambda, c)$-quasi-geodesic without vertex backtracking.

In particular, for every $\omega\in \Omega$ and $h\in H_\omega$, $|h|_{X}=|h|_{X\cap H_\omega}$.
\end{lem}

\begin{proof}
Let $\lambda, c,$ and $\wh{\Phi}$ be the constants and the set given by Theorem \ref{thm:Phi} applied to $(G,Y,\{H_\omega\}_{\omega\in \Omega})$, and let $m=m(G,Y,\{H_\omega\}_{\omega\in \Omega})$ be the constant provided by Lemma \ref{lem:nocancellation}.

For each $\wh{U}\in \wh{\Phi}$ let $\wh{V}(\wh{U})$ be a geodesic word over $Y\cup \cH$ such that 
$\wh{U}=_G \wh{V}(\wh{U})$.

Let $$\cH'=\{h\in \cH : h\text{ letter in some } \wh{V}(\wh{U}),\wh{U}\in \wh{\Phi}\}\cup \{h\in \cH : 0< |h|_Y\leq 2m\}.$$ 
Let $X$ be a finite  subset of $G$  satisfying $$Y\cup \cH' \subseteq X\subseteq Y\cup \cH.$$
Clearly $\gen{X}=G$ and $Y\cup \cH=X\cup \cH$.

Let $\Phi$ be the set of words $U$ over $X$ that are either non-geodesics of length 2, or words with no parabolic shortenings such that
the derivation $\wh{U}$ of $U$ lies in $\wh{\Phi}$.
 Since $\wh{\Phi}$ is finite and $X$ is finite, it follows that $\Phi$
is finite. 

We first show that the words in $\Phi$ are not geodesic over $X$.
Pick $U\in \Phi$; we only need to consider the case when the  derivation $\wh{U}$  from $U$  is a word in $\wh{\Phi}$.
By our choice of $X$, $\wh{V}(\wh{U})$ is a word over $X$ and is geodesic over 
$X\cup \cH=Y\cup \cH$. Hence $\wh{V}(\wh{U})$ must be geodesic when viewed as a word over $X$. We denote by $V(U)$ the word $\wh{V}(\wh{U})$ viewed as a word over $X$.  We have that $\ell_X(U)\geq \ell_{X\cup \cH}(\wh{U})>\ell_{X\cup \cH}(\wh{V})=\ell_{X}(V)$ and then $U$ is non-geodesic.

Let $W$ be a word with no parabolic shortenings that does not contain any subword of $\Phi$, and $\wh{W}$ the word derived from $W$.
Since $W$ is  $2$-local geodesic, by Lemma \ref{lem:2geod}, $\wh{W}$ is a 2-local geodesic. By construction of $\Phi$, $\wh{W}$ does not contain subwords of $\wh{\Phi}$. Hence,
by Theorem \ref{thm:Phi}, $\wh{W}$ is a $(\lambda,c)$-quasi-geodesic without vertex backtracking.

Finally, assume moreover that $W$ is a geodesic word over $X$, $W=_G h\in H_\omega$. Since $\wh{W}$  does not vertex backtrack, $\wh{W}$ must have length 1. By the Construction \ref{const:paths} $W$ is a word over $X\cap H_\omega$, which means that for $h\in H_\omega$, $|h|_X=|h|_{X\cap H_\omega}$ for all $\omega\in \Omega$. This completes the proof of the lemma.
\end{proof}

\begin{cor}
Let $G$ be a finitely generated group, hyperbolic relative to a family $\{H_\omega\}_{\omega\in \Omega}$. There exists a finite generating set $X$ and constants $k\geq 0$, $\lambda_1 \geq 1$ and $c_1\geq 0$ such that the following hold:
\begin{enumerate}
\item[{\rm (a)}] The inclusion $\ga(H_\omega, X\cap H_\omega)\to \ga(G,X)$ is  an isometric embedding.
\item[{\rm (b)}] If $W\in X^*$ is a $k$-local geodesic word with no parabolic shortenings, then $W$ is a $(\lambda_1,c_1)$-quasi-geodesic.
\end{enumerate}
\end{cor}
\begin{proof}
Take $\lambda, c$ and $X$ as in Lemma \ref{lem:goodgenset}, and $\Phi \subset X^*$ the associated finite set of non-geodesic words. Claim
(a) follows from the last claim of the Lemma.

To see (b), take $k=\max\{\ell(U)\mid U\in \Phi\}$ and consider any word $W_1$ over $X$ that is a $k$-local geodesic with no parabolic shortenings. Let $W_2$ be a geodesic word over $X$ representing the same element as $W_1$.
By Lemma \ref{lem:goodgenset}, $\wh{W}_1$ labels a $(\lambda,c)$-quasi-geodesic over $X\cup \cH$ without vertex backtracking, so
\begin{equation}\label{eq:corquasigeodesic}
\ell(\wh{W}_1)\leq \lambda |\wh{W}_1|_{X\cup \cH} +c \leq \lambda \ell(\wh{W}_2)  +c \leq \lambda \ell(W_2)+c.\end{equation}

By Lemma \ref{lem:goodgenset},  $\wh{W}_2$ also labels a $(\lambda,c)$-quasi-geodesic over $X\cup \cH$. Let $\wh{p}_1$ and $\wh{p}_2$ be $0$-similar paths  in $\ga(G, X\cup \cH)$ labelled by $\wh{W}_1$ and $\wh{W}_2$, respectively. Let $\e=\e(\lambda,c,0)$ be the constant of the Bounded Coset Penetration property (Theorem \ref{thm:BCP}), and $\wh{r}_1,\dots \wh{r}_n$ be the components of $\wh{p}_1$ of $X$-length greater than $\e$. By the BCP property, there are components $\wh{s}_1,\dots, \wh{s}_n$  of $\wh{p}_2$, such that $\wh{s}_i$ is connected to $\wh{r}_i$ and
$\d_X((\wh{r}_i)_-,\wh{r}_i)_+)\leq \d_X((\wh{s}_i)_-,(\wh{s}_i)_+)+2\e$.

Then we have that
$$\sum_{i=1}^n\d_X((\wh{r}_i)_-,\wh{r}_i)_+)\leq \sum_{i=1}^n \big(\d_X((\wh{s}_i)_-,(\wh{s}_i)_+)+2\e \big)\leq \ell(W_2)+2n\e\leq (2\e+1)\ell(W_2).$$

We can assume that the edges in $\wh{p_1}$ that do not belong to $\wh{r}_1,\dots, \wh{r}_n$ have $X$-length less or equal to $\e$.
By \eqref{eq:corquasigeodesic}, there are at most $\lambda \ell(W_2)+c$ such edges.
Finally 
$$\ell(W_1)=\sum_{e \text{ edge in } \wh{p}_1} \d_X(e_-,e_+)\leq \e(\lambda \ell(W_2)+c)+(2\e+1)\ell(W_2) = (\e\lambda + 2\e+1)\ell(W_2)+\e c .$$
Since $W_1$ is arbitrary and $\ell(W_2)=|W_1|_X$, we get the desired result with $\lambda_1=\e\lambda + 2\e+1$ and $c_1=c\e$.

\end{proof}


\section{Relatively hyperbolic groups with geodesic bicombings}
\label{sec:bicom}

In this section we suppose that 
 $G$ is hyperbolic relative to $\{H_\omega\}_{\omega \in \Omega}$,
$X$ is a finite generating set of $G$ and for each parabolic subgroup $H_\omega$ we have a preferred set of geodesic words $\cL_\omega$ over $X\cap H_\omega$
that fellow travel. The aim of this section is to extend the
fellow traveler property to words over $X$. 

More precisely, let $K$ be a group, $Z$ a finite generating set of $K$ and
$\cL$ a set of words over $Z$. We define the following two properties for $(K,Z,\cL)$, and we recall the property (L1).

\begin{enumerate}
\item[{\rm (L1)}] $ \cL \subseteq \geol(K,Z)$ (i.e.  $\cL$ is a set of geodesics over $Z$) and $\cL$ contains at least one representative for each element of $K$. 

\item[{\rm (L$\forall$)}] There is $M>0$ so that for every $W\in \cL$, $g,h\in K$ and all $U\in \cL$ with $U=_G gWh$, $p,q$ asynchronously $M(|g|_{Z}+|h|_{Z})$-fellow travel, where $\Lab(p)\equiv U$, $\Lab(q)\equiv W$, $p_-=1,$ $q_-=g$.

\item[{\rm (L$\exists$)}] There is $M>0$ so that for every $W\in \cL$, $g,h\in K$ there exists $U\in \cL$ with $U=_G gWh$ such that  $p,q$ asynchronously $M(|g|_{Z}+|h|_{Z})$-fellow travel, where $\Lab(p)\equiv U$, $\Lab(q)\equiv W$, $p_-=1,$ $q_-=g$.
\end{enumerate}

It would be equivalent to state {\rm (L$\exists$)} and {\rm (L$\forall$)} for $g,h\in Z^*$ rather than in $K$, as the following Lemma shows. We have omitted its proof since it is exactly the same as that of \cite[Lemma 4.5]{NeumannShapiro}. There the case $\cL=\geol(K,Z)$, where $(K,Z)$ has $M$-\fftp, was considered.
\begin{lem}
\label{lem:esimilarfftp}
Let $K$ be a group  finitely generated by $Z$. 
Suppose $\cL$ is a set of geodesic words over $Z$
for which there is $M>0$ such that for all 
$x,y\in Z$, $W\in \cL$, for all (resp. there exist) $U,V\in \cL$
with $U=_Gx^{-1}W$, $V=_G Wy^{-1}$,
the path $p$ asynchronously $M$-fellow travels with 
$q_U$ and $q_V$, where
$\Lab(p)\equiv W$, $\Lab(q_U)\equiv U$,
$\Lab(q_V)\equiv V$, $p_-=1=(q_V)_-$, and $(q_U)_-=x$.
Then $\cL$ satisfies {\rm (L$\forall$)} (resp. {\rm (L$\exists$)}).
\end{lem}

\begin{ex}
In view of Lemma \ref{lem:esimilarfftp},  if 
$(K, Z)$ has \fftp{} then $(K,Z,\geol(K,Z))$ satisfies (L1) and (L$\exists$). 

Similarly, if  $\cL$ is a geodesic biautomatic structure for $(K,Z)$, then $(K,Z,\cL)$ satisfies (L1) and  (L$\forall$).
\end{ex}

\begin{rem}\label{rem:Lsync} In view of Lemma \ref{lem:geosync}, if $(K,Z,\cL)$ satisfies (L1), $(K,Z,\cL)$ satisfies (L$\forall$) if and only if it satisfies a synchronous version of (L$\forall$). The same holds for (L$\exists$).
\end{rem}

 The main result of the section is the following. We refer the reader to Definition \ref{defn:rel} for the definition of $\rell(X,\{\cL_\omega\}_{\omega\in \Omega})$.
\begin{prop} \label{prop:geobicom}
Let $G$ be a finitely generated group, hyperbolic with respect to a family of subgroups $\{H_\omega\}_{\omega\in \Omega}$.
Let $Y$ be a finite generating set of $G$ and $\cH=\cup_{\omega\in \Omega} (H_\omega-\{1\})$. 

There exists a finite subset $\cH'$ of $\cH$ such that, for every finite  generating set $X$ of $G$ satisfying
\begin{enumerate}
\item[{\rm (i)}] $Y\cup \cH'\subseteq X\subseteq Y\cup \cH$ and
\item[{\rm (ii)}] for all $\omega\in \Omega$, there is a language $\cL_\omega$  such that $(H_\omega, X\cap H_
\omega, \cL_\omega)$ satisfies {\rm (L1)} and {\rm (L$\forall$)} (resp. {\rm (L1)} and {\rm (L$\exists$)}),
\end{enumerate}
the triple $(G,X, \cL)$ satisfies {\rm (L1)} and {\rm (L$\forall$)} (resp. {\rm (L1)} and {\rm (L$\exists$)}), where $\cL=\geol(G,X)\cap \rell(X,\{\cL_\omega\}_{\omega\in \Omega})$.
\end{prop}
 
\begin{rem}
A set $\cL$ such that $(K,Z,\cL)$ satisfies (L1) and (L$\forall$), and which contains exactly one representative for each element of the group, is called a {\it geodesic bicombing} in \cite{Short}. It is easy to adapt the proof of  Proposition \ref{prop:geobicom} in a way that if we further assume that each $\cL_\omega$ is a geodesic bicombing
then $\geol(G,X)\cap \rell(X,\{\cL_\omega\}_{\omega\in \Omega})$
contains a geodesic bicombing.
\end{rem}

\begin{lem}\label{lem:boundingt1t2}
Let $\lambda\geq 1$, $c\geq 0$ and $k\geq 0$.
Let $\e=\e(\lambda,c,k)$ be the constant of the Bounded Coset Penetration property (Theorem \ref{thm:BCP}).
There exists $K_1=K_1(\e,\lambda,c)$ such that the following hold.

Let $p$ and $q$ be two $k$-similar $(\lambda,c)$-quasi-geodesics. Then
for any subpath $p_1$ of $p$, $\ell(p_1)>K_1$, and any subpath $q_1$ of $q$ 
such that $p_1$ and $q_1$ are $\e$-similar, there is a vertex $u$ in $p_1$, 
$u\notin\{ (p_1)_-, (p_1)_+\}$,
and a vertex $v$ in $q_1$, such that $\d_X(u,v)\leq \e$.
\end{lem}
\begin{proof}
Let $K_1$ be a constant satisfying
$$
K_1/\lambda -c >  \lambda(2\e+2+c)+2\e.
$$

Let $p$, $q$, $p_1$ and $q_1$ be as in the hypothesis of the lemma.

Suppose that the lemma does not hold, i.e. for all
$u$ in $p_1$, 
$u\notin\{(p_1)_-, (p_1)_+\}$,
and all vertices $v$ in $q_1$, we have $\d_X(u,v)>\e$.

\begin{figure}[ht]
\begin{center}
\includegraphics[scale=0.6]{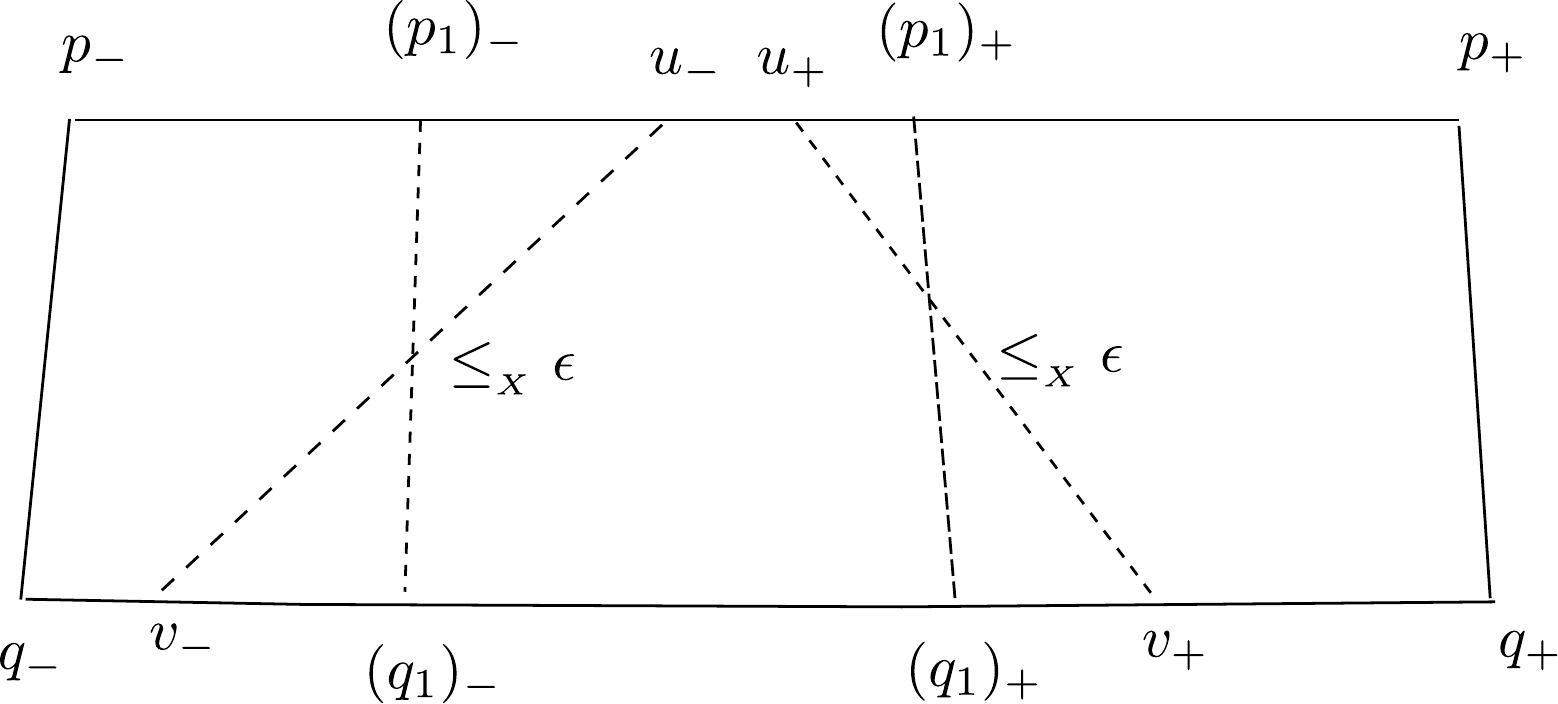} 
\end{center}
\caption{}\label{fig:59}
\end{figure}

By Theorem \ref{thm:BCP}, for each vertex $u$ in $p_1$ there is a vertex $v$ in $q$ such that $\d_X(u,v)\leq \e$.
As shown in Figure \ref{fig:59}, let $u_-\neq (p_1)_+$ be the closest vertex in $p_1$ to $(p_1)_+$ for which there exists $v_-$ in the subpath from $q_-$ to $(q_1)_-$ such that $\d_X(u_-,v_-)\leq \e$. Notice that we are allowing $u_-=(p_1)_-$.
Let $u_+$ be the closest vertex in $p_1$ to $u_-$, between $u_-$ and $(p_1)_+$. In this case there is a vertex $v_+$ in the subpath from $(q_1)_+$ to $q_+$ such that $\d_X(v_+,u_+)\leq \e$.

Thus $\d_{X\cup \cH}(v_-,v_+)\leq 1+2\e$.

Since $p$ is a $(\lambda,c)$-quasi-geodesic 
$$d_{X\cup \cH}((p_1)_-,(p_1)_+)\geq K_1/\lambda -c > \lambda(2\e+2+c)+2\e.$$
Thus $$\d_{X \cup \cH}((q_1)_-,(q_1)_+)\geq d_{X\cup \cH}((p_1)_-,(p_1)_+)-2\e\geq \lambda(2\e+2+c).$$
Let $q_2$ be the subpath of $q$ from $v_-$ to $v_+$. Then $q_1$ is a subpath of $q_2$ and hence $$\ell(q_2)\geq \ell(q_1) \geq \d_{X \cup \cH}((q_1)_-,(q_1)_+)\geq \lambda(2\e+2+c).$$
Since $q$ is a $(\lambda,c)$-quasi-geodesic, we obtain from the previous equation that 
$$2\e+2 \leq \ell(q_2)/\lambda -c  \leq\d_{X\cup \cH}((q_2)_-,(q_2)_+)= \d_{X\cup \cH}(v_-,v_+),$$
which contradicts $\d_{X\cup \cH}(v_-,v_+)\leq 1+2\e$.
\end{proof}

\begin{lem}\label{lem:mainshortpieces}
Let $\lambda\geq 1$, $c\geq 0$, $k\geq 0$ and $B>0$. There exists $K_2=K_2(\lambda,c,k,B)$ such that the following hold.
Let $p$ and $q$ be two paths in $\ga(G,X)$ with no parabolic shortenings. Suppose that $\wh{p}$ and $\wh{q}$  are $\e$-similar and are subpaths of two $k$-similar $(\lambda,c)$-quasi-geodesics, and
assume that the $X$-length of the components of $\wh{p}$ and $\wh{q}$ is 
at most $B$. 

Then $p$ and $q$ asyncronously $K_2$-fellow travel.
\end{lem}
\begin{proof}
Let $K_1$ be the constant of Lemma \ref{lem:boundingt1t2}. 
We will prove the lemma by induction on $\ell(\wh{p})$.
Suppose that $\ell(\wh{p})\leq K_1$.
Then, since all the components of $\wh{p}$ have $X$-length at most $B$,
$p$ is a path of length at most $B\cdot K_1$.
We can bound $\ell(q)$ in the following way. First observe that
 $\d_{X\cup \cH}(\wh{q}_-,\wh{q}_+)\leq 2\e+\ell(\wh{p})\leq 2\e + K_1$. Since $\wh{q}$ is  a $(\lambda,c)$-quasi-geodesic,
$\ell(\wh{q})\leq \lambda (2\e+K_1)+c$,
so $\ell(q)\leq B[\lambda (2\e+K_1)+c]$.
As $\d_X(p_-,q_-)\leq \e$ and $\d_X(p_+,q_+)\leq \e$,
$p$ and $q$ asynchronously $K_2$-fellow travel,
where $$K_2\coloneq B K_1+ B[\lambda (2\e+K_1)+c]+2\e \geq \ell(p)+\ell(q)+2\e.$$

Suppose that $\ell(\wh{p})> K_1$, and we have proven the result for shorter paths.
By Lemma \ref{lem:boundingt1t2}, there is a vertex $u$ in $\wh{p}$, $u\notin\{ \wh{p}_-,\wh{p}_+\}$, and a vertex $v$ in $\wh{q}$ such that $\d_X(u,v) \leq \e$.
Then $u$ and $v$ divide $\wh{p}$ and $\wh{q}$ into paths $\wh{p_1}$, $\wh{p_2}$ 
and $\wh{q_1}$, $\wh{q_2}$, respectively, $\wh{p_i}$ and $\wh{q_i}$ are $\e$-
similar, and $\ell(\wh{p}_i)<\ell(\wh{p})$. For $i=1,2$, let $p_i$ and $q_i$ be 
the subpaths of $p$ and $q$ projecting to $\wh{p}_i$ and $\wh{q}_i$, 
respectively, via $\wh{-}\,$. By induction, $p_i$ and $q_i$, $i=1,2$,  asynchronously $K_2$-fellow travel and hence, by Lemma \ref{lem:comptravel}, $p$ and $q$ asynchronously $K_2$-fellow travel.
\end{proof}

We note that the proof of Lemma \ref{lem:main} below follows the same lines as  \cite[Lemma 4.7]{NeumannShapiro}, except here we use (L$\forall$)/(L$\exists$) instead of the falsification by fellow traveler property. 
Recall that `special words' were introduced in Definition \ref{def:special}.

\begin{lem}\label{lem:main}
Let $\lambda \geq 1$ and $M,c,k\geq 0$.
Suppose that for all $\omega \in \Omega$, $(H_\omega, X\cap H_\omega, \cL_\omega)$ satisfies {\rm (L1)}.
There exists a constant $K_3=K_3(\lambda,c,k,M)$ such that the following holds.

Let $p$ and $q$ be paths in $\ga(G,X)$ 
such that
$\max\{\d_X(p_-, q_-),\d_X(p_-, q_-)\}\leq k$. 
Suppose that $\Lab(p), \Lab(q)\in \rell(X, \{\cL_\omega\}_{\omega\in \Omega})$, and $\wh{p}$ and $\wh{q}$ are
$(\lambda,c)$-quasi-geodesics without backtracking.
 
\begin{enumerate}
\item[{\rm (a)}] If all $(H_\omega, X\cap H_\omega, \cL_\omega)$ satisfy {\rm (L$\forall$)} with fellow traveler constant $M$, then $p$ and $q$ asynchronously $K_3$-fellow travel.

\item[{\rm (b)}] If all $(H_\omega, X\cap H_\omega, \cL_\omega)$ satisfy {\rm (L$\exists$)}   with fellow traveler constant $M$
and $\Lab(q)$ is special, then there exists a geodesic path $\varrho$ in $\ga(G,X)$ with
 $\Lab(\varrho)$ special and 
$\Lab(\wh{\varrho})\equiv\Lab(\wh{q})$, such that $\varrho$ and $p$ asynchronously $K_3$-fellow travel.
 \end{enumerate}

\end{lem}

\begin{proof}

Let $p$ and $q$ be two paths in $\ga(G,X)$ satisfying the hypotheses of the Lemma.

Let $\e=\e(\lambda, c, k)$ be the constant of Theorem \ref{thm:BCP}.
Without loss of generality we can assume that $\e\geq k$.
 By the Bounded Coset Penetration property (Theorem \ref{thm:BCP}), 
for every vertex $v$ of $\wh{p}$ there is a vertex $u$ in $\wh{q}$ with $\d_X(u,v)\leq \e$, 
and if $s$ is a component of $\wh{p}$ with $\d_X(s_-,s_+)>\e$, then there is a component $r$ 
of $\wh{q}$ connected to $s$ such that $\d_X(s_-,r_-)\leq \e$ and $\d_X(s_+,r_+)\leq \e$.

Let $\e_2=\e(\lambda, c,\e)$ be the constant of Theorem \ref{thm:BCP}.
Suppose that $\wh{p_1}< \wh{p_2}< \dots < \wh{p_n}$ are the (isolated)  components of $\wh{p}$ satisfying $\d_X((\wh{p_i})_-,(\wh{p_i})_+)>\e_2+2\e$.

{\bf Claim:}
There exist components $\wh{q_1}<\dots< \wh{q_n}$ of $\wh{q}$ such that $\wh{q_i}$ is 
connected to $\wh{p_i}$,  $\d_X((\wh{p_i})_-,(\wh{q_i})_-)<\e$ and 
 $\d_X((\wh{p_i})_+,(\wh{q_i})_+)<\e$ for $i=1,\dots, n$.

By Theorem \ref{thm:BCP} (2) and (3), there exist components
 $\wh{q_1}<\dots < \wh{q_n}$ of $\wh{q}$ and a permutation $\sigma$ of $\{1,\dots,n\}$ such that  for each $i=1,\dots, n$, $\wh{p_i}$ and $\wh{q_{\sigma(i)}}$ are connected,
 $\d_X((\wh{p_i})_-,(\wh{q_{\sigma(i)}})_-)\leq\e$ and $\d_X((\wh{p_i})_+,(\wh{q_{\sigma(i)}})_+)\leq \e$. The lack of backtracking in both $\wh{p}$ and $\wh{q}$ implies that an isolated component $\wh{p_i}$ cannot be connected to two different isolated components of $\wh{q}$.
Also notice that $\d_X((\wh{q_{\sigma(i)}})_-,(\wh{q_{\sigma(i)}})_+)> \e_2$. 
 
 We will show that $\sigma$ is the identity. Suppose that $i>\sigma(i)$ for some $i$.

Let $\wh{s}$ be the subpath of $\wh{p}$ from $\wh{p}_{-}$ to $(\wh{p_i})_+$ and
$\wh{r}$ the subpath of $\wh{q}$ from $\wh{q}_-$ to $\wh{q_{\sigma(i)}}_+$.
Then $\wh{s}$ and $\wh{r}$ are $\e$-similar (recall that $k\leq \e$) $(\lambda,c)$-quasi-geodesics without backtracking. Hence, by Theorem \ref{thm:BCP}(2),  
the components $\wh{p_1},\dots, \wh{p_{i-1}}$ are connected to some components of $\wh{s}$.
Since $\wh{q}$ does not backtrack,   for $l=1,\dots, i-1$, $\wh{p_l}$ should be connected to 
$\wh{q_{\sigma(l)}}$. Thus $\wh{q_{\sigma(1)}},\dots, \wh{q_{\sigma(i-1)}}$ lie in $\wh{s}$, 
contradicting that $i>\sigma(i)$. See Figure \ref{fig:sigma}.
\begin{figure}[ht]
\begin{center}
\includegraphics[scale=0.50]{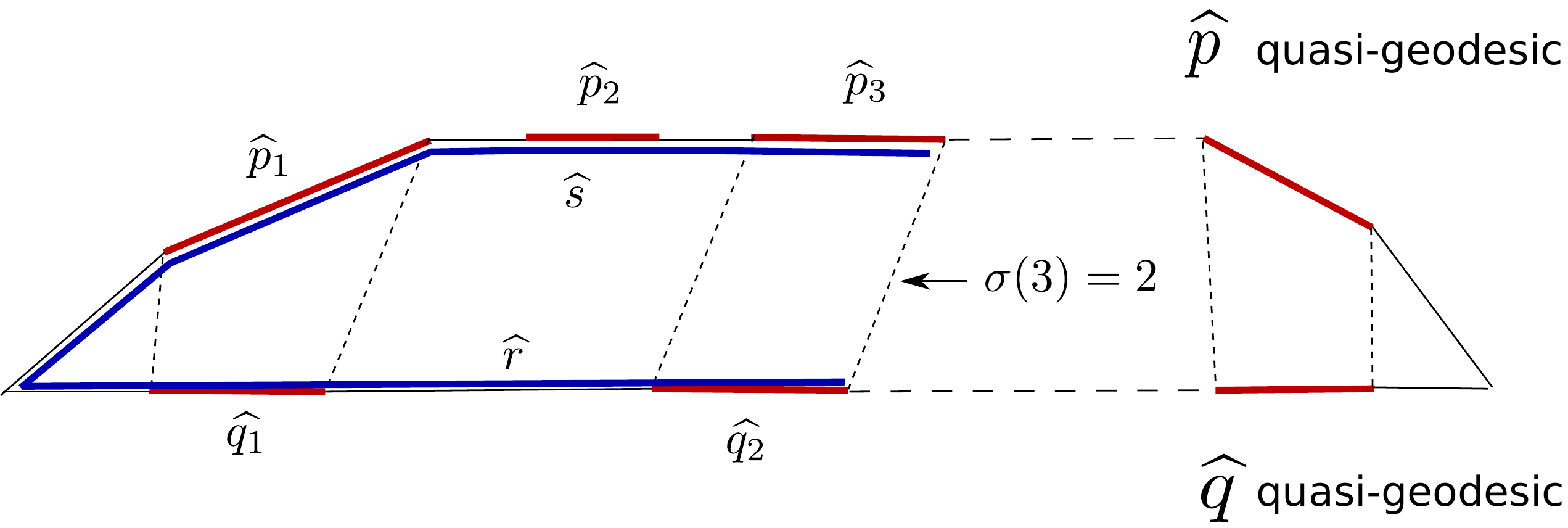} 
\end{center}
 \caption{Proof of Claim: the paths $\wh{r}$ and $\wh{s}$ (here $k=0$).}\label{fig:sigma}
\end{figure}

A similar argument holds if $i<\sigma(i)$. This completes the proof of the claim.

Now, for $i=1,\dots, n$, let $p_i$ be the  subpath of $p$ that is sent to the component $\wh{p_i}$ via Construction \ref{const:paths}, and similarly for $q_i$. Note that since we are assuming there are no parabolic shortenings, $p_i$ is a geodesic path in the copy $\ga(gH_\omega, X\cap H_\omega)$ of $\ga(H_\omega, X\cap H_\omega)$ contained in $\ga(G,X)$, where $g=(p_i)_-$. 


Let $a_i=(p_i)_-$, $b_i=(q_i)_-$, and view them as elements of $G$. Set $g_i=a_i^{-1}b$. Similarly, let $c_i=(p_i)_+$, $d_i=(q_i)_+$ and $h_i=d_i^{-1}c_i$. 
See Figure \ref{fig:perturbedgeodesic}. 
\begin{figure}[ht]
\begin{center}
\includegraphics[scale=0.30]{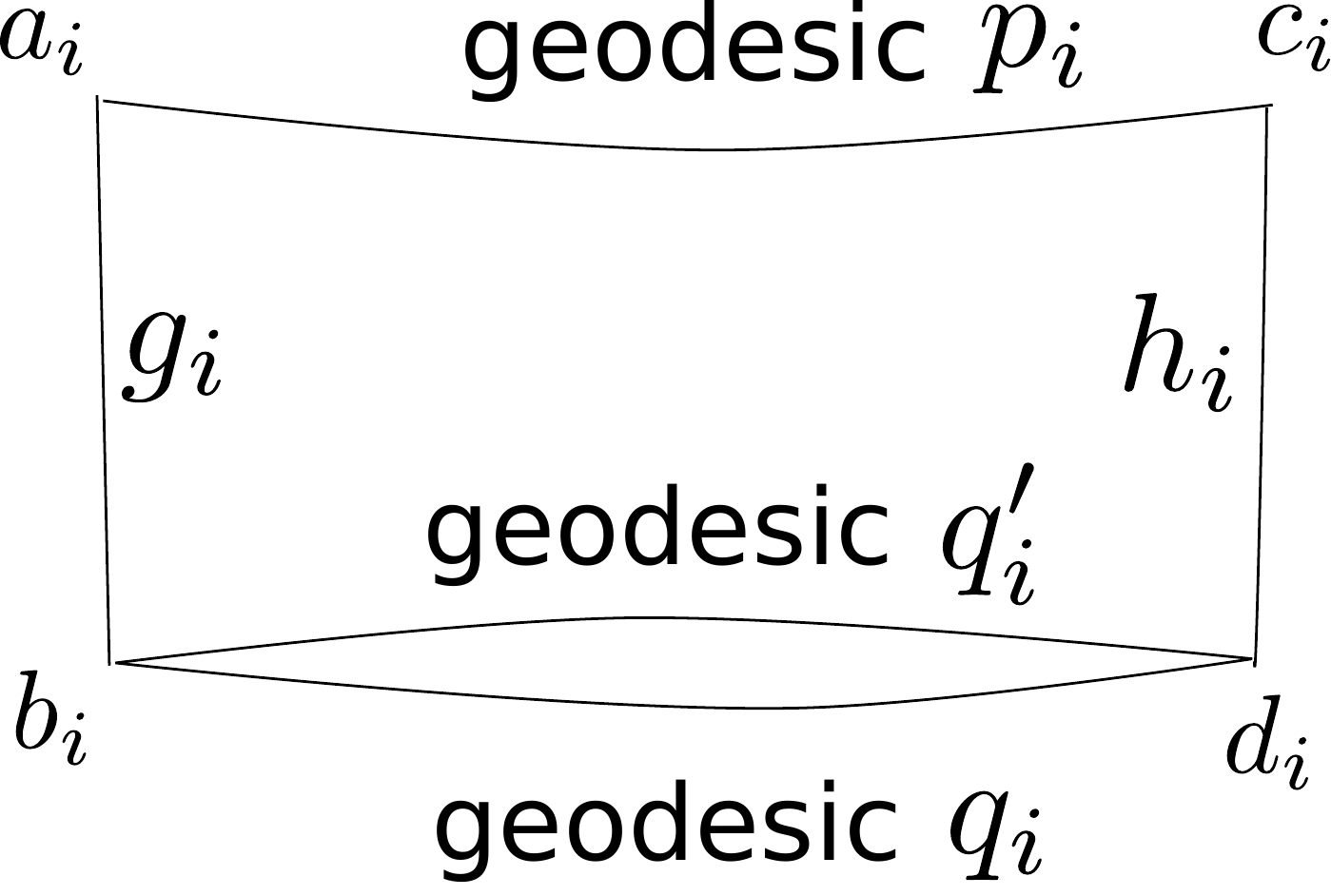} 
\end{center}
\caption{Geodesic $q_i$ replaced by geodesic $q'_i$.}\label{fig:perturbedgeodesic}
\end{figure}

Let $U_i\equiv \Lab(p_i)$. Since $\Lab(p)\in \rell(X, \{\cL_\omega\}_{\omega\in \Omega})$, by Construction \ref{const:paths},  $U_i$ is a word in $\cL_{\omega_i}$. 
Let $K_3=\max \{K_2,M(\e_2+2\e)\}$. We are going to prove the conditions (a) and (b) separately.

(a) 
Since $\Lab(q)\in \rell(X, \{\cL_\omega\}_{\omega\in \Omega})$, by Construction \ref{const:paths},   $V_i\equiv \Lab(q_i)$ is a word in $\cL_{\omega_i}$. 
By (L$\forall$), $p_i$ and $q_i$ asynchronously $(2\e M)$-fellow travel.

Let $r_0,\dots, r_n$ (resp. $s_0,\dots, s_n$) be the subpaths of $p$  (resp. $q$) such that $p$ is the composition of paths
$r_0, p_1, r_1, p_2, \dots, p_n, r_n$ ($q$ is the composition of paths $s_0, q_1,s_1,\dots, q_n, s_n$).
Since the components of each $\wh{r}_i$ have $X$-length bounded by $B=(\e_2+2\e)$, and $\wh{r}_i$ and $\wh{s}_i$ are $\e$-similar, Lemma \ref{lem:mainshortpieces} implies that $r_i$ and $s_i$ asynchronously $K_2$-fellow travel.

Then by Lemma \ref{lem:comptravel}, $p$ and $q$ asynchronously $K_3$-fellow travel.

(b) By (L$\exists$) there is a word $V_i$ over $X\cap H_{\omega_i}$ in $\cL_{\omega_i}$
such that $V_i=_G g_i U_i h_i$ and the paths $p_i$ and $q_i'$, asynchronously $(2\e M)$-fellow travel, where $(q_i')_-=(q_i)_-$ and $\Lab(q_i')\equiv V_i$. Notice that $(q_i')_+=(q_i)_+$.

We replace the subpaths $q_i$ of $q$ by the paths $q_i'$ to obtain $\varrho$. Notice that, by (L1), $q_i'$ and $q_i$ are geodesic over $X\cap H_{\omega_i}$, and since $\Lab(q_i')=_G\Lab(q_i)$, $\ell(q_i')=\ell(q_i)$.  Therefore  $\ell(q)= \ell(\varrho)$ and $\Lab(\varrho)$ is geodesic.
Since $\Lab(q)$ is special, by Lemma \ref{lem:specialchange},  $\Lab(\varrho)$ is special and $\Lab(\wh{\varrho})\equiv \Lab(\wh{q})$. 

Using the same argument as in case (a) with $\varrho$ instead of $q$ we get that $p$ and $\varrho$ asynchronously $K_3$-fellow travel.

\end{proof}

We are now ready to prove Proposition \ref{prop:geobicom}.
\begin{proof}[Proof of Proposition \ref{prop:geobicom}]
Let $\cH'$,  $\lambda$ and $c$ be the sets and constants provided by the 
Generating Set Lemma (Lemma \ref{lem:goodgenset}).  Fix $X$, a finite generating 
set for $G$ satisfying (i) and (ii), and let $\Phi=\Phi(X)$ be the set provided 
by the Generating Set Lemma.

By Lemma \ref{lem:specialrep},  $\cL=\geol(G,X)\cap \rell(X,\{\cL_\omega\}_{\omega\in \Omega})$ contains at least one representative for each element of $G$ and hence (L1) holds for $(G,X,\cL)$.
By (ii), all $(H_\omega, X\cap H_\omega, \cL_\omega)$ satisfy (L$\forall$) (resp. (L$\exists$)). 
We need to show that $(G,X,\cL)$ satisfies (L$\forall$) (resp. (L$\exists$)) with fellow traveler constant $M$.

Let $K_3(\lambda, c, 1, M)$ be the constant of Lemma \ref{lem:main}.
Let $U$ be any word in $\geol(G,X)\cap \rell(X,\{\cL_\omega\}_{\omega\in \Omega})$. Then $U$ does not contain subwords of 
$\Phi$ and has no parabolic shortenings. By Lemma \ref{lem:goodgenset}, 
$\wh{W}$ labels a $(\lambda,c)$-quasi-geodesic path without backtracking.
Let $x,y\in X\cup \{1\}$. 

Suppose that  all $(H_\omega, X\cap H_\omega, \cL_\omega)$ satisfy  (L$\exists$).

Let $V$ be any special word in $\geol(G,X)\cap \rell(X,\{\cL_\omega\}_{\omega\in \Omega})$ such that $V=_GxUy$.  By Lemma \ref{lem:specialrep} such $V$ exists.

It follows from Lemma \ref{lem:main}(b) that there is a special $V'$
such that $p$ and $q$ asynchronously $K_3$-fellow travel, where $p_-=1$, $q_-=x$, 
$\Lab(p)=U$ and $\Lab(q)=V'$. 
Now, by Lemma \ref{lem:esimilarfftp}, $(G,X,\cL)$ satisfies (L$\exists$).

Suppose that all $(H_\omega, X\cap H_\omega, \cL_\omega)$ satisfy (L$\forall$).
Let $V$ be any  word in $\geol(G,X)\cap \rell(X,\{\cL_\omega\}_{\omega\in \Omega})$ such that $V=_GxUy$.
Let $p,q$ be paths in $\ga(G,X)$,
$\Lab(p)\equiv U$, $\Lab(q)\equiv V$ 
with $p_-=1$ and $q_-=x$.
It follows from Lemma \ref{lem:main}(a) that $p$ and
$q$ $K_3$-asynchronously fellow travel.
Now, by Lemma \ref{lem:esimilarfftp}, $\geol(G,X)\cap \rell(X,\{\cL_\omega\}_{\omega\in \Omega})$ satisfies (L$\forall$).
\end{proof}

%
%
\section{\fftp{} and biautomaticity for relatively hyperbolic groups}\label{sec:relhyp_fftp}

The primary goal of this section is to prove our first main result, Theorem \ref{thm:main_intro} of the Introduction. We also prove the existence of geodesic biautomatic structures in Theorem \ref{thm:biaut}.

\begin{thm}\label{thm:fftpgenset}
Let $G$ be a finitely generated group, hyperbolic relative to a family of subgroups $\{H_\omega\}_{\omega\in \Omega}$.
Let $Y$ be a finite generating set of $G$ and let $\cH=\cup_{\omega\in \Omega} (H_\omega-\{1\})$. 

There exists a finite subset $\cH'$ of $\cH$ such that, for every finite generating set $X$ of $G$ satisfying 
\begin{enumerate}
\item[{\rm (i)}] $Y\cup \cH'\subseteq X\subseteq Y\cup \cH$, and
\item[{\rm (ii)}] for all $\omega\in \Omega$, $X\cap H_\omega$ generates $H_\omega$ and $(H_\omega, X\cap H_\omega)$ has \fftp{},
\end{enumerate}
the pair $(G,X)$ has \fftp.
\end{thm}

\begin{proof}
Let $\cH'$,  $\lambda$ and $c$ be the sets and constants provided by the Generating Set Lemma (Lemma \ref{lem:goodgenset}).  Fix $X$, a finite generating set for $G$ satisfying (i) and (ii), and let $\Phi=\Phi(X)$ be the set provided by the Generating Set Lemma.

By Lemma \ref{lem:factsRH}(ii) $\Omega$ is finite. So there is an $M>0$ such that $(H_\omega, X\cap H_\omega)$ has $M$-\fftp{} for every $\omega\in \Omega$.
Let $M_1= \max \{\ell_{X}(U) : U\in \Phi \}$.

Let $W$ be a non-geodesic word over $X$. There are several possibilities.

(i) If $W$ is not $2$-local geodesic, that is, $W\equiv AxyB$, where $xy=_G z$, $z\in X$, then $W$ asynchronously $1$-fellow travels with $AzB$.

(ii) If $W$ is $2$-local geodesic and has parabolic shortenings, then $W\equiv ACB$, where
$C$ is a word over $X\cap H_\omega$, for some $\omega\in \Omega$, that is not geodesic.
Then there exists a shorter word $C'$ over $X\cap H_\omega$  such that $C$ and $C'$ 
asynchronously $M$-fellow travel. Hence $W$ asynchronously $M$-fellow travels
with $AC'B$.

(iii) If $W$ contains some word $U\in \Phi$, then $W\equiv AUB$, and since $U$ is non-geodesic
there exists a shorter word $V$ such that $U=_G V$. Then $W$ and $AVB$ $M_1$-asynchronously 
fellow travel.

(iv) If $W$ is a $2$-local geodesic with no parabolic shortenings that does not contain any subword of $\Phi$, then by Lemma \ref{lem:goodgenset}, $\wh{W}$ is a $(\lambda, c)$-quasi-geodesic without  backtracking.  In this case we use the following lemma.

\begin{lem}\label{lem:main2}
There exists a constant $M_2=M_2(\lambda,c,\delta)$ such that the following holds.
For every non-geodesic path $p$ in $\ga(G,X)$ with no parabolic shortenings and with $\wh{p}$ a $(\lambda,c)$-quasi-geodesic without backtracking, there exists a shorter path $q$ in $\ga(G,X)$ with $q_-=p_-$ and $q_+=p_+$ such that 
$p$ and $q$ asynchronously $M_2$-fellow travel.
\end{lem}
\begin{proof}
We will apply Lemma \ref{lem:main} with $\cL_\omega=\geol(H_\omega, X\cap H_\omega)$. Notice that $(H_\omega, X\cap H_\omega, \cL_\omega)$ trivially satisfies (L1), and (L$\exists$) follows from Lemma \ref{lem:esimilarfftp}.

Let $p$ be a non-geodesic path  with no parabolic shortenings such that $\wh{p}$ 
is a $(\lambda,c)$-quasi-geodesic without backtracking. Notice that $\Lab(p)$ 
is in $\rell(X, \{\cL_\omega\}_{\omega\in \Omega})$. 
By Lemma \ref{lem:specialrep}, there exists a geodesic path $q$  in $\ga(G,X)$ with same end points as $p$ and such that $\Lab(q)$ is special.
Notice that $q$ does not contain subwords of $\Phi$ and has no parabolic shortenings. Thus $\Lab(q)\in \rell(X, \{\cL_\omega\}_{\omega\in \Omega})$ and by the Generating Set Lemma (Lemma \ref{lem:goodgenset}), $\wh{q}$ is a $(\lambda,c)$-quasi-geodesic without backtracking.

Let $M_2=K_3(\lambda,c,0,M)$ be the constant of Lemma \ref{lem:main}(b).
Then there is a geodesic path  $\varrho$ in $\ga(G,X)$ with same end points as $q$, and such that $p$ and $\varrho$ asynchronously $M_2$-fellow travel.
\end{proof}
 
Thus Lemma \ref{lem:main2} implies that there exists a
shorter word $W'$ such that $W$ and $W'$ asynchronously $M_2$-fellow travel.

So in all cases, for $K=\max\{1,M,M_1,M_2\}$, a non-geodesic word over $X$ asynchronously  $K$-fellow travels
with a shorter word and we thus obtain the falsification by fellow traveler property for the group $G$ with generating set $X$. 
\end{proof}

We are now ready to prove Theorem \ref{thm:main_intro} of the Introduction.

\begin{proof}[Proof of Theorem \ref{thm:main_intro}]
Suppose that $G$ is  hyperbolic relative to $\{H_\omega\}_{\omega\in\Omega}$. 
For each $\omega\in \Omega$ there is a finite generating set $Y_\omega$ of $H_\omega$
such that $(H_\omega,Y_\omega)$ has \fftp.
Let $Y$ be a finite generating set of $G$ such that
$Y\cap \cH=\cup_{\omega\in \Omega} Y_\omega$.
 
According to Theorem \ref{thm:fftpgenset} it is enough to show that 
for any finite subset $\cH'$ of $\cH$ there is a finite generating set $X$ of $G$ such that 
$Y\cup \cH'\subseteq X\subseteq Y\cup \cH$ and $(H_\omega, H_\omega\cap X)$ has \fftp{} for all $\omega\in \Omega$.

By Lemma \ref{lem:nocancellation}, there is $m>0$ such that
if $g\in H_\omega\cap H_\mu$, $\mu\neq \omega$, then $|g|_Y\leq m$.

For each $\omega\in \Omega$, let $$k_\omega = \max \{|h|_{Y_\omega}: h\in  (\cH'\cup \Theta_Y(m))^{\pm 1}\cap H_\omega\}+1$$ and let $k=\max\{ k_\omega : \omega \in \Omega\}$. 
Let $$X=Y\cup 
(\bigcup_{\omega \in \Omega} \{h\in H_\omega : |h|_{Y_\omega}\leq k \}).$$

Then $X$ is a finite  generating set of $G$ and  $Y\cup  \cH'\subseteq X \subseteq Y\cup \cH$.
Observe that  $H_\omega \cap X= \{h\in H_\omega : |h|_{Y_\omega}\leq k\}$ and then by Proposition \ref{lem:extendingfftpgenset}, $(H_\omega, H_\omega\cap X)$ has the falsification by fellow traveler property.
\end{proof}
%
%
\subsection{Geodesically biautomatic relatively hyperbolic groups}

After presenting a preliminary version of this paper at the Group Theory 
International Webinar,  Olga Kharlampovich asked if it was possible to use 
our techniques to give another proof of  Rebbechi's result about 
biautomaticity of groups hyperbolic relative to biautomatic groups with a prefix 
closed normal form \cite{Rebbechi}. We can indeed recover 
this result when some additional assumptions are made, assumptions which hold for virtually abelian groups, for example.
One of the main technicalities in Rebbechi's thesis is dealing with a variation of the falsification by fellow traveler property to obtain that certain languages are regular. In our approach, however, we use generating sets with the standard falsification by fellow traveler property to obtain the regularity of the language. Moreover, the biautomatic structure  we obtain is geodesic.

We use \cite[Lemma 2.5.5]{ECHLPT} as definition for biautomaticity.
\begin{defn}
Let $G$ be a group, $X$ a finite  generating set and 
$\cL$ a regular language 
over $X$ that maps onto $G$.
 Then $\cL$ is  a {\it biautomatic structure} if  
 there exists a constant $M$ such that for each 
 $W\in \cL$, each pair $x,y\in X\cup \{1\}$,
 and all $U\in \cL$ with $U=_G xWy$, 
 the paths  $p$ and $q$ synchronously $M$-fellow travel,
 where $\Lab(p)\equiv W$, $\Lab(q)\equiv U$,
 $p_-=1$ and $q_-=x$.

We say that $\cL$ is a {\it geodesic biautomatic} structure if all words in $\cL$ are geodesic over $G$. It is clear from the definition that if $\cL$ is a geodesic biautomatic structure for $G$, then $\cL$ satisfies (L1) and (L$\forall$).

\end{defn}

The proof of the following lemma is basically the same as that of \cite[Theorem 2.5.9]{ECHLPT}.
\begin{lem}\label{lem:prefixclosed}
Let $\cL$ be a geodesic biautomatic structure for $(G,X)$
and let $\cL^{{\rm pc}}$ be the prefix closure of $\cL$.
Then $\cL^{{\rm pc}}$ is a geodesic biautomatic structure.
\end{lem}

The following is based on \cite[Proposition 4.1]{NeumannShapiro}, and can be viewed as the restriction of the \fftp{} to prefix-closed regular languages over $X$.
\begin{prop}\label{prop:restrictedfftp}
Let $X$ be a finite  generating set for $G$ and  $\cL\subseteq X^*$ be a prefix-closed regular language.
Suppose that there is $C>0$ such that for any non-geodesic word $U\in \cL$
there is a word $V$, $\ell(V)<\ell(U)$, $V=_G U$, such that $U$ and $V$
asynchronously $C$-fellow travel.
Then $\geol(G,X)\cap \cL$ is a regular language.
\end{prop}
\begin{proof}
Let $A$ be the automaton  of \cite[Proposition 4.1]{NeumannShapiro},
that is, $X$ is the input alphabet,
the states are given by the set $$S= \{\rho\}\cup \{ \phi \in \mathrm{Maps}(\Theta(C) \rightarrow \{-C,\dots, -1,0,1,\dots, C\}) \mid \phi(1)=0\}$$ where  $\rho$ is the fail state and $\Theta(C)=\{g\in G : |g|_X\leq C\}$. 
The transition function $T\colon S\times X\to S$ is given by
$T(\phi,x)=\rho$ if $\phi(x)\neq 1$, and if this is not the case,
then 
for $g\in \Theta(C)$,
$T(\phi,x)(g)=\phi(xg)-1$ if $xg\in \Theta(C)$ and $T(\phi,x)(g)=\min \{ \phi(h): h\in \Theta(C), \d(h,xg)\leq 1\}$ if $xg\notin \Theta(C)$. We note that all but $
\rho$ are accepting states and therefore $\cL(A)$, the language accepted by $A$, is prefix-closed.

Given a word $W$ accepted by $A$ at state $\phi$, we get that
for all $g\in \Theta(C)$ there is a word $W'$ of length $\ell(W)+\phi(g)$
such that $Wg=_G W'$ and $W$ and $W'$ asynchronously $C$-fellow travel.

{\bf Claim:} $W\in \cL$ is accepted by $A$ if and only if $W$ is geodesic.

We prove the claim by induction on $\ell(W)$. If $\ell(W)=1$ it is easy to see
that $A$ accepts $W$.

So suppose that our claim holds for words $U$, $\ell(U)\leq n$.
Let $W\in \cL$, $\ell(W)=n+1$. Since $\cL$ is prefix-closed,
there is $x\in X$ and $U\in \cL$ such that $W\equiv Ux$, $\ell(U)=n$.

If $U$ is not geodesic, $W$ is not geodesic. Since, by induction, $U$ is not accepted by $A$, neither is $W$.
So assume that $U$ is geodesic and has been accepted at state $\phi$.
If $W$ is not geodesic, by assumption there is a geodesic word $V=_G W$
 such that $W$ and $V$ asyncronously $C$-fellow travel.
Since $U$ is geodesic and $W$ is not, $\ell(V)\leq \ell(U)$.
Thus $\phi(x)=\ell(V)-\ell(U)\neq 1$, and $W$ is not accepted.

 Conversely, if $W$ is geodesic then $\phi(x)=1$, so $T(\phi,x)\neq \rho$.
This completes the proof of the claim.

By the claim, $\geol(G,X)\cap \cL= \cL(A)\cap \cL$.  Since the class of regular languages is closed under intersection, $\geol(G,X)\cap \cL= \cL(A)\cap \cL$ is regular.
\end{proof}

We can now state a criterion for biautomaticity.

\begin{thm}\label{thm:biaut}
Let $G$ be a finitely generated group, hyperbolic relative to a family of subgroups $\{H_\omega\}_{\omega\in \Omega}$.
Let $Y$ be a finite  generating set of $G$ and $\cH=\cup_{\omega\in \Omega}(H_\omega-\{1\})$.

There exists  a finite subset $\cH'\subseteq \cH$ such that, for every finite  generating set $X$ of $G$ satisfying
\begin{enumerate}
\item[{\rm (i)}] $Y\cup \cH'\subseteq X\subseteq Y\cup \cH$, and
\item[{\rm (ii)}] for all $\omega\in \Omega$, there is a geodesic  biautomatic structure $\cL_\omega$  for $(H_\omega, H_\omega\cap X)$,
\end{enumerate}
the set $\geol(G,X)\cap\rell(X,\{\cL_\omega^{{\rm pc}}\}_{\omega\in \Omega})$ is a geodesic  biautomatic structure for $(G,X)$.
\end{thm}
\begin{proof}
Let $\cH'_1$ be the set of Proposition \ref{prop:geobicom}.
Let $\cH'_2$,  $\lambda$ and $c$ be the sets and constants provided by the Generating Set Lemma (Lemma \ref{lem:goodgenset}).
Let $\cH'=\cH_I\cup \cH'_1\cup \cH_2'$.
Fix $X$, a finite  generating set for $G$ satisfying (i)-(ii), and let $\Phi=\Phi(X)$ be the set provided by the Generating Set Lemma.

By Lemma \ref{lem:prefixclosed}, the language $\cL_\omega^{{\rm pc}}$ is a prefix-closed
biautomatic structure, and
by Lemma \ref{lem:rel}  $\rell(X,\{\cL_\omega^{{\rm pc}}\}_{\omega \in \Omega})$ is regular and prefix-closed.

Let $\cL_\Phi=X^*- \cup_{W\in \Phi} X^*WX^*$ be the set of words that do not contain
subwords in $\Phi$. Since $\Phi$ is finite, this language is regular.
By definition $\cL_\Phi$ is prefix-closed.
Thus $\cL=\rell(X, \{\cL_\omega\}_{\omega\in \Omega})\cap \cL_\Phi$ is regular, prefix-closed and for all $W\in \cL$, $\wh{W}$ is a $(\lambda,c)$-quasi-geodesic with no parabolic shortenings. Since for any geodesic
$U$ the word $\wh{U}$ is a $(\lambda,c)$-quasi-geodesic we obtain, by Lemma \ref{lem:main}(a), that if $W\in \cL$ is not geodesic, then it asynchronously  $M$-fellow travels with a geodesic word.

By Proposition \ref{prop:restrictedfftp}, $\geol(G,X)\cap \cL=\geol(G,X)\cap \rell(X,\{\cL_\omega^{\rm pc}\}_{\omega\in \Omega})$ is regular. The synchronous fellow traveler property follows from Proposition \ref{prop:geobicom} and Remark \ref{rem:Lsync}.
\end{proof}
If $\mathsf{ShortLex}(H_\omega,H_\omega\cap X)$ is a biautomatic structure for $H_\omega$, then it clearly satisfies (ii) of  Theorem  \ref{thm:biaut}. We have the following.
\begin{thm}\label{thm:shortlex}
Suppose the assumptions of Theorem \ref{thm:biaut} hold. Then
if $X$ is an ordered generating set satisfying (i) of Theorem \ref{thm:biaut}  and
\begin{enumerate}
\item[{\rm ($\mathrm{ii}^\prime$)}] for all $\omega\in \Omega$, $\mathsf{ShortLex}(H_\omega,H_\omega\cap X)$ is a biautomatic structure for $H_\omega$ (here the order on $H_\omega\cap X$ is the restriction of the order on $X$)
\end{enumerate}
then $\mathsf{ShortLex}(G, X)$ is a biautomatic structure for $G$. 
\end{thm}
\begin{proof}
If ($\mathrm{ii}^\prime$) holds, then (ii) also holds, so we have that, 
 by Theorem \ref{thm:biaut}, $$\cL=\geol(G,X)\cap\rell(X,\{\cL_\omega\}_{\omega\in \Omega}) $$ is a geodesic biautomatic structure for $G$, where $\cL_\omega=\mathsf{ShortLex}(H_\omega,H_\omega\cap X)$.
It is easy to see that $\mathsf{ShortLex}(G,X)\subseteq\cL$.
Then, by \cite[Theorem 2.5.1]{ECHLPT},
 $\mathsf{ShortLex}(G,X)$ is a regular language.  The biautomaticity follows from the fellow traveler property for words in $\cL$.
\end{proof}


\section{Cayley graphs with bounded conjugacy diagrams}\label{sec:conjhyp}

Let $G$ be a group and $X$ a  generating set.
An $n$-gon $p_1 \cdots p_n$ in $\ga(G,X)$ is a sequence of paths
$p_1, p_2, \dots, p_n$ in $\ga(G,X)$ such that $(p_i)_+=(p_{i+1})_-$ for $i=1,\dots, n-1$, and $(p_n)_+=(p_1)_-$.
A {\it conjugacy diagram} over $(G,X)$ is a quadruple $(p,q,r,s)$ where
$p$ and $q$ are paths, $r,s$ are geodesic paths with the same label, and $p s q^{-1} r^{-1}$ is a 4-gon in $\ga(G,X)$. See Figure \ref{fig:conjquad}.

\begin{figure}[!ht]
\includegraphics[scale=0.4]{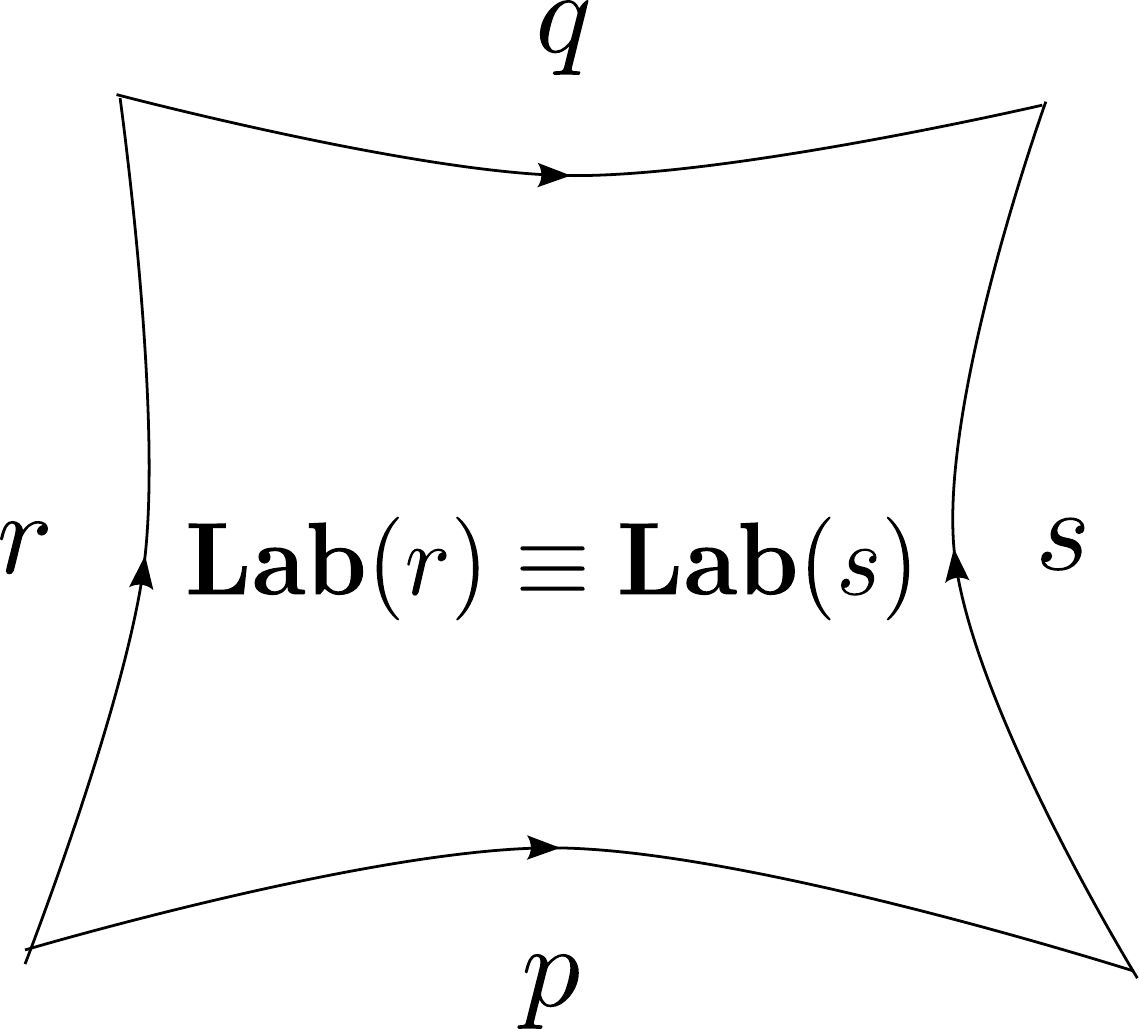}
\caption{A conjugacy diagram over $(G,X)$}
\label{fig:conjquad}
\end{figure}

Suppose that $(p,q,r,s)$ and $(p',q',r',s')$ are conjugacy diagrams. We write $(p,q,r,s)\sim(p',q',r',s')$ if $\Lab(p')$ and $\Lab(q')$ are  cyclic permutations of $\Lab(p)$ and $\Lab(q)$, respectively. 
Let $\lambda \geq 1$ and $c\geq 0$. A conjugacy diagram $(p,q,r,s)$ is a {\it minimal conjugacy $(\lambda, c)$-diagram} if all cyclic permutations of $\Lab(p),\Lab(q)$ are $(\lambda,c)$-quasi-geodesic  and
$\ell(r)\leq \ell(r')$ for all  $(p',q',r',s')\sim(p,q,r,s)$.

 Let $k \geq 0$. We will say that $(G,X)$ has {\it $k$-bounded minimal conjugacy $(\lambda,c)$-diagrams} if  for every minimal conjugacy $(\lambda,c)$-diagram $(p,q,r,s)$ 
 \begin{equation}\label{kbounded}
 \min\{ \max\{\ell(p),\ell(q)\}, \ell(r) \}\leq k.
 \end{equation}

\begin{ex} Let $G$ be a finitely generated abelian group and $X$ any finite  generating set. Then $(G,X)$ has $0$-bounded minimal conjugacy $(1,0)$-diagrams.

More generally, if $G$ is a finitely generated finite-by-abelian group, then every conjugacy class is finite, and two elements in the same conjugacy class are conjugated by an element of the finite subgroup. 
This is exactly the class of finitely generated  BFC-groups (groups with bounded finite conjugacy classes). BFC-groups were studied by B. H. Neumann in \cite{BFC}.
\end{ex}

The idea behind bounded minimal conjugacy $(1,0)$-diagrams comes from solving the conjugacy problem in free groups when working over free generating sets. There, two words are conjugate if after cyclic reduction one word is a cyclic permutation of the other. So in free groups minimal conjugacy $(1,0)$-diagrams have the bound $k=0$ for free generating sets. This can be generalized to hyperbolic groups. 
The following lemma is a well known result, which can be found in a slightly different version in \cite[III.$\Gamma$.Lemma 2.9]{BH} and \cite[Lemma 7.3]{BridsonHowie}. We notice
that its proof only depends on the $\delta$-hyperbolicity of the Cayley graph, and 
it remains valid if one relaxes geodesics to quasi-geodesics. We leave the details of
the proof to the reader.

\begin{lem}\label{lem:conjhyp}
Let $G$ be a group and $Z$ a (possibly infinite) generating set such that $\ga(G,Z)$ is hyperbolic. Given $\lambda \geq 1,$ $c\geq 0$, there exists $K=K(G,Z,\lambda,c)$ such that    
 $(G,Z)$ has $K$-bounded minimal conjugacy $(\lambda,c)$-diagrams.
\end{lem}
The previous lemma motivates the following definition.
\begin{defn}
Let $G$ be a group and $X$ a generating set. We say that $(G,X)$ has {\it bounded conjugacy diagrams} (\bcd{}) if there is some $k \geq 0$ such that $(G,X)$ has $k$-bounded minimal conjugacy $(1,0)$-diagrams.
\end{defn}


\begin{ex}
As already mentioned, abelian and hyperbolic groups have \bcd{} with respect to any generating set.
It is not hard to see that right-angled Artin groups have \bcd{} with respect to the standard generating set.
\end{ex}

The \bcd{} property turns out to be useful for efficiently solving the conjugacy problem (see \cite[Section 7]{BridsonHowie} or Remark \ref{rem:complexity}). However, in order to prove the regularity of the language of conjugacy geodesics, a weaker condition suffices (see Proposition \ref{prop:CHHR38}). This condition requires that any long enough cyclic geodesic word in $X^*$ that is not a conjugacy geodesic has a cyclic permutation that can be shortened via conjugation by an element of bounded length, as the following definition explains.

\begin{defn}\label{def:nsc}
Let $B\geq 0$. We say that a group has the property $B$-\nsc{}, which stands for the {\em neighboring shorter conjugate} property, if for any cyclic geodesic $U$ such that $\ell(U) \geq B$ and  $U$ is conjugate to some element of length  less than $\ell(U)$, there is a cyclic permutation $U'$ of $U$ and $g\in G$ with  $|g|_X\leq B$, such that $|gU'g^{-1}|_X<\ell(U)$.

We say then that $(G,X)$ satisfies the \nsc{} property if there is some $B \geq 0$ such that $(G,X)$ is $B$-\nsc{}. 

\end{defn}
\begin{rem}
If $(G,X)$ has $A$-\bcd{} then it also has $A$-\nsc{}, where $A \geq 0$.
\end{rem}

\section{Conjugacy diagrams in relatively hyperbolic groups}\label{sec:conjdiag}

In this section we assume that $G$ is hyperbolic with respect to $\{H_\omega\}_{\omega\in \Omega}$, $X$ is a finite  generating set of $G$ and $\cH=\cup (H_\omega -\{1\})$. Also, $\lambda\geq 1$ and $c>0$ are fixed constants.

We first need the following result, which is a version of \cite[Proposition 3.2]{OsinPF}.
\begin{lem}\label{lem:gon}
There exists $D=D(G,X,\lambda,c)>0$ such that the following hold. Let 
 $\mathcal{P}=p_1p_2\cdots p_n$ be an $n$-gon in $\ga(G,X\cup \cH)$ and  $I$ a distinguished subset  of sides of $\mathcal{P}$ such that if $p_i\in I$, $p_i$ is an isolated component in $\mathcal{P}$, and
if $p_i\notin I$, $p_i$ is a $(\lambda,c)$-quasi-geodesic. Then
$$\sum_{i\in I}\d_X((p_i)_-,(p_i)_+)\leq D n.$$
\end{lem}

For the rest of the section $D$ will denote the constant of Lemma \ref{lem:gon}, and we assume $D>1$.

The following corollary is an immediate application of Lemma \ref{lem:gon}.
\begin{cor}\label{cor:exem}
Let $p_1p_2p_3p_4$ be a 4-gon in $\ga(G,X\cup \cH)$, where all components of $p_1$ are edges and are isolated in the closed path $p_1p_2p_3p_4$. Suppose that $p_1,p_2,p_3,p_4$ are $(\lambda,c)$-quasi-geodesics.
Then $\d_X((p_1)_-,(p_1)_+)\leq (2\ell(p_1)+3)D$.
\end{cor}
\begin{proof}
Consider the polygon with sides $p_2,p_3,p_4$,  and the edges of $p_1$.
Let $I$ be the (isolated) components of $p_1$. By Lemma \ref{lem:gon}, the total sum of the  $X$-lengths of the components of $p_1$ is less than or equal to $(\ell(p_1)+3)D$.
Then $\d_X((p_1)_-,(p_1)_+)\leq \ell(p_1)+ (\ell(p_1)+3)D\leq (2\ell(p_1)+3)D$.
\end{proof}

We need some extra  terminology.  A minimal conjugacy $(\lambda,c)$-diagram $(p,q,r,s)$ in $\ga(G,X\cup \cH)$ is {\it without  vertex backtracking} if none of the cyclic permutations of  $\Lab(p)$ and $\Lab(q)$ vertex backtracks. In particular, all components of $p$ and $q$ are edges. Also, since $r$ and $s$ are geodesic, all their components are edges. We say that a component is {\it isolated in $(p,q,r,s)$} if it is isolated in the closed path $rqs^{-1}p^{-1}$.
 Finally, we say that $p$ (resp. $q$) is {\it parabolic} if $\Lab(p)$ (resp.  $\Lab(q)$) is a letter from $\cH$.

From now on $(p,q,r,s)$ will be a minimal conjugacy $(\lambda,c)$-diagram without vertex backtracking and $K>0$ the constant of Lemma \ref{lem:conjhyp} for $(G,X\cup \cH)$, that is,
\begin{equation}\label{eq:hyprect}
\min \{\max\{\ell(p),\ell(q)\},\ell(r)\}\leq K.
\end{equation}
The rest of this section is concerned with analyzing such minimal conjugacy diagrams, first arriving at Theorem \ref{thm:bmgcq}, where these diagrams are shown to be bounded (in terms of $X$-length) if $p$ and $q$ are not parabolic, and culminating with the proof of Theorem \ref{thm:rh_conjhyp}.

The first and easiest case is when all components are isolated in $(p,q,r,s)$.
\begin{cor}\label{cor:exem2}
If all components of  $p,q,r,s$ are isolated in $(p,q,r,s)$, then  $$\min\{\max \{\d_X(p_-,p_+),\d_X(q_-,q_+)\}, \d_X(r_-,r_+)\}\leq (2K+3) D.$$
\end{cor}
\begin{proof}
The result follows from \eqref{eq:hyprect} and Corollary \ref{cor:exem}.
\end{proof}

The second case, in which non-isolated components are allowed, is split into a number of lemmas that explain either how the components are positioned in the diagram, or how their length with respect to $X$ can be bounded in terms of other known parameters.
 Each of the lemmas exploits the minimality of $(p,q,r,s)$. 
 
 We say that a vertex $v$ is {\em adjacent} to an edge $e$ if $v=e_{-}$ or $v=e_{+}$. 

\begin{lem}\label{lem:rconnect}
Let $u$ be a vertex of $p$ and $v$ a vertex of $r$.
If $\d_{X\cup \cH}(u,v)\leq 1$, then $v$ is adjacent to the
first edge of $r$.
\end{lem}
\begin{proof}
The proof of this lemma is easiest to follow while consulting Figure \ref{fig:minquad}.
\begin{figure}[ht]
\begin{center}
\includegraphics[scale=0.4]{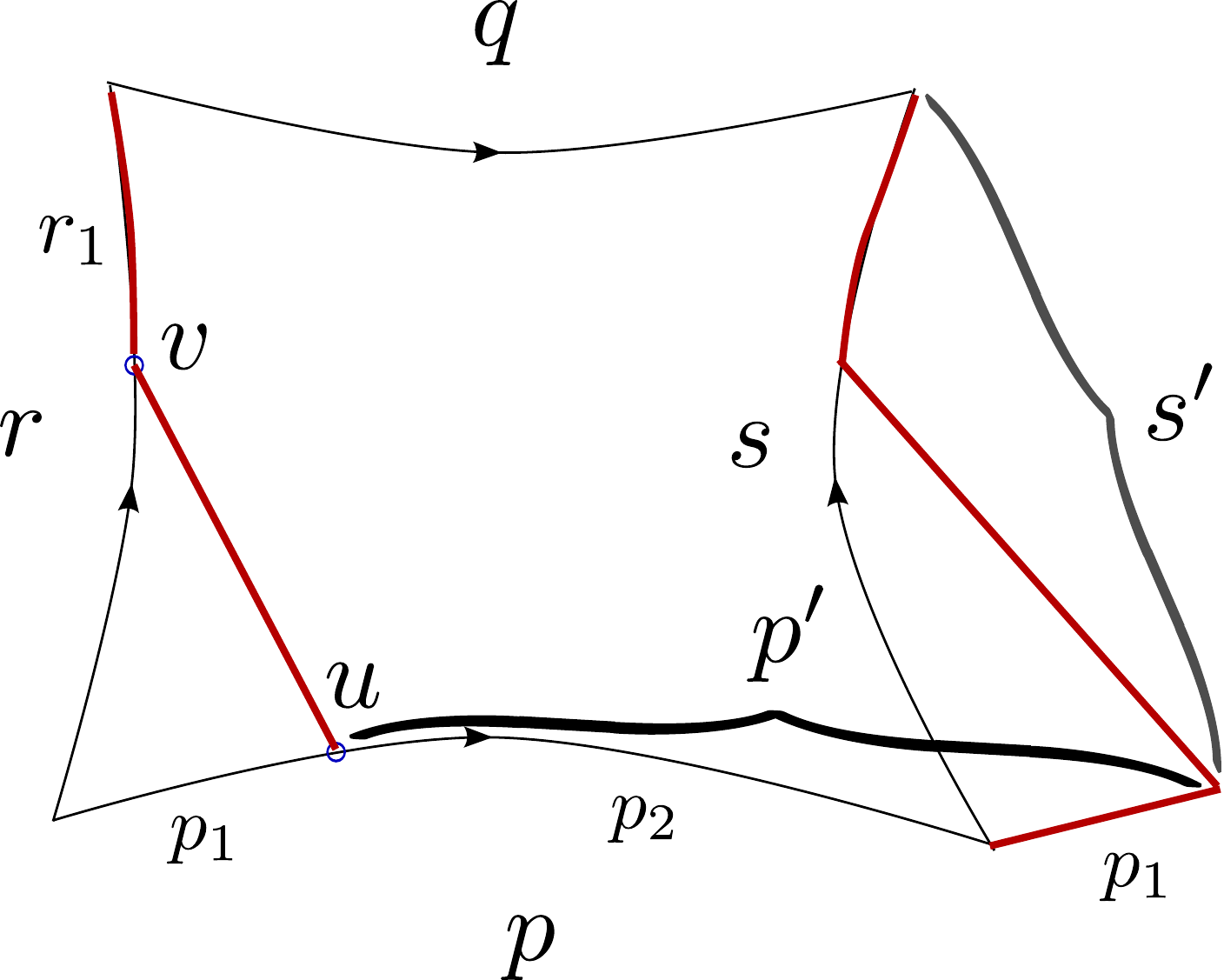} 
\end{center}

\caption{Proof of Lemma \ref{lem:rconnect}.}
\label{fig:minquad}
\end{figure}

Suppose $v$ is not adjacent to the first edge of $r$, and
let $r_1$ be the subpath of $r$ from $v$ to $r_+$. Note that
$\ell(r_1)\leq \ell(r)-2$.
Let $z\in X\cup \cH$ be the label of the edge form $u$ to $v$, and let $p_1$ be the subpath of $p$ from $p_-$ to $u$ and $p_2$ be the subpath of $p$ from $u$ to $p_+$.

Let $U_1\equiv \Lab(p_1)$, $U_2\equiv \Lab(p_2)$ and
$V_1\equiv \Lab(r_1)$.

If $p'$ is the path with label $U_2U_1$ starting at $u$, and
$r'$ and $s'$ are the paths with labels $zV_1$ starting at $p'_-$ and $p'_+$, respectively, it is easy to check that $q'=q$ satisfies $q'_-=r'_+$ and $q'_+=s'_+$.

Thus $(p',q',r',s')$ is a conjugacy diagram, $\ell(r')=\ell(r_1)+1<\ell(r)$
and $\Lab(p')$ and $\Lab(q')$ are cyclic permutations of $\Lab(p)$ and $\Lab(q)$, and hence $(p',q',r',s')\sim(p,q,r,s)$. This contradicts the minimality of $(p,q,r,s)$.
\end{proof}

The next lemma shows that if two consecutive sides of $(p,q,r,s)$ have connected components, these components occur at the same corner.

\begin{lem} \label{lem:corners}
If $p$ has a component connected to a component of $r$, these components must be the first edge of $p$ and the first edge of $r$, respectively. The same behavior holds for the pairs of sides $(q,r)$, $(p,s)$ and $(q,s)$.
\end{lem}

\begin{proof}
Suppose that $p_1$ is a component of $p$, $r_1$ a component of $r$, and $p_1$ and $r_1$ are connected $H_\omega$-components.
Also suppose that $r_1$ is not the first edge. Then $(r_1)_+$ is not adjacent to the first edge of $r$, and $\d_{X\cup \cH}((r_1)_+),(p_1)_+)\leq 1$, which contradicts Lemma \ref{lem:rconnect}. So $r_1$ must be the first edge of $r$.

Suppose now $p_1$ is not the first edge of $p$, that is, $(p_1)_-\neq p_-$. Since $r_1$ is the first edge of $r$, $p_-=r_-=(r_1)_-\in (p_1)_+ H_\omega$. Thus the subpath of $p$ from $p_-$ to $(p_1)_+$ has length at least 2 and its label represents some element of $H_\omega$, contradicting that $(p,q,r,s)$ is without vertex backtracking.
\end{proof}

Let $m$ be the constant provided by Lemma \ref{lem:nocancellation}. 
In particular,
\begin{equation}\label{eq:8m}
\text{ if $g\in H_\mu\cap H_\omega$, $\omega\neq \mu$, then $|g|_X\leq m$.}
\end{equation}

\begin{lem}\label{lem:reducing-conjugators}
If $p$ is not parabolic, then there exists a minimal conjugacy $(\lambda,c)$-diagram $(p',q,s',r')$ such that $(p',q,s',r')\sim(p,q,r,s)$ and
 if a component $p_1$ of $p'$ is connected to a component $s_1$ of $s'$, then $\d_X((s_1)_-,(s_1)_+)\leq m$. Also, $\Lab(s')$ and $\Lab(s)$ share a suffix of length $\ell(s)-1$.
\end{lem}

\begin{proof}
Suppose that $(p,q,r,s)$ does not satisfy the claim of the lemma. That is, there exists a component $p_1$ of $p$ connected to a component $s_1$ of $s$ such that $\d_X((s_1)_-,(s_1)_+)>m$.
By Lemma \ref{lem:corners}, $p_1$ is the last edge of $p$ and $s_1$ is the first edge of $s$.

Thus there is $\omega\in \Omega$ such that $\Lab(p_1)\equiv h_1\in H_\omega$,
and $\Lab(p)\equiv Uh_1$. Since $p$ is not parabolic, $U$ is non-empty.
Suppose that $\Lab(s_1)\equiv h_2\in H_\omega$, $\Lab(s)\equiv h_2V$. Let $h_3=h_1h_2$.  Then $(h_3V) \Lab(q)(h_3V)^{-1}=_G h_1 U$, and we have a minimal conjugacy $(\lambda,c)$-diagram $(p',q,r',s')\sim (p,q,r,s)$, 
where $\Lab(p')\equiv h_1U$, $\Lab(r')\equiv \Lab(s')=h_3 V$ and $\Lab(s')$ and $\Lab(s)$ share a suffix of length $\ell(s)-1$. 

If $(p',q,r',s')$ does not satisfy the claim of the lemma,
then there is a component $p'_1$ of $p'$ connected to a 
component $s'_1$ of $s'$ with $\d_X((s_1')_-,(s_1')_+)>m$. By Lemma 
\ref{lem:corners}, $p'_1$ is the last edge of $p'$ and $s'_1$ 
the first edge of $s'$.
Thus there is $\mu\in \Omega$ such that 
$\Lab(p'_1)=g_1$ and $\Lab(s'_1)=g_2\in H_\mu$. 
 Recall that on one hand $\Lab(s_1')=h_3\in H_\omega$, 
 and on the other hand $\Lab(s_1')=g_2\in H_\mu$. Since $|\Lab(s_1')|_X>m$, it follows 
by \eqref{eq:8m} that $\mu=\omega$. 
Hence $\Lab(p)\equiv U'g_1h_1$ with $g_1,h_1\in H_\omega$, contradicting the no vertex backtracking hypothesis.
\end{proof}

Similarly we obtain the following.
\begin{lem}\label{lem:reducing-conjugators2}
If $q$ is not parabolic, there exists a minimal conjugacy $(\lambda,c)$-diagram $(p,q',s',r')$ such that  $(p,q',s',r')\sim(p,q,r,s)$ and
 if a component $q_1$ of $q'$ is connected to a component $s_1$ of $s'$, then $\d_X((s_1)_-,(s_1)_+)\leq m$. Also, $\Lab(s')$ and $\Lab(s)$ share a prefix of length $\ell(s)-1$.
\end{lem}

\begin{cor}
\label{cor:reducing-conjugators}
If $p$ and $q$ are not parabolic, then there exists a minimal conjugacy $(\lambda,c)$-diagram $(p',q',s'',r'')\sim(p,q,r,s)$ such that 
 if a component $s_0$ of $s'$ is connected to a component of $p'$ or $q'$, then $\d_X((s_0)_-,(s_0)_+)\leq m$.
\end{cor}
\begin{proof}
Suppose that $\ell(s)=1$. Then the result follows from either Lemma \ref{lem:reducing-conjugators} or Lemma \ref{lem:reducing-conjugators2}.
If $\ell(s)>1$ and a component of $s$ is connected to a component 
of $p$, then it has to be the first one by Lemma 
\ref{lem:corners}. By Lemma  \ref{lem:reducing-conjugators} we 
obtain a minimal conjugacy $(\lambda,c)$-diagram $(p',q,s',r')\sim(p,q,r,s)$ such that if the first component $s_1$ of $s'$ is 
connected to a component of $p$, $\d_X((s_1)_-,(s_1)_+)\leq m$ and 
$s'$ and $s$ share a suffix of length $\ell(s)-1$. Now, if $q$ has 
a component connected to $s$, by Lemma \ref{lem:corners}, it has 
to be  the last component of $s$ which is the same as the last 
component of $s'$. We use Lemma \ref{lem:reducing-conjugators2} to 
obtain a  minimal conjugacy $(\lambda,c)$-diagram 
$(p',q',s'',r'')\sim(p',q,r',s')$ such that if the last 
component $s_2$ of $s''$ is connected to a component of $q'$, then 
$\d_X((s_2)_-,(s_2)_+)\leq m$ and $s''$ and $s'$ share a 
suffix of length $\ell(s)-1$.  
The only component of $s''$ that can be connected to $p'$ is the first one 
and is exactly $s_1$.
\end{proof}

The next two lemmas deal with the case when $p$ or $q$ is long.

\begin{lem}\label{lem:longpq}
Suppose that $\ell(p)> \lambda (2K+1+c)$. Then $\ell(q)>1$, $\ell(r)=\ell(s)\leq K$ and no component of $r$ is connected to a component of $s$.
\end{lem}
\begin{proof}
Let $u$ be a vertex of $r$ and $v$ a vertex of $s$.
Suppose that $\d_{X\cup \cH}(u,v)\leq 1$.
Since $p$ is a $(\lambda,c)$-quasi-geodesic, $\d_{X\cup \cH}( p_-, p_+)\geq \ell(p)/\lambda -c > 2K+1$.
Since $\ell(p)>K$, by \eqref{eq:hyprect} one has $\ell(r)=\ell(s)\leq K$. Therefore 
$\d_{X\cup \cH}(u,v)\geq \d_{X\cup \cH} (p_-,p_+)-\d_{X\cup \cH}(u,p_-)-\d_{X\cup \cH}(v,p_+) > 2K+1-2K=1$, which is a contradiction.

In particular $\ell(q)>1$  and no component of $r$ is connected to a component of $s$.
\end{proof}

\begin{lem}\label{lem:rlessthanK}
If $\max\{\ell(p),\ell(q)\}\geq \lambda (2K+1+c)$ there is $(p',q',r',s')\sim(p,q,r,s)$ such that  $\d_X(r'_-,r'_+)=\d_X(s'_-,s'_+)\leq (2K+3) D +2 m$.
\end{lem}
\begin{proof}
By Lemma \ref{lem:longpq} neither $p$ nor $q$ is parabolic.
By Corollary \ref{cor:reducing-conjugators} there is 
$(p',q',r',s')\sim(p,q,r,s)$ such that if a component $s'_1$ 
of $s'$ is connected to a component of $p'$ or of $q'$ 
then $\d_X((s_1')_-,(s_1')_+)\leq m$. By Lemma \ref{lem:longpq} 
no component of $s'$ is connected to a component of $r'$, 
so all the components of $s'$ are either isolated in $(p',q',r',s')$ or are 
connected to $p'$ or $q'$ and have $X$-length less than $m$. 
By Lemma \ref{lem:corners} there are at most 2 
components of $s'$ connected to components of $p'$ or $q'$.

Consider the polygon with sides $p'$, $q'$, $r'$ and the edges of $s'$. This polygon has at most $K+3$ sides and we set 
$I$ to be the set of isolated components of $s'$. Applying Lemma 
\ref{lem:gon} we obtain that the sum of the $X$-lengths of the isolated components of $s'$ is at most $(K+3)D$. By the previous discussion, the sum of the $X$-length of the non-isolated components of $s'$ is at most $2m$. Since $D>1$, $\d_X(s'_-,s'_+)\leq (K+3)D+2m+KD\leq (2K+3)D+2m$.
\end{proof}

The following is immediate from the fact that $(p,q,r,s)$ is minimal.

\begin{lem}\label{lem:longr}
If $\ell(r)=\ell(s)>1$ then no component of $p$ is connected to a component of $q$.
\end{lem}

The next lemma gives an upper bound for $\d_X(p_-,p_+)$ in terms of the $(X \cup \cH)$-length of $p$. This will help us deal with the case when $p$ and $q$ are short.
\begin{lem}\label{lem:plessthanK}
Suppose that  $\ell(r)=\ell(s)>1$ and $p$ is not parabolic. Then 
$$\d_X(p_-,p_+)\leq 4(\ell(p)+4)D+4m.$$
\end{lem}
\begin{proof}
Let $I$ be the set of components of $p$. Then $\d_X(p_-,p_+)\leq \ell(p)+\sum_{t\in I}\d_X(t_-,t_+)$. Our strategy is to bound $\sum_{t\in I}\d_X(t_-,t_+).$
Notice that for all $(p',q',r',s')\sim(p,q,r,s)$, $\ell(p)=\ell(p')$, $\ell(q)=\ell(q')$ and the lengths of the components of $p'$ is the same as the length of the components of $p$. So without loss of generality, we can change $(p,q,r,s)$ by a minimal conjugacy diagram $(p',q',r',s')\sim(p,q,r,s) $ and by Corollary \ref{cor:reducing-conjugators} we can assume that if a component $s_1$ of $s$ is connected to a component of $p$ or of $q$ then $\d_X((s_1)_-,(s_1)_+)\leq m$. 

By Lemma \ref{lem:longr}, no component of $p$ is connected to a component of $q$. If no component of $p$ is connected to a component of $r$ or $s$, then the lemma follows from Corollary \ref{cor:exem}.

{\bf Case 1:}
We consider first the case when $p$ has components connected to $r$ and $s$. In this case, by Lemma \ref{lem:corners} there are only two components of $p$ that are non-isolated in $(p,q,r,s)$: the first edge $p_1$ of $p$ connected to the first edge $r_1$ of $r$ and the last edge $p_2$ of $p$ connected to  first edge $s_1$ of $s$. Since $p$ is not parabolic, $p_1\neq p_2$.   Consider the 6-gon $abcdef$ obtained from $(p,q,r,s)$ by  ``cutting the bottom corners". See Figure \ref{fig:6gon1}.
\begin{figure}
\includegraphics[scale=0.4]{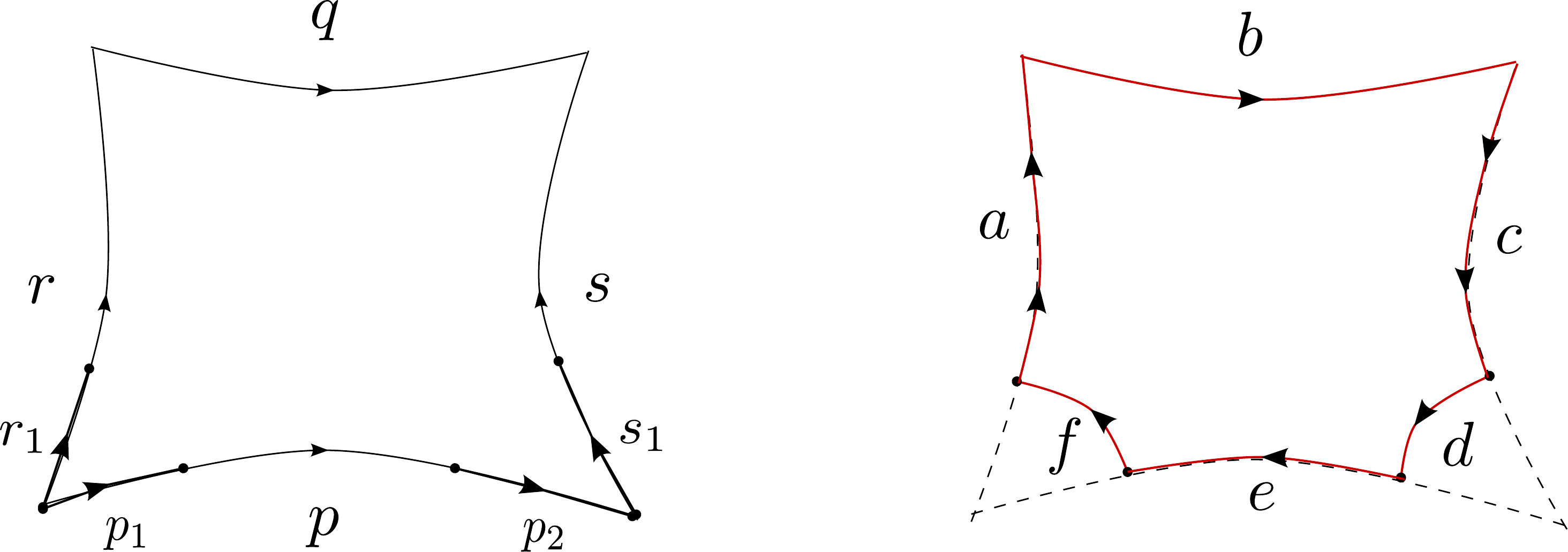}
\caption{Polygon $abcdef$ for Case 1.}
\label{fig:6gon1}
\end{figure} That is,
$a$ is the subpath of $r$ from $(r_1)_+$ to $r_+$, $b=q$,
$c$ is the subpath of $s^{-1}$ from $s_+$ to $(s_1)_+$,
$d$ is an edge from $(s_1)_+$ to $(p_2)_-$ labeled by $\Lab(s_1)^{-1}\Lab(p_2)^{-1}\in \cH$, $e$ is the (possibly empty) subpath of $p^{-1}$ from  
$(p_2)_-$ to $(p_1)_+$ and finally $f$ is the edge from $(p_1)_+$ to $(r_1)_+$ labeled by $\Lab(p_1)^{-1}\Lab(r_1)\in \cH$. Then all the sides of $abcdef$ are $(\lambda,c)$-quasi-geodesics and all the components of $e$ are isolated in $abcdef$ by construction. We note that $f$ and $d$ might be paths of length $0$ (i.e. vertices).

{\bf Claim 1:} $f$ is either an isolated component in $abcdef$ or $\d_X(f_-,f_+)\leq m$. 

To prove the claim, we assume that $f$ is not isolated and $d_X(f_-,f_+)>m$. Suppose that $r_1$ and $p_1$ are $H_\omega$-components. Then $f$ is an $H_\omega$-component, and since $\d_X(f_-,f_+)>m$, it is not an $H_\nu$-component, for any $\nu\neq \omega$. Thus, if $f$ is connected to other  component $t$ of $abcdef$ it must be an $H_\omega$-component.
Since $r$ is geodesic $t$ cannot be a component of $a$.
By Lemma \ref{lem:longr}, $p_1$ is not connected to a component of $q$ and thus $t$ can not be a component of $b$.
By Lemma \ref{lem:corners}, $p_1$ is not connected to a component of $s$ and thus $t$ can not be a component of $c$. Since $p$ has no vertex backtracking, $t$ cannot be a component of $e$. Thus $t=d$. Then $t$ is an $H_\omega$-component, and since $\Lab(s_1)\equiv \Lab(r_1)$, $s_1$ is also an $H_\omega$-component. This implies that $p_2$ is an $H_\omega$-component and is connected to $p_1$, contradicting the fact that $p$ has no vertex backtracking. The claim is proved.

Similarly we get:

{\bf Claim 2:}  $d$ is either an  isolated component in $abcdef$ or $\d_X(d_-,d_+)\leq m$.

We consider the polygon with sides $a,b,c,d,f$ and the edges of $e$. 
This is a polygon with at most $\ell(p)+3$ sides and all its sides are $(\lambda,c)$-quasi-geodesics. Recall that $I$ is the set of components of $p$. Let $I'$ be the set of isolated components of $e$ in $abcdef$ together with $d$ and $f$ in case they are isolated.
Then by Lemma \ref{lem:gon} the total sum of the $X$-lengths of the  components of $I'$ is $(\ell(p)+3)D$. 
By Claim 1 and Claim 2, $\max\{\d_X(d_-,d_+), \d_X(f_-,f_+)\}\leq \max\{m, (\ell(p)+3)D\}$. By hypothesis, $\d_X((r_1)_-,(r_1)_+)\leq m$,  and thus we obtain that $\d_X((p_1)_-,(p_1)_+)\leq \max\{m, (\ell(p)+3)D\}+m$. And similarly for $\d_X((p_2)_-,(p_2)_+)$. Hence
\begin{align*}
\sum_{t\in I}\d_X(t_-,t_+) &\leq \sum_{t\in I'}\d_X(t_-,t_+)+\d_X((p_1)_-,(p_1)_+)+\d_X((p_2)_-,(p_2)_+)\\
& \leq (\ell(p)+3)D +2 \max\{m, (\ell(p)+3)D\}+2m\\&\leq 3(\ell(p)+3)D+4m.
\end{align*}
The lemma now follows easily.

{\bf Case 2:} We consider now the case when $p$ has no component connected to $s$.  In this case, by Lemma \ref{lem:corners}, the only non-isolated component is the first edge $p_1$ of $p$ connected to the first edge $r_1$ of $r$. 

Let $s_1$ be the first edge of $s$. If $r_1$ is an $H_\omega$-component, so is $s_1$. By hypothesis, $s_1$ is not connected to a component of $p$. 
Since $\ell(s)>1$, if follows from Lemma \ref{lem:corners} that $s_1$ is not connected to a component of $q$. 
Since $s$ is geodesic, $s_1$ is not connected to a component of $s$.
Suppose $s_1$ is  connected to a component $r_2$ of $r$. 
Then there is a vertex $v=(r_2)_+$ of $r$ (the furthest away from $r_-$) such that $\d_{X\cup \cH}(u,v)\leq 1$, where $u=p_+=(s_1)_-$, and hence, by Lemma \ref{lem:rconnect}, $u$ is a vertex of $r_1$ and thus $r_2=r_1$. 
If $r_1$ and $s_1$ are connected, we obtain a contradiction to the non-parabolicity of $p$. Therefore $s_1$ is isolated.

We  consider the 6-gon $abcdef$ obtained from $(p,q,r,s)$ by ``cutting the bottom left corner". See Figure \ref{fig:6gon2}.
\begin{figure}
\includegraphics[scale=0.4]{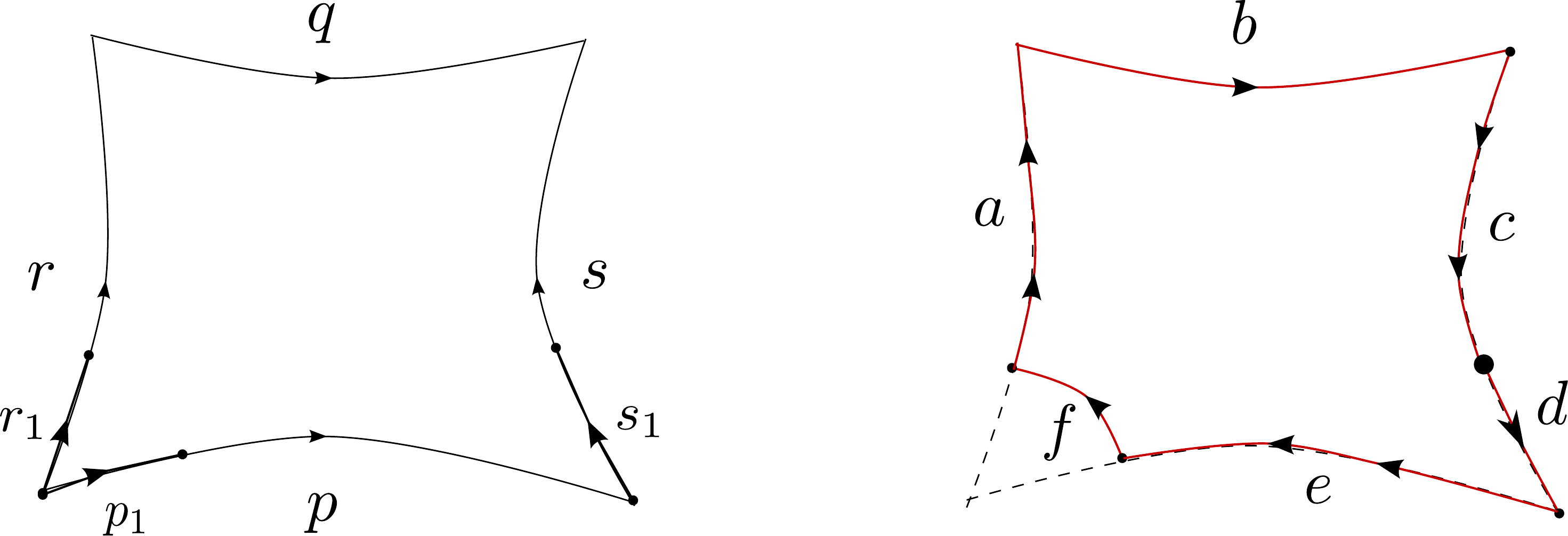}
\caption{Polygon $abcdef$ for Case 2.}
\label{fig:6gon2}
\end{figure}

 That is,
$a$ is the subpath of $r$ from $(r_1)_+$ to $r_+$, $b=q$,
$c$ is the subpath of $s^{-1}$ from $s_+$ to $(s_1)_+$,
$d=s_1^{-1}$, $e$ is the  subpath of $p^{-1}$ from  $p_+$ to $(p_1)_-$  and finally $f$ is the edge from $(p_1)_+$ to $(r_1)_+$ labelled by $\Lab(p_1)^{-1}\Lab(r_1)\in \cH$. Then all the sides of $abcdef$ are $(\lambda,c)$-quasi-geodesics and all the components of $e$ are isolated in $abcdef$ by construction, and $d$ is an isolated component in $abcdef$. 

Note that we allow the degenerate case of $f$ being a vertex. Arguing in a similar way as in Claim 1 we obtain:

{\bf Claim 3:}  $f$ is either an isolated component in $abcdef$ or $\d_X(f_-,f_+)\leq m$. 

We consider the polygon with sides $a,b,c,d,f$ and the edges of $e$. This is a polygon with at most $\ell(p)+4$ sides and all its sides are $(\lambda,c)$-quasi-geodesics. Let $I'$ be the set of isolated components of $e$ in $abcdef$ together with  $f$, if it is isolated in $abcdef$.
Then by Lemma \ref{lem:gon} the total sum of the $X$-length of the  components of $I'$ is $(\ell(p)+4)D$. Thus from Claim 3 we get that $\d_X(f_-,f_+)\leq \max\{m, (\ell(p)+4)D\}$.  Since $s_1=d$ is isolated, $\d_X((r_1)_-,(r_1)_+)=\d_X((s_1)_-,(s_1)_+)\leq (\ell(p)+4)D$. Combining this with $\d_X(f_-,f_+)\leq \max\{m, (\ell(p)+4)D\}$, we obtain that $\d_X((p_1)_-,(p_1)_+)\leq \max\{m, (\ell(p)+4)D\}+ (\ell(p)+4)D$. 

Hence
\begin{align*}
\sum_{t\in I}\d_X(t_-,t_+) &\leq \sum_{t\in I'}\d_X(t_-,t_+)+\d_X((p_1)_-,(p_1)_+)\\
& \leq (\ell(p)+4)D +\max\{m, (\ell(p)+4)D\}+ (\ell(p)+4)D\\&\leq 3(\ell(p)+4)D+m.
\end{align*}
This case now easily follows.

{\bf Case 3:} $p$ has no component connected to $r$.
In this case we can argue as in Case 2. 

Recall that the case when $p$ has no component connected to $r$ or $s$ has been already considered and hence we have proved the lemma in all possible cases.
\end{proof}

Collecting all the results up to now produces the following general statement about conjugacy diagrams, independent of the Cayley graphs of the parabolic subgroups.

\begin{thm}\label{thm:bmgcq}
Let $G$ be a finitely generated group, hyperbolic relative to a collection of subgroups $\{H_\omega\}_{\omega \in \Omega}$.
Let $X$ be a finite generating set of $G$, $\cH=\cup_{\omega \in \Omega} (H_\omega - \{1\})$, $\lambda \geq 1$ and $c\geq 0$. 

There exists a constant $k \geq 0$ such that for every minimal conjugacy $(\lambda,c)$-diagram without vertex backtracking $(p,q,r,s)$, where $p$ and $q$ are not parabolic, one can find $(p',q',r',s')\sim(p,q,r,s)$ such that

$$\min\{ \max\{\d_X(p'_-,p'_+),\d_X(q'_-,q'_+)\}, \d_X(r'_-,r'_+) \}\leq k.$$ 
\end{thm}

\begin{proof}
We set $k= \max\{(2K+3)D+2m, 4(\lambda (2K+1+c)+4)D+4m\}$.

If $\ell(r)=\ell(s)=1$, by Corollary \ref{cor:reducing-conjugators}, there exists a conjugacy diagram $(p',q',s',r')\sim(p,q,r,s)$ such that either $\d_X(r'_-,r'_+)\leq m$ or $r'$ is isolated in $(p',q',r',s')$. In the latter case, we use Corollary \ref{cor:exem} to conclude that $\d_X(r'_-,r'_+)\leq 4D.$ 

Suppose that $\max\{\ell(p),\ell(q)\}\geq \lambda(2K+1+c)$. Then
by Lemma \ref{lem:rlessthanK} there is $(p',q',r',s')\sim(p,q,r,s)$ such that $\d_X(r'_-,r'_+)\leq (2K+3)D+2m\leq k$.

Now suppose $\max\{\ell(p),\ell(q)\} < \lambda(2K+1+c)$ and $\ell(r)>1$. Since $p$ and $q$ are not parabolic,  Lemma \ref{lem:plessthanK} implies $\max\{\d_X(p_-,p_+),\d_X(q_-,q_+)\}\leq 4(\lambda (2K+1+c)+4)D+4m\leq k$.
\end{proof}

\begin{lem}\label{lem:conjPara}
Suppose that $p$ is parabolic (and hence $p$ is a component) and is connected to a component of $q$. Then $\Lab(p),\Lab(q),\Lab(r),\Lab(s)$ are letters of some $H_\omega$.
\end{lem}
\begin{proof}
If $p$ is an $H_\omega$-component connected to an $H_\omega$-component of $q$, after a cyclic permutation of $q$, we can assume that this component of $q$ is the first edge of $q$ and $\Lab(r)\equiv \Lab(s)\in H_\omega$.
Since no cyclic  permutation of $\Lab(q)$ is allowed to have vertex backtracking due to the fact that $(p,q,r,s)$ is minimal, we can conclude that $q$ is of length one and hence parabolic.
\end{proof}

We now include the property \bcd{} of the parabolic subgroups in the discussion.

\begin{lem}\label{lem:Apara}
Assume that for every $\omega\in \Omega$ and every $h\in H_\omega$, $|h|_X=|h|_{X\cap H_\omega}$ and $(H_\omega, X\cap H_\omega)$ has $A$-\bcd{} for some $A \geq 0$.
Suppose that there is a cyclic geodesic word over $X\cap H_\omega$ representing  $h\in H_\omega$. Then for any $h',g\in H_\omega$ such that $h=gh'g^{-1}$, we have $|h|_{X}\leq 2A+|h'|_{X}$.
\end{lem}
\begin{proof}
If $|h|_{X\cap H_\omega}=|h|_X\leq A$ there is nothing to prove. So assume $|h|_{X\cap H_\omega}\geq A$.
Since $h$ is conjugate to $h'$, there exists a cyclic geodesic word $U$ over $X\cap H_\omega$ such that $\ell(U)\leq |h'|_{X\cap H_\omega}$ and $U$ represents some conjugate of $h$. 
Let $V$ be a cyclic geodesic over $X\cap H_\omega$ representing $h$. 
Then there is an $A$-bounded minimal $(1,0)$-diagram $(p',q',r',s')$ in $(H_\omega, H_\omega\cap X)$ with $p'$ and $q'$ labeling cyclic permutations of $V$ and $U$, respectively. 
Since $|h|_{X\cap H_\omega} \geq A$, it follows that $\ell(s')=\ell(r')\leq A$. 
Then $$\ell(p')=\ell(V)= |h|_{X\cap H_\omega}\leq \ell(s')+\ell(r')+\ell(q')\leq 2A+|h'|_{X\cup H_\omega }=2A+|h'|_{X}.$$
\end{proof}

\begin{lem}\label{lem:pnonconnected}
Assume that for every $\omega\in \Omega$ and every $h\in H_\omega$, $|h|_X=|h|_{X\cap H_\omega}$.
Suppose that there is a cyclic geodesic word $U$ over $X\cap H_\omega$ representing  $\Lab(p)$ and $p$ is not connected to a component of $q$.
\begin{itemize}
\item[{\rm (a)}] If $(H_\omega, X\cap H_\omega)$ has $A$-\bcd{}, then $\d_X(p_-,p_+)\leq  2A+4D+m$.
\item[{\rm (b)}] If $(H_\omega, X\cap H_\omega)$ has $A$-\nsc{}, then either $\d_X(p_-,p_+)\leq 4D+m$ or there is a cyclic permutation $U'$ of $U$ and words $V$ and $C$ over $X\cap H_\omega$, $\ell(V)<\ell(U)$, $\ell(C)\leq A$ such that $CU'C^{-1}=_G V$.
\end{itemize}
\end{lem}
\begin{proof}
If $p$ is not connected to a component of $r$ or $s$, then $p$ is isolated and by Corollary \ref{cor:exem}, 
$\d_{X}(p_-,p_+)\leq 4D$.

 If $p$ is connected to a component of $r$, by Lemma \ref{lem:corners}, it has to be the first edge of $r$,
and hence $p$ is also connected to the first edge of $s$. Suppose that $r_1$ is the first edge of $r$, 
$s_1$ is the first edge of $s$ and let  $p'$ be a path of length one from  $(r_1)_{+}$ to $(s_1)_+$.
Let  $r'$ be the subpath of $r$ from $(r_1)_+$ to $r_+$ and $s'$ the subpath of $s$ from $(s_1)_+$ to $s_+$. 

We claim that  $\d_{X}(p'_-,p'_+)\leq 4D+m$. Suppose the converse, i.e. $\d_{X}(p'_-,p'_+)>4D+m$. Then, by Corollary \ref{cor:exem}, $p'$ can't be isolated in the 4-gon $p's'q^{-1}r'^{-1}$.
Also, $p'$ can't be connected to a component $q_1$ of $q$, since in this case, $\d_X(p'_-,p'_+)>m$ implies that $q_1$ is also an $H_\omega$-component and has to be connected to $p$. A similar argument shows that $p'$ can't be connected to a component of $r'$ or $s'$. This leads to a contradiction, and hence the claim is proved.

(a) Since $\d_{X}(p'_-,p'_+)\leq 4D+m$, by Lemma \ref{lem:Apara} we get $\d_X(p_-,p_+)\leq   2A+4D+m$.

(b) If  $\d_{X}(p_-,p_+)>4D+m$ then, since $\d_{X}(p'_-,p'_+)\leq 4D+m$, the $A$-\nsc{} property gives the existence of $U'$, $V$ and $C$ with the desired properties.
\end{proof}

The following statement implies Theorem \ref{thm:rh_conjhyp} in the Introduction. 
\begin{thm}\label{thm:rh_conjhyp_ext}
Let $G$ be a finitely generated group, hyperbolic relative to a family of subgroups $\{H_\omega\}_{\omega\in \Omega}$. Let $Y$ be a finite generating set and $\cH=\cup_{\omega\in \Omega} (H_\omega -\{1\})$.
There exists a finite subset $\cH'$ of $\cH$ such that the following hold.
For every finite set $X$ satisfying
$$Y\cup \cH'\subseteq X \subseteq Y \cup \cH$$
there is a finite subset $\Phi$ of non-geodesic words over $X$ such that if $W$ is a word over $X$ with no parabolic shortenings and without subwords in $\Phi$, then $W$ represents the trivial element if and only if $W$ is empty. Moreover,
\begin{enumerate}
\item[{\rm (a)}] If $(H_\omega, X\cap H_\omega)$ has \bcd{} for all $\omega\in \Omega$, then there is a constant $B$ 
such that for every pair of words $U$ and $V$ over $X$ representing conjugate elements in $G$ and such that none of their cyclic shifts have parabolic shortenings or contain subwords in $\Phi$, there is an element $g\in G$ and cyclic shifts $U'$ of $U$ and $V'$ of $V$ such that $gU'g^{-1}=_GV'$ and
$$\min\{\max \{ \ell(U),\ell(V)\}, |g|_X\}\leq B.$$

\item[{\rm (b)}] If  $(H_\omega, X\cap H_\omega)$ have \nsc{} for all $\omega\in \Omega$, then $(G,X)$ has \nsc.

\end{enumerate}
 
\end{thm}
 
\begin{proof}
 Let $\lambda, c, \cH'$ be  the constants and set provided by the Generating Set Lemma (Lemma \ref{lem:goodgenset}).
Let $X$ be any  generating set of $G$ satisfying $Y\cup \cH'\subseteq X \subseteq Y \cup \cH$, and let $\Phi=\Phi(X)$ be the set of non-geodesic words produced by the Generating Set Lemma. 

For any  word $W$ over $X$ such that none of its cyclic shifts have parabolic shortenings or contain a subword in $\Phi$, the Generating Set Lemma implies that $\wh{W}$ is a cyclic $(\lambda,c)$-quasi-geodesic without vertex backtracking. 
Also, we have that $|h|_X=|h|_{X\cap H_\omega}$ for all $h\in H_\omega$  and $\omega\in \Omega$.

(a) Assume that  $(H_\omega, H_\omega\cap X)$ have $A$-\bcd{} for all $\omega\in \Omega$, where $A \geq 0$.

Let $U,V$ be two words over $X$ representing conjugate elements such that no cyclic shifts have parabolic shortenings or contain subwords in $\Phi$.
Let $(p,q,r,s)$ be a minimal conjugacy $(\lambda,c)$-diagram in $\ga(G,X\cup \cH)$, where $\Lab(p)$ and $\Lab(q)$ are some cyclic permutations of $\wh{U}$ and $\wh{V}$, respectively.  If $U$ and $V$ do not represent an element of a parabolic subgroup, we obtain by Theorem \ref{thm:bmgcq}
that there exist a cyclic permutation $U'$ of $U$ and a cyclic permutation $V'$ of $V$ and $g\in G$ such that $g U' g^{-1}= V'$ and $\min\{\max \{\ell(U'),\ell(V')\},|g|_{X}\}\leq k$, where $k$ is the constant provided by Theorem \ref{thm:bmgcq}.

So we only need to consider the case when $U$ or $V$ labels elements of the parabolic subgroups.

Suppose that $p$ is parabolic.
If $p$ is connected to a component of $q$, then by Lemma \ref{lem:conjPara}, $U$ and $V$ are words in  $X\cap H_\omega$ for some $\omega\in \Omega$ and conjugate in $H_\omega$. 
Then there exists a cyclic permutation $U'$ of $U$ and a cyclic permutation of $V'$ of $V$ and $h\in H_\omega$ such that $h U' h^{-1}= V'$ and $\min\{\max \{\ell(U'),\ell(V')\},|h|_{X}\}\leq A$.

If $p$ is not connected to a component of $q$, then by Lemma \ref{lem:pnonconnected}(a), $\d_X(p_-,p_+)\leq 2A+4D+m$.
In this case, $q$ is not connected to a component of $p$, and there are two cases. The first case is when $q$ is also parabolic. Here again, using Lemma \ref{lem:pnonconnected} we obtain that $\d_X(q_-,q_+)\leq 2A+4D+m$.
The second case is when $q$ is not parabolic. In this case,
by Lemma \ref{lem:longpq}, $\ell(q)\leq \lambda(2K+1+c)$.
Then by Lemma \ref{lem:plessthanK}, $\d_{X}(q_-,q_+)\leq 4 (\lambda(2K+1+c)+4) D+4m$.

Therefore $\max\{\ell(U),\ell(V)\}\leq B\coloneq \max\{ 4 (\lambda(2K+1+c)+4) D+4m, 2A+4D+m\}$ and then $(G,X)$ has $B$-\bcd{}.

(b) Assume that $(H_\omega,H_\omega\cap X)$ has $A$-\nsc{} for all $\omega\in \Omega$, where $A \geq 0$.

We are going to show that $(G,X)$ has $B$-\nsc{} with $$B=A+5D+4m+k+4\lambda(2K+1+c).$$

Let $U,V$ be two cyclic geodesic words over $X$ representing conjugate elements. Suppose that $\ell(V)<\ell(U)$ and $\ell(U)\geq B$.

Let $(p,q,r,s)$ be a minimal conjugacy $(\lambda,c)$-diagram in $\ga(G,X\cup \cH)$, where $\Lab(p)$ and $\Lab(q)$ are some cyclic permutations of $\wh{U}$ and $\wh{V}$, respectively.  If $U$ and $V$ do not represent an element of a parabolic subgroup, we obtain by Theorem \ref{thm:bmgcq}
that there exist a cyclic permutation $U'$ of $U$ and a cyclic permutation $V'$ of $V$, and $g\in G$, such that $g U' g^{-1}= V'$ and $\min\{\max \{\ell(U'),\ell(V')\},|g|_{X}\}\leq k$, where $k$ is the constant provided by Theorem \ref{thm:bmgcq}.

So we only need to consider the cases when $p$ or $q$ is parabolic.
 
Assume first that $p$ is parabolic.

If $p$ is connected to a component of $q$, then by Lemma \ref{lem:conjPara}, $q$ is parabolic and $U$ and $V$ are conjugate in $H_\omega$. The result in this case follows from the $A$-\nsc{} of $(H_\omega, X\cap H_\omega)$ and the fact that $|h|_X=|h|_{X\cap H_\omega}$ for $h\in H_\omega$.
If $p$ is parabolic and is not connected to a component of $q$, then the result follows from Lemma \ref{lem:pnonconnected}(b).

Assume now that $q$ is parabolic.

If $q$ is connected to a component of $p$, the result follows by arguing as in the case $p$ parabolic connected to a component of $q$.
If $q$ is not connected to a component of $p$, we consider two cases separately. In both we will conclude that  $\d_{X}(r_-,r_+)\leq B$.

Suppose first that $\ell(r)=1$.  If $\Lab(r)\in X$  then  $\d_{X}(r_-,r_+)=1\leq B$. If $r$ is just a component, it can't be connected to $q$, since this will imply that $p$ vertex backtracks or $p$ is a component, and none of the two situations holds by hypothesis. Then, by Lemma \ref{lem:reducing-conjugators}, we can assume that either $\d_{X}(r_-,r_+)\leq m\leq B$ or $r$ is isolated. If $r$ is isolated, then by Corollary \ref{cor:exem}, $\d_X(r_-,r_+)\leq 5D\leq B$. 

Suppose now that $\ell(r)>1$. Now by Lemma \ref{lem:longr}, $B\leq (4\ell(p)+4)D+4m$. Since $B> (4\lambda(2K+1+c)+4)D+4m$, $\ell(p)>\lambda(2K+1+c)$ and then by Lemma \ref{lem:longpq}, $\d_X(r_-,r_+)\leq (2K+3)D+2m\leq B$.
\end{proof}

\begin{rem}\label{rem:complexity} 
We observe here that Theorem \ref{thm:rh_conjhyp_ext} provides a cubic-time algorithm for solving the conjugacy problem in groups hyperbolic relative to abelian subgroups. This algorithm is in the same spirit as  \cite[Algorithm 7.A]{BridsonHowie}. 

The conjugacy problem in groups hyperbolic relative to parabolic subgroups with solvable conjugacy problem has been shown to be solvable by Bumagin in \cite{Bumagin}.  Recently and independently, bounds of the complexity have been obtained by Bumagin \cite{Bumagin14}. Rebbechi (\cite{Rebbechi}) or our Corollary \ref{cor:toral_languages} proves that groups hyperbolic with respect to abelian subgroups are biautomatic, so this also implies the solvability of the conjugacy problem in the main class of groups considered in this paper.

Suppose that $G$ is a finitely generated group, hyperbolic relative to a family $\{H_\omega\}_{\omega\in\Omega}$ of abelian subgroups. Using Theorem \ref{thm:rh_conjhyp_ext} we can find a finite generating set $X$ and a finite set $\Phi$ of non-geodesic words over $X$. Let $B$ be the constant provided by Theorem \ref{thm:rh_conjhyp_ext}(a). We assume that the sets $X$ and $\Phi$ have been provided to us, and do not include the complexity of obtaining these sets in our discussion below. This approach is often employed when considering algorithmic problems in hyperbolic groups, where a Dehn presentation is assumed, and the complexity of obtaining such a presentation from an arbitrary one is incorporated into a constant (produced by a computable function that took into account various parameters of the presentation).  

Given a word $U$ over $X\cap H_\omega$, since $H_\omega$ is abelian, one can find in $O(\ell(U))$ steps a geodesic word $U'$ over $X\cap H_\omega$ such that $U'=_G U$. Therefore, given a word $W$ over $X$, one can produce in a linear number of steps a word $W'$, $\ell(W')\leq \ell(W)$, that cyclically has no parabolic shortenings and represents a conjugate of $W$.
Using the same argument as that in \cite[Algorithm DA]{BridsonHowie}, from a word $W$ that cyclically has no parabolic shortenings, when replacing subwords in $\Phi$ by geodesic words representing the same elements, we can produce a word $W'$ that cyclically has no subwords in $\Phi$ and such that $W'$ represents a conjugate of $W$ with $\ell(W')\leq \ell(W)$. Moreover, if $W$ represents the trivial element, $W'$ is empty.
We call this algorithm ``the reducing algorithm."

Now given two words $U$ and $V$, we use the reducing algorithm to produce in a linear number of steps words $U'$ and $V'$ that cyclically have no parabolic shortenings and do not contain subwords in $\Phi$. By Theorem \ref{thm:rh_conjhyp_ext}(a) either $\max\{\ell(U'),\ell(V')\}\leq B$, and we decide if $U'$ and $V'$ are conjugate in a constant number of steps (depending on $B$), or $\max\{\ell(U'),\ell(V')\}> B$, and we know that if $U$ and $V$ are conjugate, then there are cyclic shifts $U''$ of $U'$ and $V''$ of $V'$ and a conjugator of length at most $B$. Then we can decide if $U$ and $V$ are conjugate by applying the reducing algorithm to all words of the form $CU''C^{-1}(V'')^{-1}$, where $C$ is a word of length less than $B$ and $U''$ and $V''$ are cyclic shifts of $U'$ and $V'$, respectively. The number of such words is $O(\ell(U')\cdot \ell(V'))$. Thus the complexity of the algorithm is $O((\ell(U)+\ell(V))\ell(U)\ell(V))$.

We remark that \cite{BridsonHowie} argues that in the hyperbolic case it is only necessary to check whether cyclic shifts of $U'$ are conjugate to $V'$ by a word of a bounded length. This is the reason why the algorithm there is quadratic. A similar situation might occur in our context. An algorithm with even lower complexity  for the conjugacy problem for hyperbolic groups was obtained Epstein and Holt in \cite{EpsteinHolt}. 

Also notice that the algorithm provides a conjugator for $U'$ and $V'$ in case $U$ and $V$ are conjugate. By keeping track of the cyclic reductions necessary to obtain $U'$ from $U$ and $V'$ from $V$ one can produce a conjugator for $U$ and $V$. We remark that an upper bound for the length of the conjugator was obtained by O'Connor in \cite{Zoe}.
\end{rem}


\section{Proof of Theorem \ref{thm:summary} and application to virtually abelian parabolics}\label{sec:lang}

In this section we prove all the remaining main results stated in the Introduction. We start with the proof of Theorem \ref{thm:summary}, which uses a combination of Theorems \ref{thm:fftpgenset}, \ref{thm:rh_conjhyp_ext},   \ref{thm:biaut} and \ref{thm:shortlex}. 

\begin{proof}[Proof of Theorem \ref{thm:summary}]
Let (P) be a property of (ordered) generating sets as in Theorem \ref{thm:summary}. Assume that $G$ is hyperbolic with respect to a family $\{H_\omega\}_{\omega\in \Omega}$ of groups that are (P)-completable.
Let $Y$ be any finite (ordered) generating set of $G$.

Let $\cH'_1$ be the subset of $\cH$ provided by Theorem \ref{thm:fftpgenset}, $\cH'_2$ the subset of $\cH$ provided by Theorem \ref{thm:biaut} and
$\cH'_3$ the subset set of $\cH$ provided by Theorem \ref{thm:rh_conjhyp_ext}. Set $\cH'=\cH'_1\cup \cH'_2\cup \cH'_3$.

Since  $\{H_\omega\}_{\omega\in \Omega}$ are (P)-completable, there is  a 
finite (ordered) generating set $X$ of $G$ satisfying 
$Y\cup \cH' \subseteq X \subseteq Y\cup \cH$ and such that for all $\omega \in \Omega$,  
$(H_\omega,H_\omega \cap X)$ has (P). 
Now the result follows from Theorems \ref{thm:fftpgenset}, \ref{thm:biaut}, \ref{thm:shortlex} and \ref{thm:rh_conjhyp_ext}.
\end{proof}

We now turn to proving Corollaries \ref{cor:toral_languages} and \ref{cor:vapar_languages}.

It will be convenient to set the following notation for the set of cyclic geodesics $$\cycgeo(G,X)\coloneq \{W\in X^* \mid \text{for  all cyclic shifts $W'$ of } W, \, W'\in \geol(G,X)\}. $$
We remark that  $\CG(G,X) \subseteq \cycgeo(G,X) \subseteq \geol(G,X)$.

\begin{prop}\label{prop:VA}
Let $G$ be a finitely generated virtually abelian group. Any finite 
generating set of $G$ 
is contained in a 
finite generating set $Z$ such that 
$(G,Z)$ has  \nsc{} and \fftp{} and $(G,Z)$ admits a geodesic prefix-closed biautomatic structure.
\end{prop}

\begin{proof}
We will refer a number of times to the proof of \cite[Proposition 3.3]{CHHR}, and in particular, the generating set we work with here is the one used in that proof, as well as in \cite[Prop.~6.3]{HHR}. We also need to introduce some terminology. For a set $A$ and $x_1,\dots,x_n\in A$ we say that $V\equiv x_1\cdots x_n$ is a {\it piecewise subword} of a word $W$ over $A$ if $W$ belongs to the set  $x_1A^*x_2A^*\cdots Ax_nA^*$.
 
Let $Z_0$ be a generating set for $G$. Since $G$ is residually finite, there is  $N \lhd G$ torsion-free abelian of 
finite index in $G$  such that $Z_0\cap N=\emptyset$. Enlarge $Z_0$ to $Z_1$, if necessary, to assume that $\cup_{z\in Z_1} z N=G$. 

We build the generating set $Z$ for $G$ as follows.
Let $Y := (Z_1-N)^{\pm 1}$.
Note that $Z_0\subseteq Y$.
Let $X'$ be the set of all $x \in N$ such that
$x =_G W \neq_G 1$ for some 
$W \in Y^*$ with $\ell(W) \le 4$.
Let $X''$ be the closure of the set $(Z_0\cap N)\cup X'$ in $G$ under
inversion and conjugation.
We now identify $N$ with a lattice in $\mathbb{R}\otimes N$ as in \cite[Theorem 4.2.1]{ECHLPT},
and we take the convex hull $\mathcal{K}$ in $\mathbb{R}\otimes N$ of the union of the images of $X''$. 
It follows form \cite[Theorem 4.2.1]{ECHLPT} that there is some $n\in \mathbb{N}$, $n>0$,
such that $X=N\cap n\mathcal{K}$ is a generating set of $N$ that admits
a geodesic prefix-closed biautomatic structure $\cL$ invariant under the $G/N$-action.

Let 
$Z := X \cup Y$.  
By \cite[Corollary 4.2.4]{ECHLPT}, $\cL \cup \cL Y$ is a geodesic prefix-closed biautomatic structure
for $(G,Z)$.

Moreover
\vspace{-.1cm}
\begin{enumerate}
\item[(i)] $X \subset N$, $Y \subset G - N$,
\item[(ii)] both $X$ and $Y$ are closed under inversion,
\item[(iii)] $X$ is closed under conjugation by elements of $G$, 
\item[(iv)] $Y$ contains at least one representative of each nontrivial coset of $N$
in $G$, and 
\item[(v)] if $W =_G xy$ with $W \in Y^*$, $\ell(W) \le 3$, 
$x \in N$ and $y \in Y \cup \{\epsilon\}$, then $x \in X$.
\end{enumerate}
\vspace{-.1cm}
Write the finite set $X$ as $X=\{x_1,...,x_m\}$.
For each $x \in X$ and $y \in Y$, let $x^y$ denote
the generator in $X$ that represents the group element $y^{-1}xy$.
Similarly, if $v=x_{i_1} \cdots x_{i_k} \in X^*$, 
the symbol $v^y$ denotes the word $x_{i_1}^y \cdots x_{i_k}^y$.
An immediate consequence of  properties (i)-(v) above is that
the language $L:=\geol(G,Z)$ of geodesics of $G$ over 
$Z$ satisfies the property that
$L \subseteq X^{\star} \cup X^{\star} Y  X^{\star} \cup X^{\star} Y
X^{\star}Y  X^{\star}.$

In order to see that $(G,Z)$ has FFTP, observe that $Z$ is in fact one of the generating sets produced by Neumann and Shapiro in \cite[Proposition 4.4]{NeumannShapiro} 
to show that $(G,Z)$ has FFTP.

We need to show that $(G,Z)$ also has \nsc{}.

Let $\wdl:=\CG(G,Z)$ be the set of conjugacy geodesics
of $G$ over $Z$. Then $\wdl:=\geocl(G,Z) \subseteq L$,
and so $\wdl$ can be partitioned as the union of the subsets
$\wdl_0:=\wdl \cap X^*$, $\wdl_1 := \wdl \cap X^*YX^*$, and
$\wdl_2:=\wdl \cap X^*YX^*YX^*$. The set $C:=\cycgeo(G,Z)$ can also be partitioned as the union of the subsets $C_0=C \cap X^*$, $C_1 := C\cap X^*YX^*$, and
$C_2:=C \cap X^*YX^*YX^*$.

It was shown in the proof of \cite[Proposition 3.3]{CHHR} that $X^*\cap L=\wdl_0$, so $C_0=\wdl_0$ since $L_0 \subseteq C_0 \subseteq X^*\cap L=\wdl_0$. Thus if $U\in X^*$ is a cyclic geodesic, there is no word strictly shorter than $U$ that is conjugate to it, and the \nsc{} property  is vacuous in this case.

We now turn to $C_1 - \wdl_1$. One can further partition the set $\wdl_1 = \cup_{r \in Y} \wdl_{1,r}$, where $\wdl_{1,r}:=\{v_1rv_2 \in \wdl \mid v_1,v_2 \in X^*\}$ for each $r$ in $Y$. It was shown in the proof of \cite[Proposition 3.3]{CHHR} that the set 
$\wdl_{1,r}$ is exactly the set of all words that do not contain a piecewise subword lying in a particular finite set, denoted there by $\wt{S}_{1,r}' \subset X^*rX^*$. The set $\wt{S}_{1,r}'$ is by construction closed under $Y-$shuffles and shuffles, so in particular under cyclic permutations. (An operation on words over $Z$ given by replacement $ayxb \rightarrow ax^{y^{-1}}yb$ with $a,b \in Z^*$, $x \in X$, and $y \in Y$
is called a {\em $Y$-shuffle}. An operation on words over $X$ given by a replacement $ux_ix_jv \rightarrow ux_jx_iv$ is a {\em shuffle}.) 
The set $\wt{S}_{1,r}'$ can thus be seen as the set of minimal 
non-conjugacy geodesic words in $X^*rX^*$, and so for each word 
$W$ in $\wt{S}_{1,r}'$ there is a $g_W \in G$ such that $|g_W^{-1}Wg_W|_Z<|W|_Z$. Take $k_r:=\max\{|g_w|_Z \ | \ w \in \wt{S}_{1,r}'\}$ and $k_1:=\max\{k_r \ | \ r \in Y\}$.

Let $V \in C_1 - \wdl_1$. Then it contains a piecewise subword in $S:=\cup_{r\in Y} \wt{S}_{1,r}'$. So $V\equiv V_0a_1V_1a_2V_2\cdots a_nV_n$, where $a_1a_2\cdots a_n\in S$, $V_0,\dots, V_n\in X^*$. After $Y$-shuffles and shuffles we can obtain a word $V'$ such that $V'=_GV$ and
$V'\equiv U W$,  where $U\in X^*$, $W\in S$, and $\ell(U)+\ell(W)=\ell(V')=\ell(V)$. 
Take the element $g_W$ and notice that $$g_W^{-1}Vg_W=_Gg_W^{-1}V'g_W=_G (g_W^{-1}Ug_W)(g_W^{-1} Wg_W),$$
by construction $|g_W^{-1} Wg_W|_Z<|W|_Z\leq \ell(W)$, and by
(iii) $|(g_W^{-1}Ug_W)|_Z\leq |U|_Z\leq \ell(U)$.
Thus $|g_W^{-1}Vg_W|_Z< \ell(U)+\ell(W)=\ell(V)=|V|_Z$.

A similar argument works for $C_2 - \wdl_2$. There is a constant, call it $k_2$, such that for every $V \in C_2 - \wdl_2$ there is a $g \in G$, $|g|<k_2$, for which   $V$ satisfies $|g^{-1}V g|_Z < |V|_Z$.

Thus $(G,Z)$ has $k$-\nsc{}, where $k=\max \{k_1, k_2\}$.
\end{proof}

We do not know if virtually abelian groups have  generating sets that satisfy both \fftp{} and \bcd{}. We remark here that Derek Holt has produced an example of a virtually abelian group that does not have \bcd{} with the generating set used in the proof of Proposition \ref{prop:VA}.

\begin{proof}[Proof of Corollary \ref{cor:toral_languages}.]
The result follows from Theorem \ref{thm:summary} and the fact that for any finite generating set $Z$ of an abelian group $A$, $(A,Z)$ has \fftp{} and \bcd{}, and for any order of $Z$, $\mathsf{ShortLex}(A,Z)$ is a biautomatic structure  (by \cite[Theorem 4.3.1]{ECHLPT} and \cite[Definition 2.5.4]{ECHLPT}). The regularity of the languages is a consequence of Proposition \ref{prop:CHHR38}, Proposition \ref{prop:NS} and the closure of regular languages under intersection.
\end{proof}

\begin{proof}[Proof of Corollary \ref{cor:vapar_languages}.]
The result follows from Theorem \ref{thm:summary} and Proposition \ref{prop:VA}.
The regularity of the languages is a consequence of Proposition \ref{prop:CHHR38}, Proposition \ref{prop:NS} and the closure of regular languages under intersection.
\end{proof}


\noindent{\textbf{{Acknowledgments}}}.
The authors thank Murray Elder, Jim Howie and Ashot Minasyan for helpful discussions. They also thank Denis Osin 
for suggesting a simplification of the proof of Theorem \ref{thm:Phi}, Derek Holt for comments regarding the conjugacy properties of virtually abelian groups, and  Olga Kharlampovich for her question that motivated the 
results on biautomaticity. They finally thank the anonymous referee for all her/his comments, which helped improve the exposition of the paper.

The first author was supported by the Swiss National Science 
Foundation grant FN 200020-137696/1 and by the MCI 
(Spain) through project MTM2011-25955. 

The second author was supported by the Swiss National Science 
Foundation grants Ambizione PZ00P-136897/1 and Professorship FN PP00P2-144681/1.  


\end{document}